\documentclass{amsart}

\usepackage{hyperref}
\usepackage{hhline} 
\usepackage{array}              
\usepackage{amssymb}              
\usepackage{lmodern,bm}              
\usepackage{longtable}
\usepackage{tikz}
\usepackage[all]{xy}
\usepackage{dsfont}
\usepackage{pdflscape}
\usepackage{bigdelim}
\usepackage{multirow}
\usepackage{wasysym}
\usepackage{enumitem}

\CompileMatrices

\newtheorem{theorem}{Theorem}[section]

\newtheorem{lemma}[theorem]{Lemma}
\newtheorem{proposition}[theorem]{Proposition}
\newtheorem{corollary}[theorem]{Corollary}

\theoremstyle{definition}

\newtheorem{definition}[theorem]{Definition}
\newtheorem{construction}[theorem]{Construction}

\newtheorem{remark}[theorem]{Remark}

\numberwithin{equation}{theorem}

\def\vector2#1#2{\left(\begin{array}{c} #1 \\ #2 \end{array}\right)}

\def\Cl{{\rm Cl}}

\def\Bb{\mathsf{B}}
\def\Qq{\mathsf{Q}}

\def\CC{{\mathbb C}}

\def\ZZ{{\mathbb Z}}

\def\QQ{{\mathbb Q}}

\def\OOO{{\mathcal O}}
\def\RRR{{\mathcal R}}

\def\HHH{{\mathcal H}}

\def\lll{{\mathfrak l}}

\def\conv{{\rm conv}}

\def\quot{/\!\!/}

\def\cDV{\mathtt{cDV}}

\def\bangle#1{{\langle #1 \rangle}}

\def\GL{{\rm GL}}

\def\GL{{\rm GL}}

\def\cone{{\rm cone}}

\def\trop{{\rm trop}}
\def\lin{{\rm lin}}

\title[Canonical threefold singularities and $k$-empty polytopes]{Canonical threefold singularities with a torus action of complexity one and $k$-empty polytopes}
\author[L.~Braun, D. H\"attig]{Lukas Braun and Daniel H\"attig}

\address{Mathematisches Institut, Universit\"at T\"ubingen,
Auf der Morgenstelle 10, 72076 T\"ubingen, Germany}
\email{braun@math.uni-tuebingen.de}

\address{Mathematisches Institut, Universit\"at T\"ubingen,
Auf der Morgenstelle 10, 72076 T\"ubingen, Germany}
\email{haettig@math.uni-tuebingen.de}

\subjclass[2010]{14B05, 14R05, 52B20, 11B57}
\keywords{Canonical singularities, Varieties with torus action, $k$-empty polytopes, Farey sequences}

\begin{document}

\begin{abstract}
We classify the canonical threefold singularities that allow an effective two-torus action. This extends classification results of Mori on terminal threefold singularities and of Ishida and Iwashita on toric canonical threefold singularities. Our classification relies on lattice point emptiness of certain polytopal complexes with rational vertices. Scaling the polytopes by the least common multiple $k$ of the respective denominators, we investigate $k$-emptiness of polytopes with integer vertices. We show that two dimensional $k$-empty polytopes either are sporadic or come in series given by Farey sequences. 
We finally present the Cox ring iteration tree of the classified singularities, where all roots, i.e. all spectra of factorial Cox rings, are generalized compound du Val singularities.
\end{abstract}

\maketitle

\tableofcontents

\section{Introduction}

The aim of the present paper is to contribute to the classification of (subclasses of) log-terminal singularities by classifying the three-dimensional canonical singularities that allow a two-torus action. 
Log-terminal singularities and their subclasses are of special importance for the minimal model program, see~\cite{kollarmori}. The focusing on subclasses of log-terminal singularities such as canonical and compound du Val singularities goes back to Miles Reid, see~\cite{reidcan3fold, reidminmod3fold, reidYPG}. 

By the \emph{complexity} of a variety we mean the difference between its dimension and the dimension of the biggest possible effectively acting torus. By a \emph{singularity}, here and throughout the paper we mean a pair $(X,x)$, where  $X$ is a normal affine variety together with a distinguished singular point $x \in X$, such that for every open neighbourhood $X'$ of $x$, we have $X' \cong X$. When it is clear what the point $x$ is, we only write $X$ for the singularity.

We begin with a summary of the status quo in dimensions two and three. The following first table shows classification results for log-terminal surface singularities: 

\begin{longtable}{l||c|c|}
& toric
& complexity one
\\
\hhline{==|=|}
terminal & \multicolumn{2}{c|}{smooth}
\\
\hline
canonical
&
$A_n$
&
$D_n$, $E_6$, $E_7$, $E_8$
\\
\hline
log-terminal
&
cyclic quotient,~\cite[\S 10.1]{CLS} 
&
finite quotient,~\cite{ltpticr, Bries}
\end{longtable}

All log-terminal surface singularities come at least with a one-torus action, which can be seen as a byproduct of the fact that they all are finite quotients $\CC^2/G$ with $G$ a finite subgroup of ${\rm GL}_2(\CC)$. The canonical ones are the well-known $ADE$-singularities - all Gorenstein -, serving as \emph{index one cover} for the others.
The next table shows the situation in dimension three:

\begin{longtable}{l||c|c|c|}
& toric
& complexity one
& higher complexity
\\
\hhline{==|=|=|}
terminal & \multicolumn{3}{c|}{\cite{mori}}
\\
\hline
compound du Val
&
\cite{Da}
&
\cite{ltpticr}
& 
\cite{mark}
\\
\hline
canonical
&
\cite{II}
&
{this paper}
&
? 
\\
\hline
log-terminal
&
{all}
&
{platonic tuples}
& ?
\end{longtable}

The terminal case for all complexities was examined by Mori in~\cite{mori}. The compound du Val singularities, introduced in~\cite{reidcan3fold} and lying between the terminal and the canonical ones, are characterized by being canonical and having a hyperplane section with again at most canonical singularities. They appear in the form of a very short list in~\cite{Da} for the toric case, while those of complexity one have been classified in~\cite{ltpticr} using a Cox ring based approach. Markushevich in~\cite{mark} did not find a complete list in the general case, but gave efficient criteria to decide compound du Valness. The compound du Val singularities will in fact appear in the \emph{Cox ring iteration}, see~\cite{ltpticr, ocri}, of the canonical ones from the present paper.
As the toric case is already settled, we proceed here with those canonical singularities that only admit a two-torus action. In this case, the aforementioned Cox ring based approach is ideally suited. In the following discussion of this approach, the meaning of \emph{platonic tuples} describing the log-terminal threefold singularities with two-torus action will become clear. 

For a normal variety $X$ that has finitely generated class group $\Cl(X)$ and only constant globally invertible functions, the \emph{Cox ring} of $X$ is
$$
\RRR(X):=\bigoplus\limits_{\Cl(X)} \Gamma(X,\OOO_X(D)).
$$
For those $X$ admitting an effective torus action of complexity one, there is an explicit combinatorial description of the Cox ring, as was elaborated in~\cite{hausenherpp, hausensüß, hausenwrob1}. In principle this description generalizes the one of toric varieties by cones and fans. As we are concerned with \emph{singularities} admitting a torus action of complexity one, we deal with the affine varieties of complexity one of \emph{Type 2} from~\cite{hausenwrob1}, as \emph{Type 1} only produces toric singularities, see also~\cite{ltpticr}. We give the most basic construction for threefolds of Type 2 in the following, for details and arbitrary dimension we refer to~\cite{hausenherpp, hausensüß, hausenwrob1}:

\begin{construction}
\label{constr:RAP0}
Fix integers $r,n > 0$, $m \ge 0$ and a partition 
$n = n_0+\ldots+n_r$.
For each $0 \le i \le r$, fix a tuple
$l_{i} \in \ZZ_{> 0}^{n_{i}}$ and define a monomial
$$
T_{i}^{l_{i}}
\ := \
T_{i1}^{l_{i1}} \cdots T_{in_{i}}^{l_{in_{i}}}
\ \in \
\CC[T_{ij},S_{k}; \ 0 \le i \le r, \ 1 \le j \le n_{i}, \ 1 \le k \le m].
$$
We write $\CC[T_{ij},S_{k}]$ for the 
polynomial ring from above in the following. 
For every $i \in \left\{0, \ldots, r-2\right\}$
define
$$
g_{i}
\ :=  \
(i+1)T_{i}^{l_i}+T_{i+1}^{l_{i+1}}+T_{i+2}^{l_{i+2}}
\ \in \
\CC[T_{ij},S_{k}].
$$\\
We  choose integral
$2 \times n_i$ matrices $d_i$ and column vectors $d_k' \in \ZZ^2$. 
From all these data, we build up an $(r+2) \times (n + m)$ matrix:
$$
P
\ := \
\left[
\begin{array}{ccccccc}
-l_{0} & l_{1} &  & 0 & 0  &  \ldots & 0
\\
\vdots & \vdots & \ddots & \vdots & \vdots &  & \vdots
\\
-l_{0} & 0 &  & l_{r} & 0  &  \ldots & 0 \\
d_0 & d_1 & \cdots & d_r & d'_{1} & \ldots & d'_m
\end{array}
\right].
$$
We require the columns of $P$ to be the pairwise 
different ray generators of a full-dimensional convex polyhedral cone in 
$\QQ^{r+2}$.  
Let $P^*$ denote the transpose of $P$ and 
consider the projection
$$
Q \colon 
\ZZ^{n+m} 
\ \to \ 
K 
\ := \ 
\ZZ^{n+m} / \mathrm{im}(P^*).
$$
Denoting by $e_{ij}, e_k \in \ZZ^{n+m}$ the 
canonical basis vectors corresponding to
the variables $T_{ij}$ and $S_k$, we obtain a 
$K$-grading on $\CC[T_{ij}, S_k]$ by setting
$$
\deg(T_{ij}) \ := \ Q(e_{ij} ) \ \in \ K,
\qquad\qquad
\deg(S_k) \ := \ Q(e_k) \ \in \ K.
$$
In particular, we have the $K$-graded factor algebra
$$
R(P)
\ := \
\CC[T_{ij},S_{k}] / \bangle{g_{i}; \ 0 \leq i \leq r },
$$
and, moreover, the quasitorus $H:={\rm Spec} \, \CC[K]$ as well as $\overline{X}(P):={\rm Spec} \, R(P)$ and the good quotient
$$
X(P)
\ := \
\overline{X}(P) \quot H
$$
having class group $\Cl(X(P))=K$.
\end{construction}

Now every normal rational threefold singularity $(X,x)$ admitting a two-torus action can be presented as $X \cong X(P)$ for some suitable defining matrix $P$, where $x$ is always the image of the origin $O \in \overline{X}(P) \subseteq {\rm Spec} \, \CC[T_{ij},S_{k}]$.

The log-terminal ones among the singularities of complexity one are characterized by \emph{platonic tuples}. We call a tuple $(a_0,\ldots,a_r)$ of positive integers \emph{platonic}, if either $r=1$ or $r\geq 2$ and after ordering the tuple decreasingly, we have $a_i=1$ for $i\geq 3$ and the triple of the first three is one of
$$
(5,3,2),
\quad
(4,3,2),
\quad
(3,3,2),
\quad
(a,2,2),
\quad
(a,b,1).
$$
In~\cite{ltpticr} it was shown that a singularity $X=X(P)$ admitting a torus action of complexity one is log-terminal if and only if for the defining tuples $l_i$ the maximal entries $\mathfrak{l}_i:=\max_{j_i}(l_{ij_i})$ form a platonic tuple $(\mathfrak{l}_0,\ldots,\mathfrak{l}_r)$.

In Construction~\ref{constr:RAP0}, it was stated that the columns of $P$ need to be the primitive ray generators of a convex polyhedral cone $\sigma_P$. This cone defines an affine toric variety $Z(P)$, the \emph{minimal ambient toric variety} of $X(P)$. In~\cite{michelepaper}, a certain polyhedral complex, the \emph{anticanonical complex} $A_X^c \subseteq \QQ^{r+2}$, lying on the intersection of $\sigma_P$ with the \emph{tropical variety} of $X(P)$, was defined. This polyhedral complex is a generalization of the polyhedron spanned by the origin and the primitive ray generators of the defining cone of an affine toric variety. As in the toric case, exceptional prime divisors correspond to integer points in and 'over' this polyhedral complex, so canonicity (and terminality etc.) depends on integer points inside it. 

This is where the $k$-empty polytopes come into play. 
Throughout the paper, we denote polytopes by sans-serif letters.
Some of the subpolytopes $\mathsf{P}$ of the anticanonical complex $A_X^c$ may have rational vertices. If the least common denominator of all the rational entries is $k$, then lattice points inside $\mathsf{P}$ correspond to $k$-fold points in $k\ZZ^{r+2}\cap k\mathsf{P}$. Now $k\mathsf{P}$ is a lattice polytope. We investigate $k$-emptiness of lattice polytopes in Section~\ref{sec:pol} and classify two-dimensional ones using \emph{Farey sequences}. Specifically, the $k$-th Farey sequence consists of all completely reduced fractions $0\le\frac{f_{1}}{f_{2}}<1$ where $f_{2}\le k$. The members of the $k$-th Farey sequence are called $k$-th Farey numbers in the following. The \emph{$k$-th Farey strip} corresponding to a $k$-th Farey number $f=\frac{f_{1}}{f_{2}}$ is the polyhedron
\[
F_{k,f} \ := \ 
\begin{cases}
  \left\{ \begin{pmatrix}x\\y\end{pmatrix}; \ 0 \ < \ \begin{pmatrix}
x\\y\end{pmatrix} \cdot \begin{pmatrix}-f_{1}\\f_{2}\end{pmatrix}
\ < \ k \right\}, \ &\mathrm{if} \ f_{2} = k,\\
  \left\{ \begin{pmatrix}x\\y\end{pmatrix}; \ 0 \ < \ \begin{pmatrix}
x\\y\end{pmatrix} \cdot \begin{pmatrix}-f_{1}\\f_{2}\end{pmatrix}
\ \le \ k \right\}, \ &\mathrm{if} \ f_{2} \ne k.
\end{cases}
\]
We obtain the following structure theorem for $k$-empty lattice triangles:

\begin{theorem}
\label{th:farey}
Let $\mathsf{S}$ be a $k$-empty lattice triangle. Then, up to $k$-affine unimodular transformation, $\mathsf{S}$ is either contained in a $k$-th Farey strip or it is one of finitely many sporadic exceptions.
\end{theorem}

We show in Section~\ref{sec:pol} that the sporadic triangles are contained in $\emph{spikes}$ attached to Farey strips and can be listed explicitly. Subsequently, the techniques developed are applied when we consider \emph{canonical polytopes}, i.e. polytopes defining canonical singularities.

There are two invariants of log-terminal singularities that turn out to be essential for our classification: the well-known \emph{Gorenstein index} $\imath$ and the \emph{canonical multiplicity} $\zeta$, defined in~\cite{ltpticr}. There is a finite number of combinations of these invariants, as Corollary 4.6 and Proposition 4.7 of~\cite{ltpticr} show, and the classification is split up in this finite number of cases.
Since for Gorenstein singularities, log-terminality equals canonicity, the Gorenstein canonical singularities are characterized by the platonic tuples together with $\imath=1$. As in the toric case, where a classification of the Gorenstein canonical threefold singularities would be equivalent to a classification of all lattice polygons up to lattice equivalence, this is no appropriate class for classification. Thus we go the same way as~\cite{II} in the toric case: we let $\imath\geq 2$.
The results are split up in two groups: those singularities that can be given by their Cox ring in general and those that come in large series admitting only a description by the defining matrix $P$. However, determination of the Cox ring of a single member of these series is always possible.

\begin{theorem}
\label{th:class}
Let $X$ be an affine canonical threefold singularity of Gorenstein index $\imath \geq 2$ admitting a two-torus action.
The following table lists those $X$ 
that are either sporadic or belong to a series of singularities with 'few' - up to three - parameters. The singularities $X$ are encoded by their Cox Ring and class group:

\renewcommand{\arraystretch}{1.3} 
\setlength{\tabcolsep}{0.5pt}
\setlength{\arraycolsep}{1pt}

\begin{longtable}{c|c|c|c|c|c}
No. & $\RRR(X)$ & $\Cl(X)$ & $Q$ & $\imath$ & $\zeta$ \\
\hline
1 
&
$\CC[T_1,\ldots,T_4]$
&
$
{\tiny
\begin{array}{c}
\ZZ \times \ZZ/ 2\mathfrak{d}\ZZ, \\
m,n \in \ZZ_{\geq 1} \\
\mathfrak{d}\!=\!\gcd(2m,m\!+\!n)
 \end{array}
 }
$
&
$
{\tiny
\begin{array}{c}
\begin{bmatrix}

-\!\frac{m+n}{\mathfrak{d}} & \frac{2n}{\mathfrak{d}} & -\!\frac{m\!+\!n}{\mathfrak{d}} & \frac{2m}{\mathfrak{d}} \\
\overline{\alpha_1 \!+\!\mathfrak{d}} & \overline{-(2\alpha_1\!+\!\alpha_2)} & \overline{\alpha_1} & \overline{\alpha_2} 
\end{bmatrix} \\
\mathrm{~with~}
2m\alpha_1+(m+n)\alpha_2=\mathfrak{d}
\end{array}
}
$ 
&
$2$
&
$1$
\\
\hline
2
&
$\CC[T_1,\ldots,T_3]$
&
$
{\tiny
\begin{array}{c}
\ZZ/\imath m \ZZ, \\
m,n \in \ZZ_{\geq 1}, \\
\gcd(n,\imath)=1
\end{array}
}
$
&
$
{\tiny
\begin{array}{c}
\begin{bmatrix}
\overline{1} & \overline{m\alpha_1} & \overline{-1}
\end{bmatrix} \\
 \mathrm{~with~}
n\alpha_1 \equiv 1 \mod \imath
\end{array}
}
$ 
&
$\geq\!2$
&
$1$
\\
\hline
3
&
$\CC[T_1,\ldots,T_3]$
&
$
{\tiny
\begin{array}{c}
\ZZ/4m\ZZ, \\
m \in \ZZ_{\geq 2}
\end{array}
}
$ 
&
${\tiny
\begin{bmatrix}\overline{2} & \overline{2m-1} & \overline{-1}\end{bmatrix}
}$
&
$2$
&
$1$
\\
\hline
4
&
$\CC[T_1,\ldots,T_3]$
&
$
\ZZ/10\ZZ
$
&
${\tiny
\begin{bmatrix}\overline{1} & \overline{1} & \overline{3}\end{bmatrix}
}$
&
$2$
&
$1$
\\
\hline
5 
&
$\CC[T_1,\ldots,T_3]$
&
$
\ZZ/9\ZZ
$
&
${\tiny
\begin{bmatrix}\overline{1} & \overline{4} & \overline{7}\end{bmatrix}
}$
&
$3$
&
$1$
\\
\hline
6
&
$\frac{\CC[T_1,\ldots,T_4]}{\left\langle T_1^3T_2+T_3^3+T_4^2\right\rangle}$
&
$\ZZ/2\ZZ$
&
${\tiny
\begin{bmatrix}
\overline{1} & \overline{1} & \overline{0} & \overline{1}
\end{bmatrix}}
$
&
$2$
&
$1$
 \\
\hline
7
&
$\frac{\CC[T_1,\ldots,T_5]}{\left\langle T_1^3T_2+T_3^3T_4+T_5^2\right\rangle}$
&
$\ZZ \times \ZZ/2\ZZ$
&
${\tiny
\begin{bmatrix}
3 & -1 & 3 & -1 & 4 \\
\overline{0} & \overline{0} & \overline{1} & \overline{1} & \overline{1}
\end{bmatrix}}
$
&
$2$
&
$1$
\\
\hline
8
&
$\frac{\CC[T_1,\ldots,T_4]}{\left\langle T_1^3+T_2^2T_3+T_4^2\right\rangle}$
&
$ \ZZ/3\ZZ$
&
${\tiny
\begin{bmatrix}
\overline{2} & \overline{1} & \overline{1} & \overline{0}
\end{bmatrix}}
$
&
$3$
&
$1$
\\
\hline
9
&
$\frac{\CC[T_1,\ldots,T_4]}{\left\langle T_1^3+T_2^2+T_3T_4\right\rangle}$
&
$ \ZZ/10\ZZ$
&
${\tiny
\begin{bmatrix}
\overline{2} & \overline{3} & \overline{9} & \overline{7}
\end{bmatrix}}
$
&
$2$
&
$1$
\\
\hline
10
&
$\frac{\CC[T_1,\ldots,T_4]}{\left\langle T_1^4+T_2^3+T_3^2\right\rangle}$
&
$ \ZZ/2\ZZ$
&
${\tiny
\begin{bmatrix}
\overline{1} & \overline{0} & \overline{1} & \overline{1}
\end{bmatrix}}
$
&
$2$
&
$2$
\\
\hline
11
&
$\frac{\CC[T_1,\ldots,T_4]}{\left\langle T_1^4+T_2^3T_3+T_4^2\right\rangle}$
&
$ \ZZ/2\ZZ$
&
${\tiny
\begin{bmatrix}
\overline{0} & \overline{1} & \overline{1} & \overline{1}
\end{bmatrix}}
$
&
$2$
&
$2$
\\
\hline
12
&
$\frac{\CC[T_1,\ldots,T_5]}{\left\langle T_1^4+T_2^3T_3+T_4^2\right\rangle}$
&
$ \ZZ  \times \ZZ/2\ZZ$
&
${\tiny
\begin{bmatrix}
5 & 7 & -1 & 10 & -1 \\
\overline{1} & \overline{0} & \overline{0} & \overline{1} & \overline{1}
\end{bmatrix}}
$
&
$2$
&
$2$
\\
\hline
13
&
$\frac{\CC[T_1,\ldots,T_5]}{\left\langle T_1^4+T_2^3+T_3^2\right\rangle}$
&
$ \ZZ  \times \ZZ/2\ZZ$
&
${\tiny
\begin{bmatrix}
-6 & -8 & -12 & 1 & 1  \\
\overline{1} & \overline{0} & \overline{1} & \overline{1} & \overline{0}
\end{bmatrix}}
$
&
$2$
&
$2$
\\
\hline
14
&
$\frac{\CC[T_1,\ldots,T_5]}{\left\langle T_1^4+T_2^3T_3T_4+T_5^2\right\rangle}$
&
$ \ZZ  \times \ZZ/2\ZZ$
&
${\tiny
\begin{bmatrix}
4 & 6 & -1 & -1 & 8 \\
\overline{0} & \overline{1} & \overline{0} & \overline{1} & \overline{1}
\end{bmatrix}}
$
&
$2$
&
$2$
\\
\hline
15
&
$\frac{\CC[T_1,\ldots,T_4]}{\left\langle T_1^3+T_2^3+T_3^2\right\rangle}$
&
$  \ZZ/3\ZZ$
&
${\tiny
\begin{bmatrix}
\overline{1} & \overline{2} & \overline{0} & \overline{2}
\end{bmatrix}}
$
&
$3$
&
$3$
\\
\hline
16
&
$\frac{\CC[T_1,\ldots,T_4]}{\left\langle T_1^3+T_2^3+T_3^2T_4\right\rangle}$
&
$  \ZZ/3\ZZ$
&
${\tiny
\begin{bmatrix}
\overline{0} & \overline{1} & \overline{2} & \overline{2}
\end{bmatrix}}
$
&
$3$
&
$3$
\\
\hline
17
&
$\frac{\CC[T_1,\ldots,T_4]}{\left\langle T_1^4+T_2^2+T_3^2T_4\right\rangle}$
&
$  \ZZ/4\ZZ$
&
${\tiny
\begin{bmatrix}
\overline{1} & \overline{0} & \overline{3} & \overline{2}
\end{bmatrix}}
$
&
$2$
&
$2$
\\
\hline
18
&
$\frac{\CC[T_1,\ldots,T_5]}{\left\langle T_1^4+T_2^2+T_3^2T_4T_5\right\rangle}$
&
$\ZZ \times  \ZZ/4\ZZ$
&
${\tiny
\begin{bmatrix}
1 & 2 & 3 & -1 & -1 \\
\overline{0} & \overline{2} & \overline{0} & \overline{3} & \overline{1}
\end{bmatrix}}
$
&
$2$
&
$2$
\\
\hline
19
&
$\frac{\CC[T_1,\ldots,T_5]}{\left\langle T_1^4+T_2^2+T_3^2T_4\right\rangle}$
&
$\ZZ \times  \ZZ/2\ZZ$
&
${\tiny
\begin{bmatrix}
3 & 6 & 7 & -2 & -2 \\
\overline{1} & \overline{1} & \overline{0} & \overline{0} & \overline{1}
\end{bmatrix}}
$
&
$2$
&
$2$
\\
\hline
20
&
$\frac{\CC[T_1,\ldots,T_4]}{\left\langle T_1^4+T_2^4+T_3T_4\right\rangle}$
&
$\ZZ/8\ZZ$
&
${\tiny
\begin{bmatrix}
\overline{5} & \overline{7} & \overline{1} & \overline{3}
\end{bmatrix}}
$
&
$2$
&
$2$
\\
\hline
21
&
$\frac{\CC[T_1,\ldots,T_4]}{\left\langle T_1^4+T_2^2+T_3T_4\right\rangle}$
&
$\ZZ/6\ZZ$
&
${\tiny
\begin{bmatrix}
\overline{2} & \overline{1} & \overline{3} & \overline{5}
\end{bmatrix}}
$
&
$2$
&
$2$
\\
\hline
22
&
$\frac{\CC[T_1,\ldots,T_4]}{\left\langle T_1^8+T_2^2+T_3T_4\right\rangle}$
&
$\ZZ/10\ZZ$
&
${\tiny
\begin{bmatrix}
\overline{6} & \overline{9} & \overline{7} & \overline{1}
\end{bmatrix}}
$
&
$2$
&
$2$
\\
\hline
23
&
$\frac{\CC[T_1,\ldots,T_4]}{\left\langle T_1^2+T_2^2+T_3T_4\right\rangle}$
&
$\ZZ/8\ZZ$
&
${\tiny
\begin{bmatrix}
\overline{5} & \overline{1} & \overline{7} & \overline{3}
\end{bmatrix}}
$
&
$4$
&
$2$
\\
\hline
24
&
$\frac{\CC[T_1,\ldots,T_4]}{\left\langle T_1^3+T_2^3+T_3T_4\right\rangle}$
&
$\ZZ/6\ZZ$
&
${\tiny
\begin{bmatrix}
\overline{3} & \overline{1} & \overline{4} & \overline{5}
\end{bmatrix}}
$
&
$3$
&
$3$
\\
\hline
25
&
$\frac{\CC[T_1,\ldots,T_4]}{\left\langle T_1^5+T_2^4+T_3T_4\right\rangle}$
&
$\ZZ/6\ZZ$
&
${\tiny
\begin{bmatrix}
\overline{4} & \overline{5} & \overline{1} & \overline{1}
\end{bmatrix}}
$
&
$2$
&
$3$
\\
\hline
26
&
$\frac{\CC[T_1,\ldots,T_5]}{\left\langle T_1^{k-1}+T_2T_3+T_4T_5\right\rangle}$
&
$\ZZ\!\times\! \ZZ/2k\ZZ$
&
${\tiny
\begin{bmatrix}
0 & 1 & -1 & -1 & 1 \\
\overline{1} & \overline{k} & \overline{-1} & \overline{k-1} & \overline{0}
\end{bmatrix}}
$
&
$2$
&
$1$
\\
\hline
27
&
$\frac{\CC[T_1,\ldots,T_4]}{\left\langle T_1^{2k}+T_2^2+T_3^2\right\rangle}$
&
$(\ZZ/2\ZZ)^2$
&
${\tiny
\begin{bmatrix}
\overline{1} & \overline{k+1} & \overline{k} & \overline{1}  \\
\overline{1} & \overline{k} & \overline{k+1} & \overline{0} 
\end{bmatrix}}
$
&
$2$
&
$2$
\\
\hline
28
&
$\frac{\CC[T_1,\ldots,T_5]}{\left\langle T_1^{2k}+T_2^2+T_3^2\right\rangle}$
&
$\ZZ \times (\ZZ/2\ZZ)^2$
&
$
{\tiny
\begin{bmatrix}
2 & 2k & 2k & -1 & -1 \\
\overline{1} & \overline{k+1} & \overline{k} & \overline{0} & \overline{1}  \\
\overline{1} & \overline{k} & \overline{k+1} & \overline{0} & \overline{0}
\end{bmatrix}}
$
&
$2$
&
$2$
\\
\hline
29
&
$\frac{\CC[T_1,\ldots,T_4]}{\left\langle T_1^{2k+1}+T_2^2+T_3^2\right\rangle}$
&
$ \ZZ/4\ZZ$
&
$
{\tiny
\begin{bmatrix}
\overline{2} & \overline{1-2k} & \overline{-1-2k} & \overline{1}
\end{bmatrix}}
$
&
$4$
&
$2$
\\
\hline
30
&
$\frac{\CC[T_1,\ldots,T_4]}{\left\langle T_1^{k\zeta-2}+T_2^2+T_3T_4\right\rangle}$
&
$ \ZZ/2k\ZZ$
&
$
{\tiny
\begin{bmatrix}
\overline{1} & \overline{k(1-\zeta/2)-1} & \overline{k\zeta-1} & \overline{1}
\end{bmatrix}}
$
&
$2$
&
$\in 4\ZZ$
\\
\hline
31
&
$\frac{\CC[T_1,\ldots,T_4]}{\left\langle T_1^{5\zeta-2}+T_2^2+T_3T_4\right\rangle}$
&
$ \ZZ/10\ZZ$
&
${\tiny
\begin{bmatrix}
\overline{4} & \overline{1} & \overline{3} & \overline{9}
\end{bmatrix}}
$
&
$2$
&
$\geq 3$
\\
\hline
32
&
$\frac{\CC[T_1,\ldots,T_4]}{\left\langle T_1^{3\zeta-4}+T_2^4+T_3T_4\right\rangle}$
&
$ \ZZ/6\ZZ$
&
${\tiny
\begin{bmatrix}
\overline{2} & \overline{1} & \overline{5} & \overline{5}
\end{bmatrix}}
$
&
$2$
&
$\in 2\ZZ\!+\!1$
\\
\hline
33
&
$\frac{\CC[T_1,\ldots,T_4]}{\left\langle T_1^{3\zeta-2}+T_2^2+T_3T_4\right\rangle}$
&
$ \ZZ/6\ZZ$
&
${\tiny
\begin{bmatrix}
\overline{2} & \overline{1} & \overline{3} & \overline{5}
\end{bmatrix}}
$
&
$2$
&
$\geq 3$
\\
\hline
34
&
$
{\tiny
\frac{\CC[T_1,\ldots,T_6]}
{\left\langle 
\begin{array}{c}
T_1^{3\zeta-2}+T_2^2+T_3T_4 \\
2T_2^2+T_3T_4+T_5T_6
\end{array}
\right\rangle}
}
$
&
$ \ZZ \times \ZZ/6\ZZ$
&
${\tiny
\begin{bmatrix}
0&0&1&-1&-1&1\\
\overline{2} & \overline{1} & \overline{3} & \overline{5} & \overline{2} & \overline{0}
\end{bmatrix}}
$
&
$2$
&
$\geq 3$
\\
\hline
35
&
$\frac{\CC[T_1,\ldots,T_4]}{\left\langle T_1^{3\zeta-2}+T_2^2+T_3T_4\right\rangle}$
&
$ \ZZ/12\ZZ$
&
${\tiny
\begin{bmatrix}
\overline{4} & \overline{2} & \overline{5} & \overline{11}
\end{bmatrix}}
$
&
$2$
&
$\geq 3$
\\
\hline
36
&
$\frac{\CC[T_1,\ldots,T_5]}{\left\langle (T_1T_2)^{3\zeta-2}+T_3^2+T_4T_5\right\rangle}$
&
$\ZZ \times  \ZZ/6\ZZ$
&
${\tiny
\begin{bmatrix}
-1 & 1 & 0 & 1 & -1 \\
\overline{2} & \overline{0} & \overline{1} & \overline{3} & \overline{5}
\end{bmatrix}}
$
&
$2$
&
$\geq 3$
\\
\hline
37
&
$\frac{\CC[T_1,\ldots,T_5]}{\left\langle T_1^{3\zeta-2}T_2^{\zeta-2}+T_3^2+T_4T_5\right\rangle}$
&
$\ZZ \times  \ZZ/2\ZZ$
&
${\tiny
\begin{bmatrix}
-1 & 3 & -2 & 3 & -7 \\
\overline{1} & \overline{1} & \overline{1} & \overline{0} & \overline{0}
\end{bmatrix}}
$
&
$2$
&
$\geq 3$
\\
\hline
38
&
$\frac{\CC[T_1,\ldots,T_4]}{\left\langle T_1^{2\zeta-2}+T_2^2+T_3T_4\right\rangle}$
&
$\ZZ/6\ZZ$
&
${\tiny
\begin{bmatrix}
\overline{5} & \overline{5} & \overline{3} & \overline{1}
\end{bmatrix}}
$
&
$3$
&
$\in 6\ZZ\!-\!1$
\\
\hline
39
&
${\tiny
\frac{\CC[T_1,\ldots,T_6]}
{\left\langle 
\begin{array}{c}
T_1^{2\zeta-2}+T_2^2+T_3T_4 \\
2T_2^2+T_3T_4+T_5T_6
\end{array}
\right\rangle}
}$
&
$\ZZ \times \ZZ/6\ZZ$
&
${\tiny
\begin{bmatrix}
0&0&1&-1&-1&1 \\
\overline{5} & \overline{5} & \overline{3} & \overline{1} & \overline{4} & \overline{0}
\end{bmatrix}}
$
&
$3$
&
$\in 6\ZZ\!-\!1$
\\
\hline
40
&
$\frac{\CC[T_1,\ldots,T_4]}{\left\langle T_1^{2\zeta-2}+T_2^2+T_3T_4\right\rangle}$
&
$\ZZ/12\ZZ$
&
${\tiny
\begin{bmatrix}
\overline{10} & \overline{10} & \overline{7} & \overline{1}
\end{bmatrix}}
$
&
$3$
&
$\in 6\ZZ\!-\!1$
\\
\hline
41
&
$\frac{\CC[T_1,\ldots,T_4]}{\left\langle T_1^{2\zeta-2}+T_2^2+T_3T_4\right\rangle}$
&
$\ZZ/6\ZZ$
&
${\tiny
\begin{bmatrix}
\overline{5} & \overline{5} & \overline{2} & \overline{2}
\end{bmatrix}}
$
&
$3$
&
$\in 6\ZZ\!-\!1$
\\
\hline
42
&
$\frac{\CC[T_1,\ldots,T_5]}{\left\langle (T_1T_2)^{2\zeta-2}+T_3^2+T_4T_5\right\rangle}$
&
$\ZZ \times \ZZ/6\ZZ$
&
${\tiny
\begin{bmatrix}
1 & -1 & 0  & -1 & 1 \\
\overline{0} & \overline{5} & \overline{5} & \overline{2} & \overline{2}
\end{bmatrix}}
$
&
$3$
&
$\in 6\ZZ\!-\!1$
\\
\hline
43
&
$\frac{\CC[T_1,\ldots,T_6]}{\left\langle (T_1T_2)^{2\zeta-2}+(T_3T_4)^2+T_5T_6\right\rangle}$
&
$\ZZ^2 \times \ZZ/6\ZZ$
&
${\tiny
\begin{bmatrix}
1 & -1 & 0 & 0 & -1 & 1 \\
0 & 0 & -1 & 1 & 1 & -1 \\
\overline{0} & \overline{5} & \overline{5} & \overline{0} & \overline{2} & \overline{2}
\end{bmatrix}}
$
&
$3$
&
$\in 6\ZZ\!-\!1$
\\
\hline
44
&
$\frac{\CC[T_1,\ldots,T_5]}{\left\langle (T_1T_2)^{2\zeta-2}+T_3^2+T_4T_5\right\rangle}$
&
$\ZZ \times \ZZ/6\ZZ$
&
${\tiny
\begin{bmatrix}
1 & -1 & 0  & -2 & 2 \\
\overline{0} & \overline{5} & \overline{5} & \overline{1} & \overline{3}
\end{bmatrix}}
$
&
$3$
&
$\in 6\ZZ\!-\!1$
\\
\hline
45
&
$\frac{\CC[T_1,\ldots,T_5]}{\left\langle T_1^{2\zeta-2}+(T_2T_3)^2+T_4T_5\right\rangle}$
&
$\ZZ \times \ZZ/6\ZZ$
&
${\tiny
\begin{bmatrix}
0&-1&1&1&-1\\
\overline{5} & \overline{5} & \overline{0} & \overline{3} & \overline{1}
\end{bmatrix}}
$
&
$3$
&
$\in 6\ZZ\!-\!1$
\\
\hline
46
&
$\frac{\CC[T_1,\ldots,T_5]}{\left\langle T_1^{2\zeta-2}+(T_2T_3)^2+T_4T_5\right\rangle}$
&
$\ZZ \times \ZZ/6\ZZ$
&
${\tiny
\begin{bmatrix}
0&-1&1&2&-2\\
\overline{5} & \overline{5} & \overline{0} & \overline{3} & \overline{1}
\end{bmatrix}}
$
&
$3$
&
$\in 6\ZZ\!-\!1$
\\
\hline
47
&
$\frac{\CC[T_1,\ldots,T_5]}{\left\langle T_1^{2\zeta-2}T_2^{\zeta-2}+T_3^2+T_4T_5\right\rangle}$
&
$\ZZ \times \ZZ/3\ZZ$
&
${\tiny
\begin{bmatrix}
-1&2&-1&1&-3\\
\overline{0} & \overline{2} & \overline{0} & \overline{1} & \overline{2}
\end{bmatrix}}
$
&
$3$
&
$\in 6\ZZ\!-\!1$
\\
\hline
48
&
$\frac{\CC[T_1,\ldots,T_4]}{\left\langle T_1^{2\zeta-3}+T_2^3+T_3T_4\right\rangle}$
&
$\ZZ/6\ZZ$
&
${\tiny
\begin{bmatrix}
\overline{3} & \overline{5} & \overline{2} & \overline{1}
\end{bmatrix}}
$
&
$3$
&
$\geq4$
\\
\hline
49
&
$\frac{\CC[T_1,\ldots,T_5]}{\left\langle T_1^{2\zeta-3}+(T_2T_3)^3+T_4T_5\right\rangle}$
&
$\ZZ \times \ZZ/6\ZZ$
&
${\tiny
\begin{bmatrix}
0&-1&1&1&-1 \\
\overline{3} & \overline{5} & \overline{0} & \overline{2} & \overline{1}
\end{bmatrix}}
$
&
$3$
&
$\geq4$
\\
\hline
50
&
$\frac{\CC[T_1,\ldots,T_4]}{\left\langle T_1^{2\zeta-3}+T_2^3+T_3T_4\right\rangle}$
&
$\ZZ/6\ZZ$
&
${\tiny
\begin{bmatrix}
\overline{1} & \overline{1} & \overline{4} & \overline{5}
\end{bmatrix}}
$
&
$3$
&
$\begin{array}{c}
\in 9\ZZ, \\
 9\ZZ-3
\end{array}$
\\
\hline
51
&
$\frac{\CC[T_1,\ldots,T_5]}{\left\langle T_1^{2\zeta-3}+(T_2T_3)^3+T_4T_5\right\rangle}$
&
$\ZZ \times \ZZ/6\ZZ$
&
${\tiny
\begin{bmatrix}
0&-1&1&1&-1 \\
\overline{1} & \overline{1}  & \overline{0} & \overline{4} & \overline{5}
\end{bmatrix}}
$
&
$3$
&
$\begin{array}{c}
\in 9\ZZ, \\
 9\ZZ-3
\end{array}$
\\
\hline
52
&
$\frac{\CC[T_1,\ldots,T_5]}{\left\langle (T_1T_2)^{2\zeta-3}+T_3^3+T_4T_5\right\rangle}$
&
$\ZZ \times \ZZ/6\ZZ$
&
${\tiny
\begin{bmatrix}
-1&1&0&1&-1 \\
\overline{1}   & \overline{0} & \overline{1} & \overline{4} & \overline{5}
\end{bmatrix}}
$
&
$3$
&
$\begin{array}{c}
\in 9\ZZ, \\
 9\ZZ-3
\end{array}$
\\
\hline
53
&
$\frac{\CC[T_1,\ldots,T_4]}{\left\langle T_1^{2\zeta-4}+T_2^4+T_3T_4\right\rangle}$
&
$\ZZ/6\ZZ$
&
${\tiny
\begin{bmatrix}
\overline{1} & \overline{1} & \overline{5} & \overline{5}
\end{bmatrix}}
$
&
$3$
&
$\begin{array}{c}
\in\!12\ZZ\!+\!1,\\
12\ZZ-5
\end{array}$
\\
\hline
54
&
$\frac{\CC[T_1,\ldots,T_5]}{\left\langle T_1^{k\zeta-1}+T_2T_3+T_4T_5\right\rangle}$
&
$\ZZ \times \ZZ/2k\ZZ$
&
${\tiny
\begin{bmatrix}
0&1&-1&-1&1 \\
\overline{1} & \overline{k} & \overline{-1} & \overline{k-1} & \overline{0}
\end{bmatrix}}
$
&
$2$
&
$\in 2\ZZ+1$
\\
\hline
55
&
$\frac{\CC[T_1,\ldots,T_5]}{\left\langle T_1^{\zeta-2}+T_2^2+T_3T_4\right\rangle}$
&
$\ZZ/4\ZZ$
&
${\tiny
\begin{bmatrix}
\overline{2} & \overline{\zeta-2} & \overline{\zeta} & \overline{\zeta}
\end{bmatrix}}
$
&
$4$
&
$\in 2\ZZ\!+\!1$
\end{longtable}
\end{theorem}

\begin{theorem}
\label{th:class2}
Let $X$ be an affine canonical threefold singularity of Gorenstein index $\imath \geq 2$ admitting a two-torus action.
The following table lists those $X$ 
that  belong to a 'many parameter series'. They are given  by their defining matrix $P$. It holds $\imath=2$ for all but No. 59, where we have $\imath \geq 2$. For the canonical multiplicity, we have $\zeta=4$ for No. 56, $\zeta=2$ for No. 57 and $\zeta>\imath$ for Nos. 58 and 59.
Moreover, in all cases we have $r\geq 2$,  $\mathbf{k} \in \ZZ_{\geq 1}^t$, $\mathbf{d_{0}} \in \ZZ^t$  for some $t \in \ZZ_{\geq 1}$.

\renewcommand{\arraystretch}{1.1} 
 \setlength{\tabcolsep}{3pt}
  \setlength{\arraycolsep}{1pt}
  $$
  \mathrm{56a:}
  {\tiny
\left[
\begin{matrix}
-2\mathbf{k}\!-\!1 & 2   & 0  & 0 & 0 & \dots & 0 &0
\\
-2\mathbf{k}\!-\!1 & 0   & 2  & 0 &0& \dots & 0 & 0
\\
-2\mathbf{k}\!-\!1 & 0  & 0  & 1 &1 & \dots & 0 &0
\\[-5pt]
\vdots & \vdots  & \vdots  & \vdots & \vdots & \ddots & \vdots & \vdots
\\
-2\mathbf{k}\!-\!1 & 0   & 0 & 0 & 0 & \dots &  1 & 1
\\
\mathbf{d_{0}} & 1  & 1  & 0 & d_{3} & \ldots & 0 & d_{r}
\\
\mathbf{k} & 0  & 1  & 0 & 0  & \ldots & 0 & 0
\end{matrix}
\right]
}
\quad
\mathrm{56b:}
{\tiny
\left[
\begin{matrix}
-2\mathbf{k}\!-\!1 & 2   & 0 & 0  & 0 & 0 & \dots & 0 &0
\\
-2\mathbf{k}\!-\!1 & 0   & 2 & 2 & 0 &0& \dots & 0 & 0
\\
-2\mathbf{k}\!-\!1 & 0  & 0  & 0& 1 &1 & \dots & 0 &0
\\[-5pt]
\vdots & \vdots  & \vdots & \vdots & \vdots & \vdots & \ddots & \vdots & \vdots
\\
-2\mathbf{k}\!-\!1 & 0 & 0  & 0 & 0 & 0 & \dots &  1 & 1
\\
\mathbf{d_{0}} & 1  & 1 & d_2  & 0 & d_{3} & \ldots & 0 & d_{r}
\\
\mathbf{k} & 0  & 1  & 1 & 0 & 0  & \ldots & 0 & 0
\end{matrix}
\right]
}
$$
$$
\mathrm{57a:}
{\tiny
\begin{bmatrix}
-\mathbf{k} & 2   & 0  & 0 & 0 & \dots & 0 & 0 & 0& 0
\\
-\mathbf{k}  & 0   & 2  & 0 &0& \dots & 0 & 0 & 0 & 0
\\
-\mathbf{k}  & 0  & 0  & 1 &1 & \dots & 0 & 0 & 0  & 0
\\[-5pt]
\vdots & \vdots  & \vdots  & \vdots & \vdots & \ddots & \vdots & \vdots & \vdots & \vdots
\\
-\mathbf{k}  & 0   & 0 & 0 & 0 & \dots &  1 & 1 & 0 & 0
\\
\mathbf{d_{0}} & 1  & 1  & 0 & d_{3} & \ldots & 0 & d_{r} & d'_{1} & d'_{2}
\\
1\!-\!\mathbf{k} & 0  & 2  & 0 & 0  & \ldots & 0 & 0 & 1 & 1
\end{bmatrix}
}
\quad
\mathrm{57b:}
{\tiny
\begin{bmatrix}
-2 & 2     & 0 & 0 & \dots & 0 & 0 & 0& 0
\\
-2  & 0    & 1 &1& \dots & 0 & 0 & 0 & 0
\\[-5pt]
\vdots   & \vdots  & \vdots & \vdots & \ddots & \vdots & \vdots & \vdots & \vdots
\\
-2  & 0    & 0 & 0 & \dots &  1 & 1 & 0 & 0
\\
1 & 1   & 0 & d_{2} & \ldots & 0 & d_{r} & d'_{1} & d'_{2}
\\
0 & 2    & 0 & 0  & \ldots & 0 & 0 & 1 & 1
\end{bmatrix}
}
$$
$$
\mathrm{58a:}
{\tiny
\begin{bmatrix}
2(\zeta-k) & 2k     & 0 & 0 & \dots & 0 & 0 
\\
2(\zeta-k)  & 0    & 1 &1& \dots & 0 & 0 
\\[-5pt]
\vdots   & \vdots  & \vdots & \vdots & \ddots & \vdots & \vdots 
\\
2(\zeta-k)  & 0    & 0 & 0 & \dots &  1 & 1 
\\
1 & 1   & 0 & d_{2} & \ldots & 0 & d_{r} 
\\
2\left(\frac{k\mu+1}{\zeta}-\mu\right) & 2\frac{k\mu+1}{\zeta}    & 0 & 0  & \ldots & 0 & 0 
\end{bmatrix}
}
\quad
\mathrm{58b:}
{\tiny
\begin{bmatrix}
2(\zeta-k) & 2k-\zeta & 2k     & 0 & 0 & \dots & 0 & 0 
\\
2(\zeta-k) & 2k-\zeta  & 0    & 1 &1& \dots & 0 & 0 
\\[-5pt]
\vdots & \vdots  & \vdots  & \vdots & \vdots & \ddots & \vdots & \vdots 
\\
2(\zeta-k) & 2k-\zeta  & 0    & 0 & 0 & \dots &  1 & 1 
\\
1 & 1 & d_{11}  & 0 & d_{2} & \ldots & 0 & d_{r} 
\\
2\left(\frac{k\mu+1}{\zeta}-\mu\right) & 2\frac{k\mu+1}{\zeta}-\mu & 2\frac{k\mu+1}{\zeta}    & 0 & 0  & \ldots & 0 & 0 
\end{bmatrix}
}
$$
$$
\mathrm{58c:}
{\tiny
\begin{bmatrix}
2(\zeta-k) & 2k-\zeta & 2k-\zeta & 2k     & 0 & 0 & \dots & 0 & 0 
\\
2(\zeta-k) & 2k-\zeta & 2k-\zeta  & 0    & 1 &1& \dots & 0 & 0 
\\[-5pt]
\vdots & \vdots & \vdots  & \vdots  & \vdots & \vdots & \ddots & \vdots & \vdots 
\\
2(\zeta-k) & 2k-\zeta & 2k-\zeta & 0    & 0 & 0 & \dots &  1 & 1 
\\
1  & d_{11} & d_{12} & 1 & 0 & d_{2} & \ldots & 0 & d_{r} 
\\
2\left(\frac{k\mu+1}{\zeta}-\mu\right) & 2\frac{k\mu+1}{\zeta}-\mu  & 2\frac{k\mu+1}{\zeta}-\mu & 2\frac{k\mu+1}{\zeta}    & 0 & 0  & \ldots & 0 & 0 
\end{bmatrix}
}
$$
$$
\mathrm{59a:}
{\tiny
\begin{bmatrix}
-k\imath & \lll & 0 & 0 & \cdots & 0 & 0 \\
-k\imath & 0 & 1 & 1 & \cdots & 0 & 0 \\[-5pt]
\vdots & \vdots & \vdots & \vdots & \ddots & \vdots & \vdots \\
-k\imath & 0 & 0 & 0 & \cdots & 1 & 1 \\
d & d & 0 & d_{2} & \cdots & 0 & d_{r} \\
\imath\frac{1-\mu k}{\lll+k \imath} & \frac{\imath+\mu \lll}{\lll+k \imath} & 0 & 0 & \cdots & 0 & 0
\end{bmatrix}
}
\quad
\mathrm{59b:}
{\tiny
\begin{bmatrix}
-k\imath & \lll & \lll & 0 & 0 & \cdots & 0 & 0 \\
-k\imath & 0 & 0 & 1 & 1 & \cdots & 0 & 0 \\[-5pt]
\vdots & \vdots & \vdots & \vdots & \vdots & \ddots & \vdots & \vdots \\
-k\imath & 0 & 0 & 0 & 0 & \cdots & 1 & 1 \\
d & d & d+d_1 &  0 & d_{2} & \cdots & 0 & d_{r} \\
\imath\frac{1-\mu k}{\lll+k \imath} & \frac{\imath+\mu \lll}{\lll+k \imath} & \frac{\imath+\mu \lll}{\lll+k \imath} & 0 & 0 & \cdots & 0 & 0
\end{bmatrix}
}
$$
$$
\mathrm{59c:}
{\tiny
\begin{bmatrix}
-k\imath & -k\imath & \lll & 0 & 0 & \cdots & 0 & 0 \\
-k\imath & -k\imath & 0 & 1 & 1 & \cdots & 0 & 0 \\[-5pt]
\vdots & \vdots & \vdots & \vdots & \vdots & \ddots & \vdots & \vdots \\
-k\imath & -k\imath & 0 & 0 & 0 & \cdots & 1 & 1 \\
d & d+d_0 & d & 0 & d_{2} & \cdots & 0 & d_{r} \\
\imath\frac{1-\mu k}{\lll+k \imath} & \imath\frac{1-\mu k}{\lll+k \imath} & \frac{\imath+\mu \lll}{\lll+k \imath} & 0 & 0 & \cdots & 0 & 0
\end{bmatrix}
}
\quad
\mathrm{59d:}
{\tiny
\begin{bmatrix}
-k\imath & -k\imath & \lll & \lll & 0 & 0 & \cdots & 0 & 0 \\
-k\imath & -k\imath & 0 & 0 & 1 & 1 & \cdots & 0 & 0 \\[-5pt]
\vdots & \vdots & \vdots & \vdots & \vdots & \vdots & \ddots & \vdots & \vdots \\
-k\imath & -k\imath & 0 & 0 & 0 & 0 & \cdots & 1 & 1 \\
d & d+d_0 & d & d+ d_1 &  0 & d_{2} & \cdots & 0 & d_{r} \\
\imath\frac{1-\mu k}{\lll+k \imath} & \imath\frac{1-\mu k}{\lll+k \imath} & \frac{\imath+\mu \lll}{\lll+k \imath} & \frac{\imath+\mu \lll}{\lll+k \imath} & 0 & 0 & \cdots & 0 & 0
\end{bmatrix}
}
$$
\end{theorem}

\begin{corollary}\label{cor:terminal}
Let $X$ be an affine \emph{terminal} threefold singularity of Gorenstein index $\imath \geq 2$ admitting a two-torus action. Then $X$ either is No. 2 with $m=1$ or one of Nos. 59a-d with $d_i=1$ for all $i$.
\end{corollary}

The total coordinate space $\overline{X}$ of a log-terminal singularity $X$ admitting a torus action of complexity one is again of complexity one and a Gorenstein canonical singularity, from what follows that, as was shown in~\cite{ltpticr}, we have a chain
$$
X_p \rightarrow \ldots \rightarrow X_1:=\overline{X} \rightarrow X,
$$ 
where $X_{k+1}$ is the total coordinate space of $X_k$ and $X_p$ is factorial. This is what we call Cox ring iteration. For certain classes of singularities, a tree of such chains can be drawn. 
In the case of log-terminal surface singularities, this tree is given by the following picture from~\cite{ltpticr}, where a non-$ADE$-singularity is denoted by its $ADE$ index one cover and - in the superscript -  its canonical multiplicity $\zeta$ if not $\zeta=1$ as well as Gorenstein index $\imath$. 

$$ 
\xymatrix@C-5pt{
&
{\CC^2}
\ar[dl]
\ar[dr]^{\mathrm{CR}}
\ar@/^1pc/[drrr]
\ar@/^2pc/[drrrrr]
\ar@/^3pc/[drrrrrrr]
&
&
&
& 
&
&
&
\\
E_8
\ar[d]
&
&
A_n
\ar[rr]^{\mathrm{CR} \quad}
\ar[d]
\ar[dr]^{n \text{ odd}}
&
&
D_{n+3}
\ar[rr]^{\mathrm{CR}}_{n=1}
\ar[d]
\ar[dr]^{n=1}
&
&
E_6
\ar[rr]^{\mathrm{CR}}
\ar[d]
&
&
E_7
\ar[d]  
\\
E_8^{\imath}
&
&
A_{n,k}^{\imath}
&
D_{(n+3)/2}^{2,\imath}
&
D_{n+3}^{\imath}
&
E_6^{3,\imath}
&
E_6^{\imath}
&
&
E_7^{\imath}
}
$$

If we put aside the quotient $\CC^2 \to E_8$ by the binary icosahedral group that does not stem from Cox ring iteration, the tree has the two factorial roots $\CC^2$ and $E_8$. 
The respective tree for compound du Val threefold singularities also was given in~\cite{ltpticr}. It comprises canonical singularities $X_1-X_5$ of dimension $\geq 4$ that are \emph{generalized} compound du Val in the sense that they  have a hyperplane section with at most canonical singularities. The following theorem shows that this tree can be extended to a \emph{joint} tree for canonical  and compound du Val threefold singularities.

\begin{theorem}
\label{th:CRI}
Denote by $1\mathrm{-}59$ the canonical singularities from Theorems~\ref{th:class},~\ref{th:class2} and by $1_{\cDV}\mathrm{-}18_{\cDV}$, $X_1\mathrm{-}X_5$  the (generalized) compound du Val singularities from~\cite[Th. 1.8]{ltpticr}. Let the canonical singularities $Y_1\mathrm{-}Y_{10}$ of complexity one be given by 
$$
Y_1\!=\!V(T_1^3T_2+T_3^3T_4+T_5^2),
\
Y_2\!=\!V(T_1^4+T_2^3T_3+T_4^2),
\
Y_3\!=\!V(T_1^3T_2T_3+T_4^3T_5T_6+T_7^2),
$$
$$
Y_4\!=\!V(T_1^4+T_2^3T_3T_4+T_5^2),
\
Y_5\!=\!V(T_1^2T_2+T_3^2T_4+T_5^2T_6),
\
Y_6\!=\!V(T_1^3+T_2^3+T_3^2T_4),
$$
$$
Y_7\!=\!V(T_2^1T_2T_3+T_4^2T_5T_6+T_7^2+),
\
Y_8\!=\!V(T_1^4+T_2^2+T_3^2T_4T_5),
$$
$$
Y_9\!=\!V(T_0^k+T_1^2+T_2^2, \ldots,
(r-1)T_{(r-2)1}T_{(r-2)2}+T_{(r-1)1}T_{(r-1)2}+T_{r1}T_{r2}),
$$
$$
Y_{10}\!=\!V(T_0^k+T_{1}^2+T_{21}^2T_{22}^2, \ldots,
(r-1)T_{(r-2)1}T_{(r-2)2}+T_{(r-1)1}T_{(r-1)2}+T_{r1}T_{r2}).
$$
We have the following complete tree of Cox ring iterations for all these singularities, where arrows indicate total coordinate spaces and double frames factorial singularities. For dashed double frames and arrows, factoriality depends on parameters.
$$
\xymatrix@C16pt@R2pt{
*+<10pt>[o][F-]+<3pt>[o][F-]{\mathbb{C}^n} \ar[rrrrr] \ar@/_0.5pc/[drr] \ar@/_1pc/[ddddrr]
&&&&&
*+<10pt>[o][F-]{1\!\mathrm{-}\!5,3_{\cDV}}
\\
&
&
*+<10pt>[o][F-]{1_{\cDV}}  \ar[r] \ar@/_1pc/[rrrdd]
&
*+<10pt>[o][F-]{4_{\cDV}}  \ar[r] \ar@/_0.5pc/[rrd]
&
*+<10pt>[o][F-]{6_{\cDV}}  \ar[r]
&
*+<10pt>[o][F-]{{\tiny\begin{array}{c}
10,13 \\
7_{\cDV}
\end{array}}}
\\
&&&&&
*+<10pt>[o][F-]{15}
\\
&&&&&
*+<10pt>[o][F-]{{\tiny\begin{array}{c}
27\,\text{-}\,29,56a_{r\!=\!2} \\
57a_{r\!=\!2},5_{\cDV}
\end{array}}}
\\
*+<7pt>[o][F-]+<3pt>[o][F-]{{\tiny\begin{array}{c}
8_{\cDV},17_{\cDV} \\
18_{\cDV}
\end{array}}}
&&
*+<10pt>[o][F-]{2_{\cDV}} \ar[rrr]
&&&
*+<10pt>[o][F-]{{\tiny\begin{array}{c}
23,12_{\cDV} \\
57b_{r\!=\!2}
\end{array}}}
\\
*+<10pt>[o][F-]+<3pt>[o][F-]{15_{\cDV}} \ar[rrrrr]
&&&&&
*+<10pt>[o][F-]{6}
\\
*+<10pt>[o][F-]+<3pt>[o][F-]{Y_1} \ar[rrrrr] \ar@/_0.5pc/[drr]
&&&&&
*+<10pt>[o][F-]{7}
\\
&&
*+<10pt>[o][F-]{Y_2} \ar[rrr]
&&&
*+<10pt>[o][F-]{11,12}
\\
*+<10pt>[o][F-]+<3pt>[o][F-]{Y_3}  \ar[rr]
&&
*+<10pt>[o][F-]{Y_4} \ar[rrr]
&&&
*+<10pt>[o][F-]{14}
\\
*+<10pt>[o][F-]+<3pt>[o][F-]{Y_5}  \ar[rr]
&&
*+<10pt>[o][F-]{Y_6} \ar[rrr]
&&&
*+<10pt>[o][F-]{16}
\\
*+<10pt>[o][F-]+<3pt>[o][F-]{Y_7}  \ar[rr]
&&
*+<10pt>[o][F-]{Y_8} \ar[rrr]
&&&
*+<10pt>[o][F-]{18}
\\
*+<10pt>[o][F-]+<3pt>[o][F-]{X_1} \ar@{-->}[r]
&
*+<10pt>[o][F-]+<3pt>[o][F--]{10_{\cDV}} \ar[rrrr]
&&&&
*+<10pt>[o][F-]{{\tiny\begin{array}{c}
8,17,19 \\
11_{\cDV},16_{\cDV}
\end{array}}}
\\
*+<10pt>[o][F-]+<3pt>[o][F-]{X_{2\!-\!5}} \ar@{-->}@/_0.5pc/[dr] \ar[rrrrr]  \ar@/_1.5pc/[dddrr] \ar@/_2pc/[ddddrr]
&&&&&
*+<11pt>[o][F-]{\tiny\begin{array}{c}
26,34,36\!\mathrm{-}\!39,42\!\mathrm{-}\!47 \\
49,57b_{r\!\geq\!2},58a_{r\!\geq\!2},58b \\
58c,59a_{r\!\geq\!2},59b,59c \\
59d,13e_{\cDV}
\end{array}}
\\
&
*+<10pt>[o][F-]+<3pt>[o][F--]{9_{\cDV}} \ar[rrrr] \ar@/_0.5pc/[dr]
&&&&
*+<10pt>[o][F-]{\tiny\begin{array}{c}
9,20\!\mathrm{-}\!22,24,25 \\
30\!\mathrm{-}\!33,35,38,40,41 \\
48,50,53,55,58a_{r\!=\!2} \\
59a_{r\!=\!2},13o_{\cDV}
\end{array}}
\\
&&
*+<10pt>[o][F-]{14_{\cDV}} \ar[rrr]
&&&
*+<10pt>[o][F-]{56b_{r\!=\!2}}
\\
&&
*+<10pt>[o][F-]{Y_9} \ar[rrr]
&&&
*+<10pt>[o][F-]{56b_{r\!\geq\!2}}
\\
&&
*+<10pt>[o][F-]{Y_{10}} \ar[rrr]
&&&
*+<10pt>[o][F-]{56/57a_{r\!\geq\!2}}
}
$$
\end{theorem}

\begin{corollary}\label{cor:cdvroot}
All roots of the tree from Theorem~\ref{th:CRI} are generalized compound du Val. In particular, every three-dimensional singularity of complexity one that is either canonical of Gorenstein index $\imath\geq 2$ or compound du Val is a quotient of a factorial generalized compound du Val singularity of complexity one.
\end{corollary}

\section{Threefold singularities with a two-torus action}
This section subsumes Sections 2, 3 and 4 of~\cite{ltpticr}. For further details and proofs of statements given in the following, we thus refer to these.
We have seen in Construction~\ref{constr:RAP0} in the introduction that all normal rational threefold singularities with two-torus action can be constructed from a defining matrix $P$. Moreover, such $X(P)$ is log-terminal if and only if the maximal entries $\mathfrak{l}_i:=\max_{j_i}(l_{ij_i})$ of the defining tuples $l_i$ form a platonic tuple $(\mathfrak{l}_0,\ldots,\mathfrak{l}_r)$.
Let in the following $P$ be a defining matrix meeting the requirements of Construction~\ref{constr:RAP0} and, moreover, let $X=X(P)$ always be $\QQ$-Gorenstein log-terminal. 

\subsection{The matrix $P$}

We write $v_{ij}=P(e_{ij})$ and $v_{k}=P(e_k)$ for  the columns of $P$ and call $(v_{i1},\ldots,v_{in_i})$ the $i$-th column block of $P$. By the data of such a block we mean $l_i$ and $d_i$.
The entries $d_i$ and $d_i'$ of $P$ can be manipulated in the following ways by so called \emph{admissible operations}, keeping the isomorphy type of $X(P)$:
\begin{enumerate}
\item
Swapping two columns inside a column block.
\item
Exchanging the data of two column blocks.
\item
Adding multiples of the upper $r$ rows
to one of the last two rows.
\item
Any elementary row operation among the last two
rows.
\item
Swapping two of the last $m$ columns.
\end{enumerate}

Two important invariants of $X(P)$ are encoded in $P$: the  familiar \emph{Gorenstein index} $\imath \in \ZZ_{\geq 1}$ and the \emph{canonical multiplicity} $\zeta \in \ZZ_{\geq 1}$, defined in~\cite{ltpticr}. We use these two invariants in the following to establish normal forms for matrices $P$ defining log-terminal $X(P)$.

\begin{proposition}
\label{prop:zetaexcep}
Let $X$ be a log-terminal threefold singularity of Gorenstein index $\imath_X$ and canonical multiplicity $\zeta_X$ admitting a two-torus action. Then there is a defining matrix $P$ such that $X\cong X(P)$ and $P$ meets the following requirements:
\begin{enumerate}[leftmargin=*]
\item We have $\mathfrak{l}_0 \geq \ldots \geq \mathfrak{l}_r$ and $l_{i1}\geq \cdots \geq l_{in_i}$ for all $i$. We call $(l_{01},\ldots,l_{r1})$ the leading platonic tuple and $(v_{01},\ldots,v_{r1})$ the leading block of $P$.
\item If $\mathfrak{l}_i=1$ for some $i$, then $n_i\geq 2$. We call $P$ irredundant if this holds.
\item Either $\zeta_X=1$ and for $1\leq i \leq r$, $ 1 \leq j \leq n_i$ and  
$1 \leq \kappa \leq m$:
$$
d_{0j2}=\imath_X(1-(r-1)l_{0j}, 
\quad
d_{ij2}=d'_{\kappa 2}=\imath_X,
$$
or $\zeta_X>1$ and one of the following cases holds:
\renewcommand{\arraystretch}{1.8} 

\begin{longtable}{c|c|c|c|c}
Case 
& 
$(l_{01},l_{11},l_{21})$ 
& 
$(d_{0j2},d_{1j2},d_{2j2})$ 
& 
$\zeta_X$ 
& 
$\imath_X$
\\
\hline
$(i)$ 
&
$(4,3,2)$ 
& 
$\frac{1}{2}
(\imath_X + l_{0j},\, \imath_X(1 - l_{1j}), \, \imath_X - l_{2j})$  
&
$2$ 
&
$ 0 \mod 2$
\\
\hline
$(ii)$ 
&
$(3,3,2)$ 
& 
$
\frac{1}{3}
(\imath_X - l_{0j},\,
\imath_X + l_{1j}, \, 
\imath_X(1 - l_{2j})) 
$
&
$3$ 
&
$0 \mod 3$
\\
\hline
$(iii)$ 
&
$(2k+1,2,2)$ 
& 
$
\frac{1}{4}
(\imath_X(1 - l_{0j}), \, 
\imath_X - l_{1j}, \,
\imath_X + l_{2j}) 
$  
&
$4$ 
&
$2 \mod 4$
\\
\hline
$(iv)$ 
&
$(2k,2,2)$ 
& 
$
\frac{1}{2}
(\imath_X - l_{0j}, \,
\imath_X + l_{1j}, \, 
\imath_X(1 -  l_{2j}))
$  
&
$2$ 
&
$0 \mod 2$
\\
\hline
$(v)$ 
&
$(k,2,2)$ 
& 
$
\frac{1}{2}
(\imath_X(1 - l_{0j}, \,
\imath_X - l_{1j}, \, 
\imath_X + l_{2j})
$  
&
$2$ 
&
$0 \mod 2$
\\
\hline
$(vi)$ 
&
$(k_0,k_1,1)$ 
& 
$
\frac{1}{\zeta_X}\left(\imath_X-\mu l_{0j}, \, \imath_X+\mu l_{1j}, \, 0\right)
$ 
&
&
\end{longtable}
\noindent
where $d_{ij2}=0$  and $d'_{\kappa 2}=\imath_X/\zeta_X$ holds for $2\leq i \leq r$,   
$1 \leq \kappa \leq m$ and $\mu \in \ZZ$ must satisfy $\gcd(\imath_X, \zeta_X, \mu)=1$.
\end{enumerate}
\end{proposition}

We say that a matrix $P$ meeting these requirements is \emph{in normal form}. Proposition~\ref{prop:zetaexcep} tells us that there is a finite number of combinations of leading blocks and values of $\zeta$ and $\imath$. This is the basis for the classification in Section~\ref{sec:class}.

\subsection{The anticanonical complex}
As we already mentioned in the introduction, the columns of the matrix $P$ are the primitive ray generators of the cone $\sigma_P$ defining the minimal ambient toric variety $Z(P)$. Let $\Sigma={\rm fan}(\sigma_P)$ be the fan of $Z(P)$ consisting of all faces of $\sigma_P$. Now $X(P)$ is $\QQ$-factorial if and only if $Z(P)$ is, which is the case if and only if $P$ is square. Let $e_1,\ldots,e_{r+2}$ be the standard basis of the column space of $P$. Set $e_0:=-e_1-\ldots-e_r$. The \emph{tropical variety} of $X(P)$ is the fan ${\rm trop}(X)$ built up by the \emph{leaves}
$$
\lambda_i:=\cone(e_i) \times \lin(e_{r+1},e_{r+2})
$$
intersecting in the \emph{lineality part} $\lambda=\lin(e_{r+1},e_{r+2})$. Now a resolution of the singularity $X(P)$ can be obtained by the following two steps
\begin{enumerate}
\item Determine the coarsest common refinement $\Sigma':= \Sigma \sqcap \trop(X)$.
\item Compute a regular subdivision $\Sigma''$ of the fan $\Sigma'$.
\end{enumerate}

Now writing the primitive ray generators of the fan $\Sigma''$ in a matrix $P''$, we obtain a (non-affine) variety $X'':=X(P'')\subseteq Z(P'')$ with torus action of complexity one. The restriction  $X''\to X$ of the toric morphism $Z(P'')\to Z(P)$ coming from $\Sigma'' \to \Sigma$ is a resolution of singularities.

Next we define the \emph{anticanonical complex} $A_X^c$. Consider the degree map $Q:\ZZ^{n+m}\to K=\Cl(X)$, the relations $g_i$ from Construction~\ref{constr:RAP0} and denote by $e_\Sigma$ the sum over all $e_{ij}$ and $e_k$. Then the canonical divisor class of $X$ is
$$
\mathcal{K}_X=\sum\limits_{i=0}^{r-2} \deg(g_i) -Q(e_\Sigma) \in \Cl(X).
$$
Now let $B(-\mathcal{K}_X):=Q^{-1}\left(-\mathcal{K}_X\right) \cap \QQ_{\geq 0}^{n+m}$ and by $B:=\sum B(g_i)$ denote the Minkowski sum of the Newton polytopes of all $g_i$. 

\begin{definition}\cite[Def. 3.1]{ltpticr}.
Let $X=X(P)$ be a $\QQ$-Gorenstein log-terminal threefold singularity with two-torus action.
\begin{enumerate}
\item The \emph{anticanonical polyhedron} $A_X$ is the dual polyhedron of 
$$
B_X:= (P^*){-1}(B(-\mathcal{K}_X)+B-e_\Sigma).
$$
\item The \emph{anticanonical complex} of $X$ is the coarsest common refinement
$$
A_X^c:={\rm faces} (A_X) \sqcap \Sigma \sqcap \trop(X).
$$
\item The \emph{relative interior} of $A_X^c$ is the interior of its support with respect to the intersection ${\rm Supp}(\Sigma) \cap \trop(X)$ and the \emph{relative boundary} $\partial A_X^c$ is its complement in ${\rm Supp}(A_X^c)$.
\end{enumerate}
\end{definition}

Consider cones of the form
$
\tau = \cone(v_{0j_0},\ldots,v_{rj_r}),
$
i.e. cones generated by exactly one column from each leaf. If $\tau$ is a face of $\sigma_P$, then we call it a \emph{$P$-elementary cone}, otherwise a \emph{fake elementary cone}. To a ($P$- or fake) elementary cone $\tau$, we associate numbers
$$
\ell_{\tau,i}:=\frac{l_{0j_0} \cdots l_{rj_r}}{l_{ij_i}},
\quad
\ell_\tau:=(1-r)l_{0j_0} \cdots l_{rj_r} +\sum_{i=0}^r \ell_{\tau,i},
$$
and set
$$
v(\tau):=\sum_{i=0}^r \ell_{\tau,i} v_{ij_i},
\quad
v(\tau)':=\ell_\tau^{-1} v(\tau).
$$
Note that $v(\tau)$ is a (not necessarily primitive) generator of the ray $\lambda \cap \tau$ and $v(\tau)'$ may be rational. In terms of these definitions, we get a concrete description of $A_X^c$: it consists of the polytopes
$$
A_X^c(\lambda_i):= \conv(v_{ij},v(\tau)'; 1\leq j \leq n_i, \tau ~P\mathrm{-elementary})
$$
for $i=0,\ldots,r$ and all their faces. Important among these faces are the
$$
A_X^c(\lambda_i,\tau):= A_X^c(\lambda_i)  \cap \tau
$$
for $P$-elementary cones $\tau$.
Now let $X'' \to X$ be a resolution of the singularity $X$ and $\rho \in \Sigma''$ be a ray in $\Sigma''$. Let moreover $v_\rho$ be the primitive generator of $\rho$ and $v_\rho'$ be the \emph{leaving point} $\rho \cap \partial A_X^c$. Then Proposition 3.2 of~\cite{ltpticr} tells us that the discrepancy of the corresponding exceptional divisor $D_\rho$ equals
$$
a_\rho = \frac{\|v'_\rho\|}{\|v_\rho\|}-1.
$$
Now assume $P$ is in normal form.  Then for the standard coordinates $x_1,\ldots,x_{r+2}$ of the column space $\QQ^{r+2}$ of $P$ and setting $x_0=-\sum x_i$, we define the halfplanes
$$
\HHH_i := V\left(\zeta_X x_{r+2} + \frac{\imath_X-\zeta_X d_{i12}}{l_{i1}}x_i-\imath_X\right) \cap \lambda_i.
$$
Then Proposition 7.3 of~\cite{ltpticr} tells us that all columns $v_{ij}$ of $P$ lie on the halfplane $\HHH_i$, the columns $v_k$ of $P$ lie on $\HHH:=\bigcap \HHH_i \subseteq \lambda$ and that
$$
\partial A_X^c(\lambda_i):=\partial A_X^c \cap \lambda_i = A_X^c(\lambda_i) \cap \HHH_i.
$$
Thus together with the discrepancy formula from above, we get the following characterization of canonicity:

\begin{lemma}
Let $X=X(P)$ be a log-terminal threefold singularity with a two-torus action and let $P$ be in normal form.
Then $X$ is canonical if and only if one of the following statements holds:
\begin{enumerate}
\item Besides the origin, there is no lattice point in the relative interior of $A_X^c$.
\item The only lattice points  of the polytopes $A_X^c(\lambda_i)$ besides the origin lie on $\HHH_i$.
\end{enumerate}
Moreover, $X$ is terminal if and only if in addition the only lattice points on $\HHH_i$ are columns of $P$.
\end{lemma}

The following figures from~\cite{ltpticr} visualize the situation for the case $r=2$.
In the first one, the three-dimensional halfspaces $\lambda_i$ are shown, meeting (in $\QQ^4$) in the dark shaded subspace $\lambda$. In the second one in addition the halfplanes $\HHH_i$ turn up, while the third figure then shows the polytopes $A_X^c(\lambda_i)$ making up the anticanonical complex $A_X^c$. In the fourth one, the $A_X^c(\lambda_i,\tau)$ for a $P$-elementary cone $\tau$ are shown.

\begin{center}
\begin{tikzpicture}[scale=0.8]


\draw[thick, draw=black, fill=gray!20!, fill opacity=0.5] (0,-1,-1)--(0,-1,1)--(-1.8,-1,1)--(-1.8,-1,-1)--cycle;
\draw[thick, draw=black, fill=gray!20!, fill opacity=0.5] (0,-1,-1)--(0,1,-1)--(-1.8,1,-1)--(-1.8,-1,-1)--cycle;

\draw[thick, draw=black, fill=gray!20!, fill opacity=0.80] (0,-1,1)--(0,1,1)--(-1.8,1,1)--(-1.8,-1,1)--cycle;
\draw[thick, draw=black, fill=gray!10!, fill opacity=0.80] (0,1,1)--(0,1,-1)--(-1.8,1,-1)--(-1.8,1,1)--cycle;

\draw[thick, draw=black, fill=gray!60!, fill opacity=0.80] (0,-1,-1)--(0,-1,1)--(0,1,1)--(0,1,-1)--cycle;

\draw[thick, draw=black, fill=gray!90!, fill opacity=0.80]
(2,-1,-2.8)--(2,-1,-0.8)--(2,1,-0.8)--(2,1,-2.8)--cycle;

\draw[thick, draw=black, fill=gray!20!, fill opacity=0.5] (2,-1,-2.8)--(2,-1,-0.8)--(3.8,-1,-1)--(3.8,-1,-3)--cycle;
\draw[thick, draw=black, fill=gray!20!, fill opacity=0.5] (2,-1,-2.8)--(2,1,-2.8)--(3.8,1,-3)--(3.8,-1,-3)--cycle;

\draw[thick, draw=black, fill=gray!10!, fill opacity=0.80] (2,1,-0.8)--(2,1,-2.8)--(3.8,1,-3)--(3.8,1,-1)--cycle;
\draw[thick, draw=black, fill=gray!20!, fill opacity=0.80] (2,1,-0.8)--(2,-1,-0.8)--(3.8,-1,-1)--(3.8,1,-1)--cycle;

\draw[thick, draw=black, fill=gray!90!, fill opacity=0.80] (2,-1,1.3)--(2,-1,3.3)--(2,1,3.3)--(2,1,1.3)--cycle;

\draw[thick, draw=black, fill=gray!20!, fill opacity=0.5] (2,-1,1.3)--(2,-1,3.3)--(4.3,-1,4.6)--(4.3,-1,2.6)--cycle;
\draw[thick, draw=black, fill=gray!20!, fill opacity=0.5] (2,-1,1.3)--(2,1,1.3)--(4.3,1,2.6)--(4.3,-1,2.6)--cycle;

\draw[thick, draw=black, fill=gray!10!, fill opacity=0.80] (2,1,3.3)--(2,1,1.3)--(4.3,1,2.6)--(4.3,1,4.6)--cycle;
\draw[thick, draw=black, fill=gray!20!, fill opacity=0.80] (2,1,3.3)--(2,-1,3.3)--(4.3,-1,4.6)--(4.3,1,4.6)--cycle;

\end{tikzpicture}
\hfill
\begin{tikzpicture}[scale=0.8]


\draw[thick, draw=black, fill=gray!20!, opacity=0.2] (0,-1,-1)--(0,-1,1)--(-1.8,-1,1)--(-1.8,-1,-1)--cycle;
\draw[thick, draw=black, fill=gray!20!, opacity=0.2] (0,-1,-1)--(0,1,-1)--(-1.8,1,-1)--(-1.8,-1,-1)--cycle;

\foreach \i in {0,-0.4,...,-1.6}
	{
\draw[draw=black] (\i,\i/4,1)--(\i,\i/4,-1);	
	}
	
\foreach \i in {0.8,0.4,...,-1.2}
	{
\draw[draw=black] (0,0,\i)--(-1.8,-0.45,\i);	
	}

\draw[thick, draw=black, fill=gray!20!, opacity=0.2] (0,-1,1)--(0,1,1)--(-1.8,1,1)--(-1.8,-1,1)--cycle;
\draw[thick, draw=black, fill=gray!10!, opacity=0.2] (0,1,1)--(0,1,-1)--(-1.8,1,-1)--(-1.8,1,1)--cycle;

\draw[thick, draw=black, fill=gray!60!, opacity=0.2] (0,-1,-1)--(0,-1,1)--(0,1,1)--(0,1,-1)--cycle;

\draw[thick, draw=black, fill=gray!90!, opacity=0.2]
(2,-1,-2.8)--(2,-1,-0.8)--(2,1,-0.8)--(2,1,-2.8)--cycle;

\draw[thick, draw=black, fill=gray!20!, opacity=0.2] (2,-1,-2.8)--(2,-1,-0.8)--(3.8,-1,-1)--(3.8,-1,-3)--cycle;
\draw[thick, draw=black, fill=gray!20!, opacity=0.2] (2,-1,-2.8)--(2,1,-2.8)--(3.8,1,-3)--(3.8,-1,-3)--cycle;

\foreach \i in {0,0.4,...,1.6}
	{
\draw[draw=black] (2+\i,0,-0.8-\i*0.1)--(2+\i,0,-2.8-\i*0.1);	
	}
	
\foreach \i in {-1,-1.4,...,-2.6}
	{
\draw[draw=black] (2,0,\i)--(3.8,0,-0.2+\i);	
	}

\draw[thick, draw=black, fill=gray!10!, opacity=0.2] (2,1,-0.8)--(2,1,-2.8)--(3.8,1,-3)--(3.8,1,-1)--cycle;
\draw[thick, draw=black, fill=gray!20!, opacity=0.2] (2,1,-0.8)--(2,-1,-0.8)--(3.8,-1,-1)--(3.8,1,-1)--cycle;

\draw[thick, draw=black, fill=gray!90!, opacity=0.2] (2,-1,1.3)--(2,-1,3.3)--(2,1,3.3)--(2,1,1.3)--cycle;

\draw[thick, draw=black, fill=gray!20!, opacity=0.2] (2,-1,1.3)--(2,-1,3.3)--(4.3,-1,4.6)--(4.3,-1,2.6)--cycle;
\draw[thick, draw=black, fill=gray!20!, opacity=0.2] (2,-1,1.3)--(2,1,1.3)--(4.3,1,2.6)--(4.3,-1,2.6)--cycle;

\foreach \i in {0,0.4,...,2}
	{
\draw[draw=black] (2+\i,0,1.3+\i*0.5652174)--(2+\i,0,3.3+\i*0.5652174);	
	}
	
\foreach \i in {1.5,1.9,...,3.1}
	{
\draw[draw=black] (2,0,\i)--(4.3,0,1.3+\i);	
	}

\draw[thick, draw=black, fill=gray!10!, opacity=0.2] (2,1,3.3)--(2,1,1.3)--(4.3,1,2.6)--(4.3,1,4.6)--cycle;
\draw[thick, draw=black, fill=gray!20!, opacity=0.2] (2,1,3.3)--(2,-1,3.3)--(4.3,-1,4.6)--(4.3,1,4.6)--cycle;

\end{tikzpicture}
\hfill
\begin{tikzpicture}[scale=0.8]


\draw[thick, draw=black, fill=gray!20!, opacity=0.2] (0,-1,-1)--(0,-1,1)--(-1.8,-1,1)--(-1.8,-1,-1)--cycle;
\draw[thick, draw=black, fill=gray!20!, opacity=0.2] (0,-1,-1)--(0,1,-1)--(-1.8,1,-1)--(-1.8,-1,-1)--cycle;

\draw[draw=black, fill=gray!20!, opacity=0.8] (0,0,-0.4)--(-0.8,-0.2,-0.4)--(0,-0.5,0)--cycle;
\draw[draw=black, fill=gray!20!, opacity=0.8] (-0.8,-0.2,-0.4)--(-1.2,-0.3,0.4)--(0,-0.5,0)--cycle;
\draw[draw=black, fill=gray!20!, opacity=0.8] (0,0,0.4)--(-1.2,-0.3,0.4)--(0,-0.5,0)--cycle;
\draw[draw=black, fill=gray!50!, opacity=0.8] (0,0,-0.4)--(-0.8,-0.2,-0.4)--(-1.2,-0.3,0.4)--(0,0,0.4)--cycle;

\foreach \i in {0,-0.4,...,-1.6}
	{
\draw[draw=black, opacity=0.4] (\i,\i/4,1)--(\i,\i/4,-1);	
	}
	
\foreach \i in {0.8,0.4,...,-1.2}
	{
\draw[draw=black, opacity=0.4] (0,0,\i)--(-1.8,-0.45,\i);	
	}

\draw[thick, draw=black, fill=gray!20!, opacity=0.2] (0,-1,1)--(0,1,1)--(-1.8,1,1)--(-1.8,-1,1)--cycle;
\draw[thick, draw=black, fill=gray!10!, opacity=0.2] (0,1,1)--(0,1,-1)--(-1.8,1,-1)--(-1.8,1,1)--cycle;

\draw[thick, draw=black, fill=gray!60!, opacity=0.2] (0,-1,-1)--(0,-1,1)--(0,1,1)--(0,1,-1)--cycle;

\draw[thick, draw=black, fill=gray!90!, opacity=0.2]
(2,-1,-2.8)--(2,-1,-0.8)--(2,1,-0.8)--(2,1,-2.8)--cycle;

\draw[thick, draw=black, fill=gray!20!, opacity=0.2] (2,-1,-2.8)--(2,-1,-0.8)--(3.8,-1,-1)--(3.8,-1,-3)--cycle;
\draw[thick, draw=black, fill=gray!20!, opacity=0.2] (2,-1,-2.8)--(2,1,-2.8)--(3.8,1,-3)--(3.8,-1,-3)--cycle;

\draw[draw=black, fill=gray!20!, opacity=0.8] (2,0,-2.2)--(3.2,0,-1.92)--(2,-0.5,-1.8)--cycle;
\draw[draw=black, fill=gray!20!, opacity=0.8] (2,0,-1.4)--(3.2,0,-1.92)--(2,-0.5,-1.8)--cycle;
\draw[draw=black, fill=gray!50!, opacity=0.8] (2,0,-2.2)--(3.2,0,-1.92)--(2,0,-1.4)--cycle;

\foreach \i in {0,0.4,...,1.6}
	{
\draw[draw=black, opacity=0.4] (2+\i,0,-0.8-\i*0.1)--(2+\i,0,-2.8-\i*0.1);	
	}
	
\foreach \i in {-1,-1.4,...,-2.6}
	{
\draw[draw=black, opacity=0.4] (2,0,\i)--(3.8,0,-0.2+\i);	
	}

\draw[thick, draw=black, fill=gray!10!, opacity=0.2] (2,1,-0.8)--(2,1,-2.8)--(3.8,1,-3)--(3.8,1,-1)--cycle;
\draw[thick, draw=black, fill=gray!20!, opacity=0.2] (2,1,-0.8)--(2,-1,-0.8)--(3.8,-1,-1)--(3.8,1,-1)--cycle;

\draw[thick, draw=black, fill=gray!90!, opacity=0.2] (2,-1,1.3)--(2,-1,3.3)--(2,1,3.3)--(2,1,1.3)--cycle;

\draw[thick, draw=black, fill=gray!20!, opacity=0.2] (2,-1,1.3)--(2,-1,3.3)--(4.3,-1,4.6)--(4.3,-1,2.6)--cycle;
\draw[thick, draw=black, fill=gray!20!, opacity=0.2] (2,-1,1.3)--(2,1,1.3)--(4.3,1,2.6)--(4.3,-1,2.6)--cycle;

\draw[draw=black, fill=gray!20!, opacity=0.8] (2,0,1.9)--(3.6,0,4)--(2,-0.5,2.3)--cycle;
\draw[draw=black, fill=gray!20!, opacity=0.8] (2,0,2.7)--(3.6,0,4)--(2,-0.5,2.3)--cycle;
\draw[draw=black, fill=gray!50!, opacity=0.8] (2,0,2.7)--(3.6,0,4)--(2,0,1.9)--cycle;

\foreach \i in {0,0.4,...,2}
	{
\draw[draw=black, opacity=0.2] (2+\i,0,1.3+\i*0.5652174)--(2+\i,0,3.3+\i*0.5652174);	
	}
	
\foreach \i in {1.5,1.9,...,3.1}
	{
\draw[draw=black, opacity=0.2] (2,0,\i)--(4.3,0,1.3+\i);	
	}

\draw[thick, draw=black, fill=gray!10!, opacity=0.2] (2,1,3.3)--(2,1,1.3)--(4.3,1,2.6)--(4.3,1,4.6)--cycle;
\draw[thick, draw=black, fill=gray!20!, opacity=0.2] (2,1,3.3)--(2,-1,3.3)--(4.3,-1,4.6)--(4.3,1,4.6)--cycle;

\end{tikzpicture}
\hfill
\begin{tikzpicture}[scale=0.8]


\draw[thick, draw=black, fill=gray!20!, opacity=0.2] (0,-1,-1)--(0,-1,1)--(-1.8,-1,1)--(-1.8,-1,-1)--cycle;
\draw[thick, draw=black, fill=gray!20!, opacity=0.2] (0,-1,-1)--(0,1,-1)--(-1.8,1,-1)--(-1.8,-1,-1)--cycle;

\draw[draw=black, fill=gray!20!, opacity=0.2] (0,0,-0.4)--(-0.8,-0.2,-0.4)--(0,-0.5,0)--cycle;
\draw[draw=black, fill=gray!20!, opacity=0.2] (-0.8,-0.2,-0.4)--(-1.2,-0.3,0.4)--(0,-0.5,0)--cycle;
\draw[draw=black, fill=gray!80!, opacity=0.9] (0,0,0.4)--(-1.2,-0.3,0.4)--(0,-0.5,0)--cycle;
\draw[thick, draw=black, opacity=0.9] (0,0,0.4)--(-1.2,-0.3,0.4);
\draw[draw=black, fill=gray!50!, opacity=0.2] (0,0,-0.4)--(-0.8,-0.2,-0.4)--(-1.2,-0.3,0.4)--(0,0,0.4)--cycle;

\foreach \i in {0,-0.4,...,-1.6}
	{
\draw[draw=black, opacity=0.4] (\i,\i/4,1)--(\i,\i/4,-1);	
	}
	
\foreach \i in {0.8,0.4,...,-1.2}
	{
\draw[draw=black, opacity=0.4] (0,0,\i)--(-1.8,-0.45,\i);	
	}

\draw[thick, draw=black, fill=gray!20!, opacity=0.2] (0,-1,1)--(0,1,1)--(-1.8,1,1)--(-1.8,-1,1)--cycle;
\draw[thick, draw=black, fill=gray!10!, opacity=0.2] (0,1,1)--(0,1,-1)--(-1.8,1,-1)--(-1.8,1,1)--cycle;

\draw[thick, draw=black, fill=gray!60!, opacity=0.2] (0,-1,-1)--(0,-1,1)--(0,1,1)--(0,1,-1)--cycle;

\draw[thick, draw=black, fill=gray!90!, opacity=0.2]
(2,-1,-2.8)--(2,-1,-0.8)--(2,1,-0.8)--(2,1,-2.8)--cycle;

\draw[thick, draw=black, fill=gray!20!, opacity=0.2] (2,-1,-2.8)--(2,-1,-0.8)--(3.8,-1,-1)--(3.8,-1,-3)--cycle;
\draw[thick, draw=black, fill=gray!20!, opacity=0.2] (2,-1,-2.8)--(2,1,-2.8)--(3.8,1,-3)--(3.8,-1,-3)--cycle;

\draw[draw=black, fill=gray!20!, opacity=0.2] (2,0,-2.2)--(3.2,0,-1.92)--(2,-0.5,-1.8)--cycle;
\draw[draw=black, fill=gray!80!, opacity=0.9] (2,0,-1.4)--(3.2,0,-1.92)--(2,-0.5,-1.8)--cycle;
\draw[thick, draw=black, opacity=0.9] (2,0,-1.4)--(3.2,0,-1.92);
\draw[draw=black, fill=gray!50!, opacity=0.2] (2,0,-2.2)--(3.2,0,-1.92)--(2,0,-1.4)--cycle;

\foreach \i in {0,0.4,...,1.6}
	{
\draw[draw=black, opacity=0.4] (2+\i,0,-0.8-\i*0.1)--(2+\i,0,-2.8-\i*0.1);	
	}
	
\foreach \i in {-1,-1.4,...,-2.6}
	{
\draw[draw=black, opacity=0.4] (2,0,\i)--(3.8,0,-0.2+\i);	
	}

\draw[thick, draw=black, fill=gray!10!, opacity=0.2] (2,1,-0.8)--(2,1,-2.8)--(3.8,1,-3)--(3.8,1,-1)--cycle;
\draw[thick, draw=black, fill=gray!20!, opacity=0.2] (2,1,-0.8)--(2,-1,-0.8)--(3.8,-1,-1)--(3.8,1,-1)--cycle;

\draw[thick, draw=black, fill=gray!90!, opacity=0.2] (2,-1,1.3)--(2,-1,3.3)--(2,1,3.3)--(2,1,1.3)--cycle;

\draw[thick, draw=black, fill=gray!20!, opacity=0.2] (2,-1,1.3)--(2,-1,3.3)--(4.3,-1,4.6)--(4.3,-1,2.6)--cycle;
\draw[thick, draw=black, fill=gray!20!, opacity=0.2] (2,-1,1.3)--(2,1,1.3)--(4.3,1,2.6)--(4.3,-1,2.6)--cycle;

\draw[draw=black, fill=gray!20!, opacity=0.2] (2,0,1.9)--(3.6,0,4)--(2,-0.5,2.3)--cycle;
\draw[draw=black, fill=gray!80!, opacity=0.9] (2,0,2.7)--(3.6,0,4)--(2,-0.5,2.3)--cycle;
\draw[thick, draw=black, opacity=0.9] (2,0,2.7)--(3.6,0,4);
\draw[draw=black, fill=gray!50!, opacity=0.2] (2,0,2.7)--(3.6,0,4)--(2,0,1.9)--cycle;

\foreach \i in {0,0.4,...,2}
	{
\draw[draw=black, opacity=0.2] (2+\i,0,1.3+\i*0.5652174)--(2+\i,0,3.3+\i*0.5652174);	
	}
	
\foreach \i in {1.5,1.9,...,3.1}
	{
\draw[draw=black, opacity=0.2] (2,0,\i)--(4.3,0,1.3+\i);	
	}

\draw[thick, draw=black, fill=gray!10!, opacity=0.2] (2,1,3.3)--(2,1,1.3)--(4.3,1,2.6)--(4.3,1,4.6)--cycle;
\draw[thick, draw=black, fill=gray!20!, opacity=0.2] (2,1,3.3)--(2,-1,3.3)--(4.3,-1,4.6)--(4.3,1,4.6)--cycle;

\end{tikzpicture}
\end{center}

\section{Polytopes}
\label{sec:pol}

\subsection{$k$-empty polytopes}

In the following we give a description of a standard form of $k$-empty
lattice triangles. Then, the Farey sequences are used to classify those. We set
$\Delta \ := \ \left\{ (x,y) \in \mathbb{Z}^{2}; \ 0 \le y < x \right\}$. Members of
$k\mathbb{Z}^{n}$ for $k \in \mathbb{Z}_{\ge 1}$ are called
\emph{$k$-fold lattice points}.

\begin{definition}
Let $n,k \in \mathbb{Z}_{\ge 1}$ and consider a convex rational polytope
$\mathsf{P} \subseteq \mathbb{Q}^{n}$. The set of vertices of $\mathsf{P}$ is
denoted by $\mathcal{V}(\mathsf{P})$, the interior by $\mathsf{P}^{\circ}$ and
the boundary by $\partial \mathsf{P}$. We call $\mathsf{P}$

\begin{enumerate}
\item a \emph{lattice polytope}, if
$\mathcal{V}(\mathsf{P}) \subseteq \mathbb{Z}^{n}$.

\item a \emph{lattice polygon}, if $\mathsf{P}$ is a lattice polytope and
$n=2$.

\item \emph{$k$-empty}, if
$\mathsf{P} \cap k \mathbb{Z}^{n} \subseteq \mathcal{V}(\mathsf{P})$.
\end{enumerate}
\end{definition}

\begin{definition}
The group $\mathrm{Aff}_{k}^{n}(\mathbb{Z})$ of
\emph{$k$-affine unimodular transformations in $\mathbb{Q}^{n}$} is
defined by
\[
\mathrm{Aff}_{k}^{n}(\mathbb{Z}) \ := \ \left\{ T \colon
\mathbb{Q}^{n} \rightarrow \mathbb{Q}^{n}; \  T(v) = Av+w, \ 
A \in \mathrm{GL}_{n}(\mathbb{Z}), w \in k\mathbb{Z}^{n} \right\}.
\]
It naturally acts on the set of lattice polytopes in $\mathbb{Q}^{n}$. Lattice
polytopes $\mathsf{P}_{1}$ and $\mathsf{P}_{2}$ are called
\emph{$k$-equivalent}, if $\mathsf{P}_{2} \in \mathrm{Aff}_{k}(\mathbb{Z})
\cdot \mathsf{P}_{1}$. Additionally, we call $1$-equivalent polytopes
\emph{lattice equivalent}.
\end{definition}

\begin{remark}
Let $T \in \mathrm{Aff}_{k}^{n}(\mathbb{Z})$ and $\mathsf{P}$ be a
lattice polytope. Then the following hold.

\begin{enumerate}
\item $T(\mathbb{Z}^{n}) = \mathbb{Z}^{n}$.

\item $T(k\mathbb{Z}^{n}) = k\mathbb{Z}^{n}$.

\item $T(\mathcal{V}(\mathsf{P})) = \mathcal{V}(T(\mathsf{P}))$.

\item $T(\partial \mathsf{P}) = \partial T(\mathsf{P})$.

\item $T(\mathsf{P}^{\circ}) = T(\mathsf{P})^{\circ}$.

\item $\mathrm{vol}(\mathsf{P}) = \mathrm{vol}(T(\mathsf{P}))$.
\end{enumerate}

Therefore, the number of vertices, the number
of (interior) lattice points and the number of (interior) $k$-fold
lattice points are invariant under the action of
$\mathrm{Aff}_{k}^{n}(\mathbb{Z})$.

Note that $k$-equivalence of lattice polytopes indeed depends
on the specific value of $k$. Consider for example the following
pair of lattice polygons. They are $1$-equivalent but not $2$-
equivalent. The marked point is the origin.

\begin{center}
\begin{tikzpicture}[scale=0.8]

\draw (0,0) circle (6pt);
\draw (3,0) circle (6pt);

\draw [help lines, dashed] (0,0) grid (2,1);
\draw [very thick] (0,0) -- (0,1) -- (2,0) -- (0,0);

\draw [help lines, dashed] (3,0) grid (5,1);
\draw [very thick] (3,0) -- (3,1) -- (5,1) -- (3,0);

\end{tikzpicture}
\end{center}
\end{remark}

\begin{definition}\label{ap}
Let $k \in \mathbb{Z}_{\ge 1}$ and $\mathsf{P}$ be a lattice polygon. We
set
\[
a_{\mathsf{P}} \ := \ \min \left\{ \text{number of lattice points in the
relative
interior of} \ \mathsf{E} \right\} \ + \ 1
\]
where $\mathsf{E}$ runs through the edges of $\mathsf{P}$ which have a
vertex
in $k\mathbb{Z}^{2}$.
\end{definition}

\begin{definition}[Standard form of $k$-empty lattice triangles]
\label{standardform}
Let $k \in \mathbb{Z}_{\ge 1}$ and $\mathsf{S}$ be a $k$-empty lattice
polygon with exactly three vertices such that one of them is in
$k\mathbb{Z}^{2}$. We refer to $\mathsf{S}$ as in \emph{standard form},
if the following conditions are satisfied.

\begin{enumerate}
\item $\mathsf{S}$ has the vertices $(0,0), (0,a_{\mathsf{S}})$ and
$(x,y)$ where
$(x,y) \in \Delta$.

\item If $(x,y) \notin k\mathbb{Z}^{2}$ and $\mathrm{gcd}
(x,y)=a_{\mathsf{S}}$ then for each $z=1,\dots,y-1$
we have $a_{\mathsf{S}} \nmid z$ or
$a_{\mathsf{S}}x \nmid a_{\mathsf{S}}^{2}-zy$.

\item If $(x,y) \in k\mathbb{Z}^{2}$ and $\mathrm{gcd}
(x,y)=a_{\mathsf{S}}$ then for each $z=1,\dots,y-1$ we have 
$a_{\mathsf{S}} \nmid z$ or
$a_{\mathsf{S}}x \nmid a_{\mathsf{S}}(z+y)-zy$.
\end{enumerate}

If $\mathsf{S}$ is in standard form, we write
$\mathsf{S}=\Delta(a_{\mathsf{S}},x,y)$. The simplex $\mathsf{S}$ is 
called \emph{minimal}, if $a_{\mathsf{S}}=1$.
\end{definition}

\begin{remark}
The last two conditions of Definition \ref{standardform} ensure that the 
second coordinate of the vertex $(x,y) \in \Delta$ is minimal. To
illustrate this, consider the $2$-equivalent $2$-empty polytopes
\begin{align*}
\mathsf{P}_{1} \ &= \ \mathrm{conv} \left( (0,0),(0,1),(5,3) \right),\\
\mathsf{P}_{2} \ &= \ \mathrm{conv} \left( (0,0),(0,1),(5,2) \right).
\end{align*}
We can see that $\mathsf{P}_{1}$ does not fulfill condition (ii) so it is
not in standard form whereas $\mathsf{P}_{2}$ is. Therefore, these
conditions make sure that the \emph{right} vertex lies on the vertical
axis. This is relevant in case that there are several edges which attain
the minimum of Definition \ref{ap}.
\end{remark}

\begin{proposition}
Let $\mathsf{S}$ be a $k$-empty lattice triangle with a vertex
$z \in k\mathbb{Z}^{2}$. Then there is a unique lattice triangle
$\mathsf{S}'$ in standard form that is $k$-equivalent to $\mathsf{S}$.
\end{proposition}

\begin{proof}
There are three cases depending on the number of vertices of
$\mathsf{S}$ in $k\mathbb{Z}^{2}$.

\smallskip

\noindent \emph{Case $1$: There is exactly one vertex
$z \in k\mathbb{Z}^{2}$.} Let $v_{1}$ and $v_{2}$ be the other two
vertices and $d_{i}$ the number of lattice points in the relative
interior of the edge of $\mathsf{S}$ with vertices $z$ and $v_{i}$.
We have $a_{\mathsf{S}}=\min \left\{ d_{1},d_{2} \right\} +1$.

\smallskip

\noindent \emph{Case $1.1$: $d_{1} \ne d_{2}$.} We can assume, without
loss of generality, that $d_{1}<d_{2}$. So we have
$a_{\mathsf{S}}=d_{1}+1$. Consider the $k$-affine unimodular
transformation $T_{1}$ given by $T_{1}(v)=v-z$. The coordinates of
$T_{1}(v_{1})$ have the greatest common divisor $a_{\mathsf{S}}$. Thus,
there is a $k$-affine unimodular transformation $T_{2}$ that leaves the
origin fixed and takes $T_{1}(v_{1})$ to $(0,a_{\mathsf{S}})$. A third
transformation $T_{3}$ sends $T_{2}(T_{1}(v_{2}))$ to a point
$(x,y) \in \Delta$ without changing the coordinates of
$T_{2}(T_{1}(v_{1}))$ and $T_{2}(T_{1}(z))$. The $k$-empty lattice
triangle $\mathsf{S}':=T_{3} \circ T_{2} \circ T_{1}(S)$ satisfies the
conditions of Definition \ref{standardform}. Note that
$\mathrm{gcd}(x,y) \ne a_{\mathsf{S}}$ since $d_{1} \ne d_{2}$. The
representation is obviously unique in this case.

\smallskip

\noindent \emph{Case $1.2$: $d_{1}=d_{2}$.} As before, let $T_{1}$ be
given by $T_{1}(v)=v-z$. Then, let $T_{2}$ be the $k$-affine unimodular
transformation that leaves the origin fixed and sends $T_{1}(v_{1})$ to
$(0,a_{\mathsf{S}})$. We choose a transformation $T_{3}$ that takes
$T_{2}(T_{1}(v_{2}))$ to a point $(x,y) \in \Delta$ without changing the
coordinates of $(0,a_{\mathsf{S}})$ and $(0,0)$. Analogously, let
$T_{2}'$ be the $k$-affine unimodular transformation that leaves the
origin fixed and sends $T_{1}(v_{2})$ to $(0,a_{\mathsf{S}})$. We choose
a transformation $T_{3}'$ that takes $T_{2}'(T_{1}(v_{1}))$ to a
point $(x,y') \in \Delta$ without changing the coordinates of
$(0,a_{\mathsf{S}})$ and $(0,0)$. Without loss of generality we have
$y<y'$. Set $\mathsf{S}':=T_{3} \circ T_{2} \circ T_{1}(S)$.

The lattice triangle $\mathsf{S}'$ fulfills condition (i) of Definition
\ref{standardform}. Suppose that it does not satisfy condition (ii).
That means there is a $z$ with $1 \le z \le y-1$ such that 
$a_{\mathsf{S}}|z$ and $a_{\mathsf{S}}x|a_{\mathsf{S}}^{2}-zy$. Consider 
the transformation $T$ given by the matrix
\[
A \ = \
\begin{pmatrix}
-\frac{y}{a_{\mathsf{S}}} & \frac{x}{a_{\mathsf{S}}}\\
\frac{a_{\mathsf{S}}^{2}-zy}{a_{\mathsf{S}}x} & \frac{z}{a_{\mathsf{S}}}
\end{pmatrix}
\ \in \ \mathrm{GL}_{2}(\mathbb{Z})
\]
and apply it to $\mathsf{S}'$. Since $(x,z) \in \Delta$ is a vertex of
of $T(\mathsf{S}')$, we have $z=y'$. So $y'=z<y<y'$ which is a
contradiction. Thus $\mathsf{S}'$ is in standard form and by
construction unique.

\smallskip

\noindent \emph{Case $2$: There are exactly two vertices 
$z_{1}, z_{2} \in k\mathbb{Z}^{2}$.} Let $v$ be the third vertex and
$d_{i}$ the number of lattice points in the relative interior of the
edge of $\mathsf{S}$ with vertices $z_{i}$ and $v$. As in Case $1$
we have $a_{\mathsf{S}}=\min \left\{ d_{1},d_{2} \right\} +1$.

\smallskip

\noindent \emph{Case $2.1$: $d_{1} \ne d_{2}$.} Without loss of
generality $d_{1}<d_{2}$. So we have $a_{\mathsf{S}}=d_{1}+1$. Consider
the $k$-affine unimodular transformation $T_{1}$ which is given by
$T_{1}(v')=v'-z_{1}$. Let $T_{2}$ be such that $T_{2}(0)=0$ and
$T_{2}(v-z_{1})=(0,a_{\mathsf{S}})$. Then, choose a third transformation
$T_{3}$ which fixes $(0,0$ and $(0,a_{\mathsf{S}})$ and sends
$T_{2}(T_{1}(z_{2}))$ to a point $(x,y) \in \Delta$. The lattice
triangle $\mathsf{S}':=T_{3} \circ T_{2} \circ T_{1}(S)$ satisfies the
conditions of Definition \ref{standardform}. As before, note that
$\mathrm{gcd}(x,y) \ne a_{\mathsf{S}}$ since $d_{1} \ne d_{2}$. The 
representation is unique.

\smallskip

\noindent \emph{Case $2.2$: $d_{1}=d_{2}$.} Let $T_{1}$ be given by
$T_{1}(v')=v'-z_{1}$ and $T_{2}$ be the transformation fixing the
origin and $T_{2}(T_{1}(v))=(0,a_{\mathsf{S}})$. Choose additionally
$T_{3}$ such that the origin and $(0,a_{\mathsf{S}})$ are fixed and
$(x,y) := T_{3}(T_{2}(T_{1}(z_{2}))) \in \Delta$. Accordingly, let
$T_{1}'$ be defined by $T_{1}'(v')=v'-z_{2}$ and $T_{2}'$ fixing $(0,0)$
and $T_{2}'(T_{1}'(v))=(0,a_{\mathsf{S}})$. Finally, choose a
transformation $T_{3}'$ which fixes $(0,0)$ and $(0,a_{\mathsf{S}})$ and
$(x,y') := T_{3}'(T_{2}'(T_{1}'(z_{2}))) \in \Delta$. Again, without
loss of generality, we have $y<y'$. We set
$\mathsf{S}':=T_{3} \circ T_{2} \circ T_{1}(S)$.

Assume that condition (iii) of Definition \ref{standardform} were not 
satisfied. Then there is a $z$ with $1 \le z \le y-1$ such that 
$a_{\mathsf{S}}|z$ and $a_{\mathsf{S}}x|a_{\mathsf{S}}(z+y)-zy$. 
Consider  the transformation $T$ given by the matrix
\[
A \ = \
\begin{pmatrix}
1-\frac{y}{a_{\mathsf{S}}} & \frac{x}{a_{\mathsf{S}}}\\
\frac{a_{\mathsf{S}}(z+y)-zy}{a_{\mathsf{S}}x} & \frac{z}
{a_{\mathsf{S}}}
-1
\end{pmatrix}
\ \in \ \mathrm{GL}_{2}(\mathbb{Z})
\]
and apply it to the lattice triangle
\[
\mathrm{conv} \left( \left( 0,0 \right),
\left( x,y-a_{\mathsf{S}}) \right), \left( x,y \right) \right)
\]
which is $k$-equivalent to $\mathsf{S}'$. The lattice triangle 
$T(\mathsf{S}')$ has the vertex $(x,z)$ and so we have $z=y'$. That
means $y'=z<y<y'$ which is a contradiction. Therefore $\mathsf{S}'$ is
in standard form and the uniqueness is clear by construction.

\smallskip

\noindent \emph{Case $3$: There are three vertices 
$z_{1}, z_{2}, z_{3} \in k\mathbb{Z}^{2}$.} Since the number of lattice
points in the relative interior of each edge of $\mathsf{S}$ is $k-1$,
the only possible standard form is the lattice triangle
$\mathrm{conv} \left( \left( 0,0 \right), \left( 0,k \right), \left( k,0 
\right) \right)$.
\end{proof}

\begin{definition}
Let $k \in \mathbb{Z}_{\ge 1}$. We define the
\emph{$k$-th Farey sequence} to be
\[
F_{k} \ := \ \left( \frac{f_{1}}{f_{2}}; \ 0 \le f_{1} < f_{2} \le k, \ 
\mathrm{gcd}(f_{1},f_{2})=1\right).
\]
Members of $F_{k}$ are called \emph{$k$-th Farey numbers}.

Let $f=\frac{f_{1}}{f_{2}}$ be a $k$-th Farey number. We define the
\emph{$k$-th Farey strip} corresponding to $f$ to be the
polyhedron
\[
F_{k,f} \ := \ 
\begin{cases}
  \left\{ \begin{pmatrix}x\\y\end{pmatrix}; \ 0 \ < \ \begin{pmatrix}
x\\y\end{pmatrix} \cdot \begin{pmatrix}-f_{1}\\f_{2}\end{pmatrix}
\ < \ k \right\}, \ &\mathrm{if} \ f_{2} = k,\\
  \left\{ \begin{pmatrix}x\\y\end{pmatrix}; \ 0 \ < \ \begin{pmatrix}
x\\y\end{pmatrix} \cdot \begin{pmatrix}-f_{1}\\f_{2}\end{pmatrix}
\ \le \ k \right\}, \ &\mathrm{if} \ f_{2} \ne k.
\end{cases}
\]
\end{definition}

\begin{remark}
By definition, the number of $k$-th Farey strips equals the length of
the $k$-th Farey sequence. That is, there are $\varphi(1)+\dots+
\varphi(k)$ many $k$-th Farey strips where $\varphi$ is the Euler 
totient function.
\end{remark}

\begin{definition}
A \emph{spike} attached to a $k$-th Farey strip $F_{k,f}$ is a
$2$-dimensional convex  polytope $\mathsf{S}$ with exactly three
rational vertices satisfying the following  conditions.

\begin{enumerate}

\item Two of the vertices of $\mathsf{S}$ are on the line
$y=\frac{f_{1}}{f_{2}}x+\frac{f_{2}}{k}$.

\item One vertex of $\mathsf{S}$ is above the line
$y=\frac{f_{1}}{f_{2}}x+\frac{f_{2}}{k}$.

\item If $(x,y)$ is a lattice point in $\mathsf{S}$, then
$\mathrm{conv}((0,0),(0,1),(x,y))$ is $k$-empty.

\end{enumerate}
\end{definition}

The following picture shows the $k$-th Farey strips and spikes attached 
to them for $k=3$. There are four strips where spikes are attached to
only two of them.

\begin{center}
\includegraphics[scale=0.3]{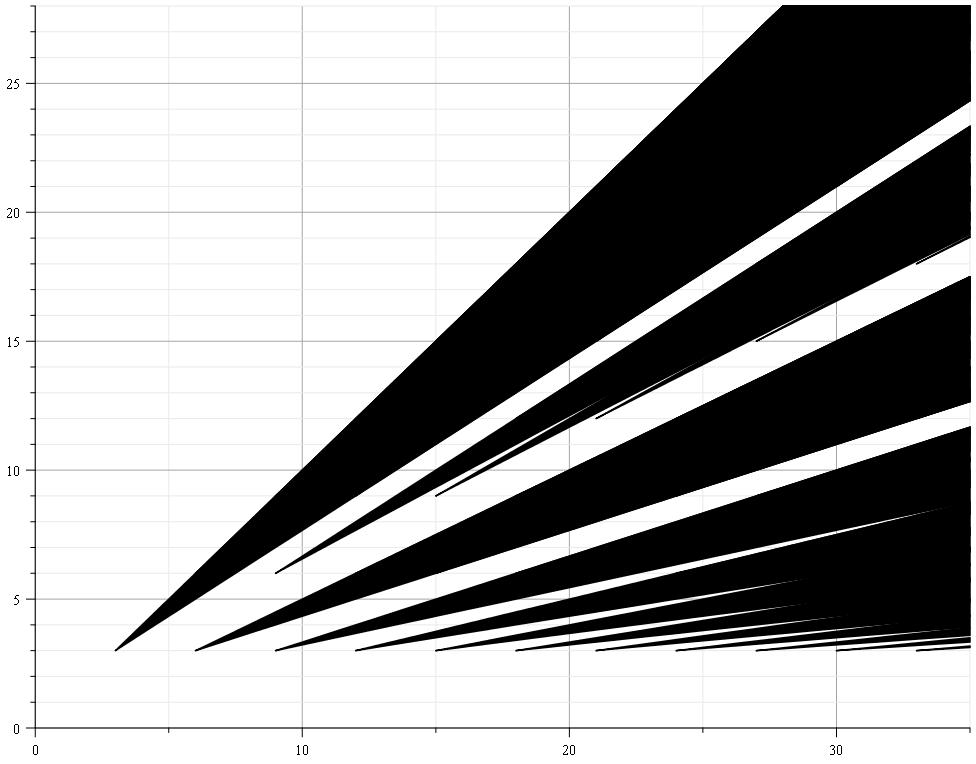}
\end{center}

\begin{proposition}
Let $k \in \mathbb{Z}_{\ge 1}$ and $f=\frac{f_{1}}{f_{2}}$ be a $k$-th
Farey number. If the lattice triangle
$\mathsf{S}=\mathrm{conv}((0,0),(0,1),(x,y))$ is contained in the $k$- th
Farey strip $F_{k,f}$, it is $k$-empty.
\end{proposition}

\begin{proof}
Assume that $\mathsf{S}$ is not $k$-empty. Then we  find
$a,b \in \mathbb{Z}_{\ge 1}$ such that
$(ka,kb) \in \mathsf{S} \subseteq F_{k,f}$. By  definition we have
$0<f_{2}kb-f_{1}ka \le k$ and therefore $0<f_{2}b-f_{1}a \le 1$. So
$f_{2}b-f_{1}a =1$. This means that $(ka,kb)$ cannot lie in the interior
of $F_{k,f}$. If $f_{2}=k$, this is  already a contradiction. If on the other
hand $f_{2} \ne k$, the point $(ka,kb)$ must be a vertex of $\mathsf{S}$
which is again a contradiction.
\end{proof}

\begin{definition}
Let $k \in \mathbb{Z}_{\ge 1}$. A $k$-empty lattice triangle in
standard form which is not contained in a $k$-th Farey strip is called
\emph{sporadic}.
\end{definition}

\begin{proposition}\label{finiteness}
Let $k \in \mathbb{Z}_{\ge 1}$. The number of minimal sporadic $k$-empty
lattice triangles in standard form is finite. The bound for the first 
coordinate of the vertex in $\Delta$ is $(k^{2}-1)k-1$.
\end{proposition}

\begin{proof}
Let $(x,y) \in \mathbb{Z}^{2}$ such that
$\mathrm{conv}((0,0,),(0,1),(x,y))$ is a $k$-empty triangle in normal 
form which is not completely contained in
a $k$-th Farey strip. Then there is a $k$-th Farey strip $F_{k,f}$ and a 
spike $\mathsf{S}$ attached to it, such that $(x,y)$ is in the interior 
of it.

We have $f_{2} \ne k$. Otherwise $\mathsf{S}$ is empty. For some
$i>k-f_{2}$ the spike has the vertices
\begin{align*}
  \biggl( ik,\frac{f_{1}}{f_{2}}i&k+\frac{k}{f_{2}}\biggr), \ \biggl(
(i+f_{2})k,\frac{f_{1}}{f_{2}}(i+f_{2})k+\frac{k}{f_{2}}\biggr),\\
&\biggl( \frac{(i+f_{2})ik}{i-k+f_{2}},\frac{\frac{f_{1}}{f_{2}}(i
+f_{2})ik+
i\frac{k}{f_{2}}}{i-k+f_{2}}\biggr).
\end{align*}
Its area is therefore
\begin{align*}
A(\mathsf{S}) \ := \ \frac{1}{2}\biggl( \frac{ik^{2}}{i-k+f_{2}}-k^{2}
\biggr).
\end{align*}
Let $I(\mathsf{S})$ be the number of interior integral points of
$\mathsf{S}$ and $B(\mathsf{S})$ the number of integral points on the 
boundary of $\mathsf{S}$. By Pick's
theorem we have
\[
A(\mathsf{S})=I(\mathsf{S})+\frac{B(\mathsf{S})}{2}-1.
\]
Furthermore, it is $A(\mathsf{S})<1$ if and only if
\[
i \ > \ \frac{k^{3}-f_{2}k^{2}+2k-2f_{2}}{2}.
\]
Since $B(\mathsf{S}) \ge 1$, if $A(\mathsf{S})<1$ then 
$I(\mathsf{S})<1$. So $x$ is bounded from above by $(k^{2}-1)k-1$.
\end{proof}

\begin{corollary}
Let $k \in \mathbb{Z}_{\ge 1}$. The number of $k$-empty lattice 
triangles in standard form which are not contained in a $k$-th Farey 
strip is finite.
\end{corollary}

\begin{proof}
Let $\mathsf{S}=\Delta(a,x,y)$ be a $k$-empty lattice triangle in 
standard form.
Then, there is a minimal $k$-empty lattice triangle $\mathsf{S}'=
\Delta(1,x,y)$ contained in $\mathsf{S}$.
\end{proof}

\begin{remark}
The proof of Proposition \ref{finiteness} shows that we can list the 
sporadic minimal $k$-empty lattice triangles in standard form explicitly 
for a given $k \in \mathbb{Z}_{\ge 1}$. The following table lists the 
number of those simplices for low values of $k$.

\begin{center}
\begin{tabular}{c||c|c|c|c|c|c}

$k$ & $1$ & $2$ & $3$ & $4$ & $5$ & $6$ \\ \hline
\# & 0 & 2 & 7 & 32 & 96 & 279  \\

\end{tabular}
\end{center}
\end{remark}

\subsection{Canonical polytopes}

\begin{definition}
Let $d \in \ZZ_{\geq 2}, r \in \ZZ_{\geq d}$ and $v_0:=0,v_1,\ldots,v_r \in \QQ^d$ be pairwise different so that the polytope 
$
\mathsf{P} := \conv(v_0,\ldots,v_r)
$
is full-dimensional and the $v_i$ are the vertices of $\mathsf{P}$. 
We say that $\mathsf{P}$ is \emph{$\QQ$-Gorenstein of index $\imath \in \ZZ_{\geq 1}$}
if there exist $\alpha_1,\ldots,\alpha_d \in \ZZ$ with
$$
\sum\limits_{j=1}^{d} \alpha_j v_{ij} = \imath
$$
for all $i=1,\ldots,r$ and $\imath$ is minimal with that property.

For a  $\QQ$-Gorenstein  polytope
$
\mathsf{P} = \conv(v_0,\ldots,v_r)
$
of index $\imath$ and $\alpha_1,\ldots,\alpha_d \in \ZZ$ with $
\sum_{j=1}^{d} \alpha_j v_{ij} = \imath
$, 
we say that $\mathsf{P}$ is \emph{canonical} (\emph{terminal})
if for each lattice point $0 \neq p \in \mathsf{P}$ we have
$\sum_{j=1}^{d} \alpha_j p_j = \imath$ 
(if $\mathsf{P} \cap \ZZ^d \subseteq \{v_0,\ldots,v_r\}$).
\end{definition}

\begin{remark}
Let $v_1,\ldots,v_r \in \ZZ^d$ be primitive and $X$ be the affine toric variety defined by the full-dimensional cone
$
\cone(v_1,\ldots,v_r).
$
Then $X$ is $\QQ$-Gorenstein if and only if the polytope $\mathsf{P} :=\conv(0,v_1,\ldots,v_r)$ is $\QQ$-Gorenstein and in this case it is canonical (terminal) if and only if $\mathsf{P}$ is canonical (terminal).
\end{remark}

It is well known that the only canonical lattice polygons are those of index $1$. From this, we can deduce the following statement for canonical polytopes with rational vertices.

\begin{lemma}
\label{le:twopoints}
Let $\mathsf{P} \subseteq \QQ^2$ be $\QQ$-Gorenstein of index $\imath \geq 2$. Then if $\mathsf{P}$ contains two lattice points different from the origin, it cannot be canonical.
\end{lemma} 

\begin{proof}
Let
$
\mathsf{P} := \conv(v_0,\ldots,v_r)
$
contain the two lattice points $p_1,p_2$ and choose
$\alpha_1,\ldots,\alpha_d \in \ZZ$ satisfying
$
\sum_{j=1}^{d} \alpha_j v_{ij} = \imath
$.
Then if $
\sum \alpha_j p_{ij} = \imath
$ for $i=1,2$, the polygon $\mathsf{P}$ contains the lattice polygon $\mathsf{P}'=\conv(0,p_1,p_2)$ of index $\imath$. But $\mathsf{P}'$ and thus also $\mathsf{P}$ cannot be canonical.
If in turn $
\sum_{j=1}^{d} \alpha_j p_{ij} < \imath
$ for some $i \in \{1,2\}$, then $\mathsf{P}$ is not canonical by definition.
\end{proof}

In dimension three, we have a classification of the canonical lattice polytopes by Ishida and Iwashita, which we recall in the following theorem:

\begin{theorem}\cite[Theorem 3.6]{II}.\label{th:toric3dim}
Let $\mathsf{P} \subseteq \QQ^3$ be a $\QQ$-Gorenstein canonical lattice polytope of index $\imath$ with vertices $0,v_1,\ldots,v_r$. Then, up to lattice equivalence, the polytope $\mathsf{P}$ belongs to one of the following cases.

\renewcommand{\arraystretch}{1.5} 

\begin{longtable}{c|c|c|c|c}
Case 
& 
$r$ 
& 
$\imath$ 
& 
$v_1,\ldots,v_r$ 
&
Conditions
\\
\hline
$(i)$ 
&
any
& 
$1$ 
&
any
&

\\
\hline
$(ii)$ 
&
$4$
& 
$
2 
$
&
$(1,0,2),(1+m,1,2),(1,2,2),(1-n,1,2)$ 
&
$
n,m \in \ZZ_{\geq 1}
$
\\
\hline
$(iii)$ 
&
$3$ 
& 
$
\geq 2
$  
&
$(1,n,\imath),(0,n,\imath),(1,n+m,\imath)$
&
$
n,m \in \ZZ_{\geq 1}
$
\\
\hline
$(iv)$ 
&
$3$ 
& 
$
2
$  
&
$(1,1,2),(0,1,2),(2,1+2m,2)$
&
$
m \in \ZZ_{\geq 2}
$
\\
\hline
$(v)$ 
&
$3$ 
& 
$
2
$  
&
$(1,-2,2),(-1,-1,2),(2,1,2)$
\\
\hline
$(vi)$ 
&
$3$ 
& 
$
3
$
&
$(1,-1,3),(0,-1,3),(2,2,3)$
\end{longtable}

\noindent
Additionally, $\gcd(n,\imath)=1$ holds in Case $(iii)$ and $\mathsf{P}$ is terminal if and only if it falls under Case $(iii)$ with $m=1$.
\end{theorem}

There is another very useful assertion from~\cite{II}:

\begin{corollary}\cite[Lemma 2.2]{II}.\label{corr:quadrangle}
Let $a,b,\imath \in \ZZ$. The three-dimensional polytope
$$
\mathsf{P} :=\conv(0,(a,b,\imath),(a+1,b,\imath),(a,b+1,\imath),(a+1,b+1,\imath))
$$ 
is only canonical for $\imath=1$.
\end{corollary}

Using the data from Theorem~\ref{th:toric3dim}, we compute the Class group and Cox ring of the toric threefold singularities in Subsection~\ref{subs:toric} of Section~\ref{sec:class}. But for now, we go one step further and have a look at certain canonical polytopes that have not only integer but also rational vertices. 

\begin{lemma}\label{le:ncone}
The rational polygon $\conv\left((0,0),(k,\imath),(k+1/q,\imath) \right)$ with $k \in \ZZ$, $\imath,q \in \ZZ_{>1}$  is canonical  if and only if there is an integral $0<c<q$ with
$$
kc \equiv - 1 \mod \imath.
$$
\end{lemma}

\begin{proof}
Let $\mathsf{P} :=\conv\left((0,0),(k,\imath),(k+1/q,\imath) \right)$ with $k\in\mathbb{Z}$ and $2\leq q$. If $\gcd(k,\imath)\neq 1$, then $\mathsf{P}$ cannot be canonical. Thus we have $\gcd(k,\imath)=1$.
The polygon $\mathsf{P}$ contains no lattice points but $(0,0)$ and $(k,\imath)$ if and only if the lattice polygon
$$
\mathsf{P}_1:=\conv\left((0,0),(qk,q\imath),(qk+1,q\imath) \right)
$$
contains no $q$-fold points but $(0,0)$ and $(qk,q\imath)$. But $\mathsf{P}_1$ is $q$-equivalent to the polygon
$$
\mathsf{P}_2:=\conv\left((0,0),(0,q),(\imath,q-c) \right)
$$
with integer $0<q-c<\imath$ and $kc \equiv \imath-1 \mod \imath$. The $q$-affine unimodular transformation is given by the matrix
$$
\begin{pmatrix}
\imath & -k \\
-c & d
\end{pmatrix}
\ \in \ \GL_2(\ZZ),
$$
where appropriate $c$ and $d$ exist since $k$ and $\imath$ are coprime. 

Now if $q-\imath < -c \leq 0$, then the $q$-fold point $(q,q)$ lies inside $\mathsf{P}_2$ while for $0<c<q$, no $q$-fold point lies inside $\mathsf{P}_2$. The assertion follows. 
\end{proof}

\begin{corollary}
\label{corr:halfcone}
The rational polygon $\conv\left((0,0),(k,\imath),(k+1/2,\imath) \right)$ with $k \in \ZZ$, $\imath \in \ZZ_{>1}$ is canonical  if and only if $k+1$ is a multiple of $\imath$, i.e. it is lattice equivalent to
$
\conv\left((0,0),(1/2,\imath),(1,\imath) \right)
$.
\end{corollary}

\begin{corollary}
\label{corr:halfhalfcone}
The rational polygon $\conv\left((0,0),(k-1/2,\imath),(k+1/2,\imath) \right)$ with $k \in \ZZ$, $\imath \in \ZZ_{>1}$ is canonical  if and only if $\imath=2$ and $k$ is odd.
\end{corollary}

\begin{corollary}
\label{corr:13cone}
The rational polygon $\conv\left((0,0),(k,\imath),(k+1/3,\imath) \right)$ with $k \in \ZZ$, $\imath \in \ZZ_{>1}$ is canonical  if and only if either $k+1$ is a multiple of $\imath$ or $\imath$ is odd and $k \equiv \frac{\imath-1}{2} \mod \imath$. 
\end{corollary}

\begin{corollary}
\label{corr:1313cone}
The rational polygon $\conv\left((0,0),(k-1/3,\imath),(k+1/3,\imath) \right)$ with $k \in \ZZ$, $\imath \in \ZZ_{>1}$ is canonical  if and only if $\imath=2$ and $k$ is odd or $\imath=3$ and $k \equiv 1,2 \mod 3$.
\end{corollary}

\begin{remark} 
\label{rem:2513polytope}
We compare polytopes $\mathsf{Q}:=\conv\left((0,0),(k,\imath),(k+2/5,\imath) \right)$ with the polytopes $\conv\left((0,0),(k,\imath),(k+1/3,\imath) \right)$ from Corollary~\ref{corr:13cone}. The polytope $\mathsf{Q}$ is canonical if and only if the polytope 
$$
\conv((0,0),(0,5),(2\imath,5-2c))
$$
with $0<5-2c<2\imath$ and $kc \equiv -1 \mod \imath$ is $5$-empty. More specifically, if it contains no $5$-fold points but $(0,0)$ and $(0,5)$. Hence, it cannot contain the $5$-fold point $(5,5)$ and so $(2\imath,5-2c)$ cannot lie in the dark gray area, but must be in the light gray area. The points of the form $(2\imath,5-2c)$ and $(0,0)$, $(0,5)$ are marked with circles. The ones which violate the condition that $c$ and $\imath$ are coprime are crossed out.

\begin{center}
\begin{tikzpicture}[scale=0.5]
\fill[fill=gray!90!, fill opacity=0.50] (5,5)--(10,10)--(16,10)--(16,5)--cycle;
\fill[fill=gray!50!, fill opacity=0.50] (0,0)--(5,5)--(16,5)--(16,0)--cycle;
\draw[gray, very thin] (-0.5,-0.5) grid (16.5,10.5);
\fill (5,5) circle (0.2);

\fill[fill=white] (0,0) circle (0.2);
\draw (0,0) circle (0.2);

\fill[fill=white] (0,5) circle (0.2);
\draw (0,5) circle (0.2);

\fill[fill=white] (4,1) circle (0.2);
\draw (4,1) circle (0.2);
\draw[line width=0.5pt] (3.8,1.2)--(4.2,0.8);
\draw[line width=0.5pt] (4.2,1.2)--(3.8,0.8);

\fill[fill=white] (6,1) circle (0.2);
\draw (6,1) circle (0.2);

\fill[fill=white] (8,1) circle (0.2);
\draw (8,1) circle (0.2);
\draw[line width=0.5pt] (7.8,1.2)--(8.2,0.8);
\draw[line width=0.5pt] (8.2,1.2)--(7.8,0.8);

\fill[fill=white] (10,1) circle (0.2);
\draw (10,1) circle (0.2);

\fill[fill=white] (12,1) circle (0.2);
\draw (12,1) circle (0.2);
\draw[line width=0.5pt] (11.8,1.2)--(12.2,0.8);
\draw[line width=0.5pt] (12.2,1.2)--(11.8,0.8);

\fill[fill=white] (14,1) circle (0.2);
\draw (14,1) circle (0.2);

\fill[fill=white] (16,1) circle (0.2);
\draw (16,1) circle (0.2);
\draw[line width=0.5pt] (15.8,1.2)--(16.2,0.8);
\draw[line width=0.5pt] (16.2,1.2)--(15.8,0.8);

\fill[fill=white] (4,3) circle (0.2);
\draw (4,3) circle (0.2);

\fill[fill=white] (6,3) circle (0.2);
\draw (6,3) circle (0.2);

\fill[fill=white] (8,3) circle (0.2);
\draw (8,3) circle (0.2);

\fill[fill=white] (10,3) circle (0.2);
\draw (10,3) circle (0.2);

\fill[fill=white] (12,3) circle (0.2);
\draw (12,3) circle (0.2);

\fill[fill=white] (14,3) circle (0.2);
\draw (14,3) circle (0.2);

\fill[fill=white] (16,3) circle (0.2);
\draw (16,3) circle (0.2);

\end{tikzpicture}

\end{center}

Altogether, we see that $c=1$ or $c=2$, if $\imath$ is odd. Thus $\mathsf{Q}$ is canonical if and only if $\conv\left((0,0),(k,\imath),(k+1/3,\imath) \right)$ is canonical. 
\end{remark}

\begin{lemma}
\label{le:1245cone}
The rational polygon $\conv\left((0,0),(k+1/2,\imath),(k+4/5,\imath) \right)$ with $k \in \ZZ$, $\imath \in \ZZ_{>1}$ contains no lattice points but $(0,0)$ if and only if it is lattice equivalent to one of the polytopes
$$
\conv\left((0,0),(-1/2,\imath),(-1/5,\imath) \right),
$$
$$
\conv\left((0,0),(1/2,\imath),(4/5,\imath) \right),
$$
$$
\conv\left((0,0),\left(\frac{\imath}{2}-\frac{1}{2},\imath\right),\left(\frac{\imath}{2}-\frac{1}{5},\imath\right) \right) 
\qquad \text{only if}~~ 2 | \imath.
$$
\end{lemma}

\begin{proof}
Let $\mathsf{P}:=\conv\left((0,0),(k+1/2,\imath),(k+4/5,\imath) \right)$ with $k\in\mathbb{Z}$ and $\imath\in\mathbb{Z}_{>1}$. The polygon $\mathsf{P}$ contains no lattice points but $(0,0)$ if and only if the lattice polygon
$$
\mathsf{P}_1:=\conv\left((0,0),(10k+5,10\imath),(10k+8,10\imath) \right)
$$
contains no $10$-fold points but $(0,0)$. Consider the $10$-affine unimodular transformation given by
$$
\begin{pmatrix}
-2\imath & 2k+1 \\
c & d
\end{pmatrix}
\ \in \ \GL_2(\ZZ)
$$
which yields a $10$-equivalent lattice polygon
$$
\mathsf{P}_2:=\conv\left((0,0),(0,g),(6\imath,g+3c) \right)
$$
where $g:=c\cdot(10k+5)+d\cdot10\imath=\mathrm{gcd}(10k+5,10\imath)\ge5$. Note that $g\le9$. Since $g$ can't be $6,7,8$ or $9$ we have $g=5$. We also assume that $0<5+3c<6\imath$. The lattice polygon $\mathsf{P}_{2}$ contains no $10$-fold points but $(0,0)$ if and only if one of the following cases occurs.

\noindent \emph{Case $1$: The lattice point $(6\imath,5+3c)$ lies in the relative interior of the Farey strip $F_{10,\frac{1}{2}}$.} Then we have $5+3c=3\imath+2$, so $c=\imath-1$.

\noindent \emph{Case $2$: The lattice point $(6\imath,5+3c)$ lies in the relative interior of the Farey strip $F_{10,0}$.} In this case we  have $c=\pm1$.

\noindent \emph{Case $3$: The lattice point $(6\imath,5+3c)$ lies in the relative interior of a spike attached to $F_{10,0}$.} The spikes are of the form
$$
\conv\left(\left(10m,10\right),\left(10(m+1),10\right),\left(\frac{10m(m+1)}{m-1},\frac{10m}{m-1}\right) \right)
$$
where $2\le m\le11$. This leaves the possibilities $(\imath,c)\in\left\{(5,2),(7,2),(9,2)\right\}$. But those cannot occur since $c\equiv\pm1\mod 2\imath$.
\end{proof}

\begin{lemma}
\label{le:halfheight}
The rational polygon $\conv\left((0,0),(i,k+1/2),(j,k+1/2) \right)$ with $i\neq j \in \ZZ$, $k \in \ZZ_{\geq 0}$ contains no lattice points but $(0,0)$ if and only if $k=0$.
\end{lemma}

\begin{proof}
Let $\mathsf{P}:=\conv\left((0,0),(i,k+1/2),(j,k+1/2) \right)$ with $i,j\in\mathbb{Z}, k\in\mathbb{Z}_{\ge0}$ and without loss of generality $i<j$. The polygon $\mathsf{P}$ contains no lattice points but $(0,0)$ if and only if the lattice polygon
$$
\mathsf{P}_1:=\conv\left((0,0),(2i,2k+1),(2j,2k+1) \right)
$$
contains no $2$-fold points but $(0,0)$. Consider the $2$-affine unimodular transformation given by
$$
\begin{pmatrix}
-(2k+1) & 2i \\
c & d
\end{pmatrix}
\ \in \ \GL_2(\ZZ)
$$
which yields a $2$-equivalent lattice polygon
$$
\mathsf{P}_2:=\conv\left((0,0),(0,g),(2\cdot(2k+1)(j-i),g+c(j-i)) \right)
$$
where $g:=c\cdot2i+d\cdot(2k+1)=\mathrm{gcd}(2i,2k+1)\ge1$. Since there are no $2$-fold points but $(0,0)$, we have $g=1$. Moreover, we can assume that $0<1+c(j-i)<2\cdot(2k+1)(j-i)$. Now, the lattice polygon $\mathsf{P}_{2}$ contains no $2$-fold points but $(0,0)$ if and only if one of the following cases occurs.

\noindent \emph{Case $1$: The lattice point $(2\cdot(2k+1)(j-i),1+c(j-i))$ lies in the relative interior of the Farey strip $F_{2,\frac{1}{2}}$.} Then the lattice point lies on the line $y=\frac{1}{2}x+1$ and we have $2k+1=c$. Since the matrix above has determinant $\pm1$, we get $2k+1=\pm1$ and so $k=0$.

\noindent \emph{Case $2$: The lattice point $(2\cdot(2k+1)(j-i),1+c(j-i))$ lies in the relative interior of the Farey strip $F_{2,0}$.} Then we have $c=0$ and by considering the determinant again we see $k=0$.

\noindent \emph{Case $3$: The lattice point $(2\cdot(2k+1)(j-i),1+c(j-i))$ lies in the relative interior of a spike attached to $F_{2,0}$.} The spikes are of the form
$$
\conv\left(\left(2m,2\right),\left(2(m+1),2\right),\left(\frac{2m(m+1)}{m-1},\frac{2m}{m-1}\right) \right)
$$
where $2\le m\le4$. These interiors do not contain any lattice points, so there are no additional possibilities.
\end{proof}

\section{Canonical threefold singularities with two-torus action}
\label{sec:class}

\subsection{Toric singularities}
\label{subs:toric}

This subsection deals with the \emph{toric} canonical threefold singularities. From the data provided by Theorem~\ref{th:toric3dim}, we compute the class groups and the Cox rings.

\begin{proposition}
\label{prop:ClXQtoric}
If $X$ is a toric canonical threefold singularity, then it is either - Case $(i)$ - Gorenstein or belongs to one of the following other cases with respective total coordinate space, class group and grading:
\renewcommand{\arraystretch}{1.5} 

\begin{longtable}{c|c|c|c}
Case 
& 
$
\overline{X}
$ 
&
$\Cl(X)$ 
& 
$Q$ 
\\
\hline
$(ii)$ 
&
$\CC^4$
& 
$
\begin{array}{c}
\ZZ \times \ZZ/ 2\mathfrak{d}\ZZ \\
\mathfrak{d}=\gcd(2m,m+n)
\end{array}
$
&
$
\begin{array}{c}
\begin{bmatrix}

-\frac{m+n}{\mathfrak{d}} & \frac{2n}{\mathfrak{d}} & -\frac{m+n}{\mathfrak{d}} & \frac{2m}{\mathfrak{d}} \\
\overline{\alpha_1+\mathfrak{d}} & \overline{-(2\alpha_1+\alpha_2)} & \overline{\alpha_1} & \overline{\alpha_2} 
\end{bmatrix} \\
\mathrm{~with~}
2m\alpha_1+(m+n)\alpha_2=\mathfrak{d}
\end{array}
$ 
\\
\hline
$(iii)$ 
&
$\CC^3$ 
& 
$
\ZZ/\imath m \ZZ
$  
&
$
\begin{bmatrix}
\overline{1} & \overline{m\alpha_1} & \overline{-1}
\end{bmatrix} \mathrm{~with~}
n\alpha_1 \equiv 1 \mod \imath
$ 
\\
\hline
$(iv)$ 
&
$\CC^3$ 
& 
$
\ZZ/4m\ZZ
$  
&
$[\overline{2},\overline{2m-1},\overline{-1}]$
\\
\hline
$(v)$ 
&
$\CC^3$ 
& 
$
\ZZ/10\ZZ
$  
&
$[\overline{1},\overline{1},\overline{3}]$
\\
\hline
$(vi)$ 
&
$\CC^3$ 
& 
$
\ZZ/9\ZZ
$
&
$[\overline{1},\overline{4},\overline{7}]$
\end{longtable}

\noindent
The singularity $X$ is terminal if and only if it falls under Case $(iii)$ with $m=1$.
\end{proposition}

\begin{proof}
By writing the vertices of the canonical polytopes from Theorem~\ref{th:toric3dim} as columns, we get a matrix $P$ representing the canonical toric threefold singularities $X$ as affine threefolds with a two-torus action. We only observe the Cases $(ii)$ to $(vi)$, since Case $(i)$ consists of all Gorenstein toric threefold singularities. In any case, we determine unimodular matrices $V$ and $W$, so that $S := V \cdot P^* \cdot W$ is in Smith normal form. Then if $\beta_1,\ldots,\beta_\nu$ are the elementary divisors and $\beta$ the number of zero rows of $S$, we have that
$$
\Cl(X) \cong \ZZ^\beta \oplus \ZZ/\beta_1\ZZ \oplus \ldots \oplus \ZZ/\beta_\nu \ZZ.
$$
Furthermore, the degree matrix $Q$ is the stack of the last $\nu+\beta$ rows of $V$. We go through the  Cases $(ii)$ to $(vi)$ in the following.
For Case $(ii)$, let $\mathfrak{d}:=\gcd(2m,m+n)$ and $\alpha_1,\alpha_2 \in \ZZ$ with $2m\alpha_1+(m+n)\alpha_2=\mathfrak{d}$. Then
$$
{\tiny
S=
\begin{bmatrix}
1 & 0 & 0 & 0 \\
-(m+1) & 1 & 0 & 0 \\
\alpha_1+\mathfrak{d} & -(2\alpha_1+\alpha_2) & \alpha_1 & \alpha_2 \\
-\frac{(m+n)}{\mathfrak{d}} & \frac{2n}{\mathfrak{d}} & -\frac{(m+n)}{\mathfrak{d}} & \frac{2m}{\mathfrak{d}}
\end{bmatrix}
\cdot P^* 
\cdot
\begin{bmatrix}
1 & 0 & -2 \\
0 & 1 & 2m \\
0 & 0 & 1
\end{bmatrix}
 = 
 \begin{bmatrix}
1 & 0 & 0 \\
0 & 1 & 0 \\
0 & 0 & 2\mathfrak{d} \\
0 & 0 & 0 
\end{bmatrix}.}
$$
So $\Cl(X)=\ZZ/2\mathfrak{d}\ZZ$ and
$$
{\tiny
Q =
\begin{bmatrix}
\overline{\alpha_1+\mathfrak{d}} & \overline{-(2\alpha_1+\alpha_2)} & \overline{\alpha_1} & \overline{\alpha_2} \\
-\frac{m+n}{\mathfrak{d}} & \frac{2n}{\mathfrak{d}} & -\frac{m+n}{\mathfrak{d}} & \frac{2m}{\mathfrak{d}}
\end{bmatrix}.}
$$
In Case $(iii)$, with $\alpha_1,\alpha_2$ meeting $\alpha_1 n+ \alpha_2 m =1$, the equality $S=V \cdot P^* \cdot W$ becomes
$$
{\tiny
\begin{bmatrix}
1 & 0 & 0 \\
0 & 1 & 0 \\
0 & 0 & m\imath
\end{bmatrix}
=
\begin{bmatrix}
1 & -1 & 0 \\
0 & 1 & 0 \\
-1 & -m\alpha_1  & 1
\end{bmatrix}
\cdot\begin{bmatrix}
1 & n & \imath \\
0 & n & \imath \\
1 & n+m & \imath
\end{bmatrix}
\cdot
\begin{bmatrix}
1 & 0 & 0 \\
0 & \alpha_1 & \imath \\
0 & \alpha_2 & -n
\end{bmatrix},}
$$
leading to the respective class group and degree matrix. The same equation has the form
$$
{\tiny
\begin{bmatrix}
1 & 0 & 0 \\
0 & 1 & 0 \\
0 & 0 & 4m
\end{bmatrix}
=
\begin{bmatrix}
1 & -1 & 0 \\
0 & 1 & 0 \\
-2 & 1-2m  & 1
\end{bmatrix}
\cdot\begin{bmatrix}
1 & 1 & 2 \\
0 & 1 & 2 \\
2 & 1+2m & 2
\end{bmatrix}
\cdot
\begin{bmatrix}
1 & 0 & 0 \\
0 & 1 & -2 \\
0 & 0 & 1
\end{bmatrix}}
$$
in Case $(iv)$, while in the cases $(v)$ and $(vi)$, the data $\Cl(X)$ and $Q$ can easily be computed using the Maple package~\cite{HaKe}.
\end{proof}

\subsection{Matrices $P$ for non-toric singularities}
In the following, we classify the matrices $P$ defining non-toric canonical threefold singularities of complexity one. The classification is divided into the subcases of Proposition~\ref{prop:zetaexcep}~(iii). 
We begin with those singularities of canonical multiplicity one:

\begin{proposition}
\label{prop:zeta=1}
Let $X$ be a non-toric threefold singularity of complexity one.
Assume that~$X$ is of canonical multiplicity 
one, Gorenstein index $\imath \geq 2$ and is at most canonical. 
Then $X \cong X(P_i)$,
where $P_i$ is one of the following matrices:
\setlength{\arraycolsep}{4pt}
 $$
 {\tiny
 P_1=
\begin{bmatrix}
-3 & -1 & 3 & 0  \\
-3 & -1 & 0 & 2  \\
1 & 0  & 0 & -1 \\
-4 & 0 & 2 & 2 
\end{bmatrix}
\quad
P_2=
\begin{bmatrix}
-3 & -1 & 3 & 1 & 0  \\
-3 & -1 & 0 & 0 & 2  \\
1 & 0  & 0 & -1 & -1 \\
-4 & 0 & 2 & 2 & 2 
\end{bmatrix}
\quad
P_3=
\begin{bmatrix}
-3 & 2 & 1 & 0  \\
-3 & 0 & 0 & 2  \\
2 & 0 & 2 &  1 \\
-6 & 3 & 3 & 3 
\end{bmatrix}}
$$
$$
{\tiny
P_4 =
\begin{bmatrix}
-3 & 2  & 0 & 0 
\\
-3 & 0 & 1 & 1 
\\
0 & 1 & 0 & 1 
\\
-4 & 2 & 2 & 2 
\end{bmatrix}
\quad
P_5
=
\begin{bmatrix}
-k & 1 & 1  & 0 & 0 
\\
-k & 0 & 0 & 1 & 1 

\\
1 & 0 & 1 & 0 & 1 
\\
2-2k & 2 & 2 & 2 & 2 
\end{bmatrix}}
$$
\end{proposition}

\begin{proof}
Let $P$ be a matrix so that $X(P)$ is a canonical threefold with $\zeta_X=1$ admitting a two-torus action.
We may assume that $P$ is in the normal form of Proposition~\ref{prop:zetaexcep}. 
Since $l_{ij}=1$ holds for $i\geq 3$, we have $n_i\geq 2$ for $i\geq 3$. 
We now go through all possible different leading platonic tuples. 
Let $e_1,\ldots,e_{r+2}$ be the standard basis of the column space of $P$. 
Set $e_0:=-e_1-\ldots-e_r$. 

Let in the following $\tau_1$ be the (fake or $P$-) elementary cone generated by the columns of the leading block and $\tau_i$ for $i\geq 2$ denote other (fake or $P$-) elementary cones. 
For such cones, recall that $v(\tau_i)'=\partial A_X^c (\lambda) \cap \tau_i$ denotes the point 
where $\lambda \cap \tau_i$ leaves $A_X^c$.

\bigskip

\noindent
\emph{Case 1: leading platonic tuple $(5,3,2)$}.
The matrix $P$ has leading block 
$$
{\tiny
\begin{bmatrix}
-5 & 3  & 0 & 0 & \dots & 0 
\\
-5 & 0 & 2 & 0 & \dots & 0 
\\
-5 & 0 & 0 & 1 & \dots & 0 
\\
\vdots & \vdots & \vdots & \vdots & \ddots & \vdots 
\\
-5 & 0 & 0 &  0 &  \dots &  1 
\\
d_{011} & d_{111} & d_{211} & 0 & \ldots & 0 
\\
6\imath -5\imath r & \imath & \imath & \imath & \ldots & {\imath} 
\end{bmatrix}
}
$$
due to Proposition~\ref{prop:zetaexcep}. 
If $r\geq 3$, due to irredundancy we have at least another column $v_{32}=e_{3}+d_{321}e_{r+1}+\imath e_{r+2}$ with $d_{321}\neq 0$. We thus have an elementary cone $\tau_2$ generated by $v_{32}$ and all columns of the leading block but $v_{31}$.
So $A_X^c(\lambda)$ contains the integer points
\begin{align*}
v(\tau_1)' &= (0,\ldots,0,6d_{011}+10 d_{111}+15 d_{211},\imath), \\
v(\tau_2)' &= (0,\ldots,0,6d_{011}+10 d_{111}+15 d_{211}+30d_{321},\imath)
\end{align*}
and thus the corresponding variety can not be canonical with Lemma~\ref{le:twopoints}. 
So $r=2$ remains. We have leading block
$$
{\tiny
\begin{bmatrix}
-5 & 3 & 0  \\
-5 & 0 & 2  \\
d_{011} & d_{111} & d_{211} \\
-4\imath & \imath & \imath 
\end{bmatrix}
}
$$
here. So we get 
$
v(\tau_1)'=(0,0,6d_{011}+10d_{111}+15d_{211},\imath) \in \ZZ^4
$.
Thus $ \partial A_X^c(\lambda_i,\tau_1)$ contains the integer points
\begin{align*}
p_{0,t}
&:=(-t,-t,(6-t)d_{011}+2(5-t)d_{111}+3(5-t)d_{211},(1-t)\imath)
& t=0,1,\ldots,5
\\
p_{1,t}
&:=(t,0,2(3-t)d_{011}+(10-3t)d_{111}+5(3-t)d_{211},\imath)
& t=0,1,2,3
\\
p_{2,t}
&:=(0,t,3(2-t)d_{011}+5(2-t)d_{111}+(15-7t)d_{211},\imath)
& t=0,1,2
\end{align*}
These points are not allowed to be columns of $P$, as they would lie inside $\tau_1$. 
Any column of $P$ not contained in the leading block must now be of the form $v_{i2}=p_{i,t}+se_3$ for $0 \neq s \in \ZZ$ and $p_{i,t}$ one of the points above. 
Let now $\tau_2$ be the elementary cone generated by this additional column and the two leading block columns from the other two leaves. 
Depending on $i=0,1,2$, we have $v(\tau_2)'= v(\tau_1)'+\delta_i e_3$, with
$$
\delta_0 = 6s/(6-t),
\quad
\delta_1 = 10s/(10-3t),
\quad
\delta_2 = 15s/(15-7t).
$$
Now since $|ks/(k-x)| \geq 1$ for $0 \neq s \in \ZZ$ and $0 \leq x  <k$, 
between the integer point $v(\tau_1)'$ and $v(\tau_2)'$ lies another integer point in $A_X^c(\lambda)$. Lemma~\ref{le:twopoints} tells us that the corresponding variety cannot be canonical.

\bigskip

\noindent
\emph{Case 2: leading platonic tuple $(4,3,2)$}.
The matrix $P$ has leading block 
$$
{\tiny
\begin{bmatrix}
-4 & 3  & 0 & 0 & \dots & 0 
\\
-4 & 0 & 2 & 0 & \dots & 0 
\\
-4 & 0 & 0 & 1 & \dots & 0 
\\
\vdots & \vdots & \vdots & \vdots & \ddots & \vdots 
\\
-4 & 0 & 0 &  0 &  \dots &  1 
\\
d_{011} & d_{111} & d_{211} & 0 & \ldots & 0 
\\
5\imath -4\imath r & \imath & \imath & \imath & \ldots & {\imath} 
\end{bmatrix}
}
$$
due to Proposition~\ref{prop:zetaexcep}. 
If $r\geq 3$, due to irredundancy we have at least another column $v_{32}=e_{3}+d_{321}e_{r+1}+\imath e_{r+2}$ with $d_{321}\neq 0$. 
So $A_X^c(\lambda)$ contains the integer points
\begin{align*}
v(\tau_1)' &= (0,\ldots,0,3d_{011}+4 d_{111}+6 d_{211},\imath), \\
v(\tau_2)' &= (0,\ldots,0,3d_{011}+4 d_{111}+6 d_{211}+12d_{321},\imath)
\end{align*}
and as in Case 1, the corresponding variety can not be canonical. 
So we have $r=2$ again and the leading block
$$
{\tiny
\begin{bmatrix}
-4 & 3 & 0  \\
-4 & 0 & 2  \\
d_{011} & d_{111} & d_{211} \\
-3\imath & \imath & \imath 
\end{bmatrix}
}
$$
with $v(\tau_1)' = (0,0,3d_{011}+4 d_{111}+6 d_{211},\imath)\in \ZZ^4$. 
As in Case 1, any column of $P$ not contained in the leading block is of the form $v_{i2}=p_{i,t}+se_3$, now with
\begin{align*}
p_{0,t}
&:=(-t,-t,(3-t/2)d_{011}+(4-t)d_{111}+(6-3t/2)d_{211},(1-t)\imath)
& t=0,1,\ldots,4 
\\
p_{1,t}
&:=(t,0,(3-t)d_{011}+(4-t)d_{111}+2(3-t)d_{211},\imath)
& t=0,1,2,3
\\
p_{2,t}
&:=(0,t,(3-3t/2)d_{011}+2(2-t)d_{111}+(6-5t/2)d_{211},\imath)
& t=0,1,2
\end{align*}
 For $i=0,1,2$, we have $v(\tau_2)'= v(\tau_1)'+\delta_i e_3$, with
$$
\delta_0 = 6s/(6-t),
\quad
\delta_1 =  8s/(8-2t),
\quad
\delta_2 = 12s/(12-5t).
$$
The third coordinate of the points $p_{i,t}$ is in $\ZZ$ for $i=1$ and also for $i=0,2$ if additionally $t \in 2\ZZ$. 
It is as well in $\ZZ$ if $d_{011} - d_{211} \in 2\ZZ$. 
In these cases we use $|ks/(k-x)| \geq 1$  for $0 \leq x  <k$ to obtain another (in addition to $v(\tau_1)'$) integer point in $A_X^c(\lambda)$. 
In the remaining cases, the third coordinate of the points $p_{i,t}$ is in $\ZZ + 1/2$ and we have $s \in \ZZ + 1/2$.
We first examine $i=0$ and $t=3$, here the distance $6s/(6-t)$ between $v(\tau_1)'$ and $v(\tau_2)'$ becomes $2s$, which means that again there must be an integer point inbetween.

Only $t=1$ and $i=0,2$ remain. Since $d_{011} - d_{211} \in 2\ZZ + 1$, the Gorenstein index $\imath$ must be odd, otherwise the first or the third column of the leading block would not be primitive. 
By admissible operations, we achieve $d_{011}=1$, $d_{211}=0$.

We first examine the case $i=0$. This leads to 
$
v(\tau_1)' = (0,0,3+4d_{111},\imath)$
and
$
v(\tau_2)' = (0,0,3+4d_{111}+6s/5,\imath)$. 
Only for $s=\pm 1/2$, there is no integer point inbetween. 
In these two cases, the polytope
$$
\conv(0_{\ZZ_4},(0,0,3+4d_{111},\imath),(0,0,3+4d_{111}\pm 1/2,\imath)
$$ 
is contained in $A_X^c(\lambda)$. 
In the first case $s=1/2$, Corollary~\ref{corr:halfcone} tells us that $4+4d_{111}$ is a multiple of $\imath$. Our leading block together with the column $v_{02}$ has the form
$$
{\tiny
\begin{bmatrix}
-4 & -1 & 3 & 0  \\
-4 & -1 & 0 & 2  \\
1 &  3d_{111}+2    & d_{111} & 0 \\
-3\imath & 0 & \imath & \imath 
\end{bmatrix}.
}
$$
By the admissible operations of adding the $(1+d_{111})$-fold of the first, the $(2+2d_{111})$-fold of the second and the $-4(1+d_{111})/\imath$-fold of the fourth all to the third row, we obtain 
$$
{\tiny
\begin{bmatrix}
-4 & -1 & 3 & 0  \\
-4 & -1 & 0 & 2  \\
1 &  -1 & -1 & 0 \\
-3\imath & 0 & \imath & \imath 
\end{bmatrix}.
}
$$
But $A_X^c(\lambda_0)$ contains the point $(-1,-1,0,-1)$ and thus the corresponding variety can not be canonical. 
In the second case $s=-1/2$,  Corollary~\ref{corr:halfcone} tells us that $2+4d_{111}$ is a multiple of $\imath$. 
By the admissible operations of adding the $(1+d_{111})$-fold of the first, the $(1+2d_{111})$-fold of the second and the $-2(1+2d_{111})/\imath$-fold of the fourth all to the third row, we obtain the exact same matrix as above and are done with this case as well.
The final case $i=2$ is analogue to the case $i=0$.

\bigskip

\noindent
\emph{Case 3: leading platonic tuple $(3,3,2)$}.
The matrix $P$ has leading block 
$$
{\tiny
\begin{bmatrix}
-3 & 3  & 0 & 0 & \dots & 0 
\\
-3 & 0 & 2 & 0 & \dots & 0 
\\
-3 & 0 & 0 & 1 & \dots & 0 
\\
\vdots & \vdots & \vdots & \vdots & \ddots & \vdots 
\\
-3 & 0 & 0 &  0 &  \dots &  1 
\\
d_{011} & d_{111} & d_{211} & 0 & \ldots & 0 
\\
4\imath -3\imath r & \imath & \imath & \imath & \ldots & {\imath} 
\end{bmatrix}
}
$$
due to Proposition~\ref{prop:zetaexcep}. 
If $r\geq 3$, due to irredundancy we have at least another column $v_{32}=e_{3}+d_{321}e_{r+1}+\imath e_{r+2}$ in $P$ with $d_{321}\neq 0$. 
So $A_X^c(\lambda)$ contains the integer points
\begin{align*}
v(\tau_1)' &= (0,\ldots,0,2(d_{011}+ d_{111}) + 3 d_{211},\imath), \\
v(\tau_2)' &= (0,\ldots,0,2(d_{011}+ d_{111}) + 3 d_{211}  + 6d_{321},\imath)
\end{align*}
and with Lemma~\ref{le:twopoints}, the corresponding variety can not be canonical. 
Thus only $r=2$ is possible and the leading block is
$$
{\tiny
\begin{bmatrix}
-3 & 3 & 0  \\
-3 & 0 & 2  \\
d_{011} & d_{111} & d_{211} \\
-2\imath & \imath & \imath 
\end{bmatrix}
}
$$
with $v(\tau_1)' = (0,0,2(d_{011}+ d_{111})+3 d_{211},\imath)\in \ZZ^4$. 
The points $p_{i,t}$ in the intersection of $\partial A_X^c(\lambda_i,\tau_1)$ with the hyperplanes $\{x_i=t\}$ have the forms:
\begin{align*}
p_{0,t}
&:=(-t,-t,(2-t/3)d_{011}+(2-2t/3)d_{111}+(3-t)d_{211},(1-t)\imath)
& t=0,1,2,3
\\
p_{1,t}
&:=(t,0,(2-2t/3)d_{011}+(2-t/3)d_{111}+(3-t)d_{211},\imath)
& t=0,1,2,3
\\
p_{2,t}
&:=(0,t,(2-t)d_{011}+(2-t)d_{111}+(3-t)d_{211},\imath)
& t=0,1,2
\end{align*}
Now any column of $P$ not contained in the leading block must be of the form $v_{i2}=p_{i,t}+se_3$ with $0 \neq s \in \QQ$. 
This is because it must firstly be contained in $\partial A_X^c$, which fixes the last coordinate, secondly it is not in the leading block, restricting its first two coordinates and lastly it must not be contained in $\tau_1$, so in fact $s \neq 0$. 
Now  for  $i=0,1,2$ let $\tau_2$ be the elementary cone generated by $v_{i2}$ and $v_{j1}$ for $j \neq i$.
 For $i=0,1,2$, we have $v(\tau_2)'= v(\tau_1)'+\delta_i e_3$, with
$$
\delta_0 = 6s/(6-t),
\quad
\delta_1 =  6s/(6-t),
\quad
\delta_2 = 3s/(3-t).
$$ 
The third coordinate of the points $p_{i,t}$ is in $\ZZ$ for $i=2$ and also for $i=0,1$ if additionally $d_{011} - d_{111} \in 3\ZZ$ or $t=0,3$. 
Since $v_{i2}$ must be integer, in these cases also $s$ must be integer. 
Then we use $|ks/(k-t)| \geq 1$ for $0 \leq t  <k$ and integer $s \neq 0$ to obtain another (in addition to $v(\tau_1)'$) integer point in $A_X^c(\lambda)$. 
Lemma~\ref{le:twopoints} shows that such variety can not be canonical.

It remains $i=0,1$ with $t=1,2$ and $d_{011} - d_{111} \notin 3\ZZ$. 
Here $\imath \notin 3\ZZ$, because otherwise neither of $d_{011}$ and $d_{111}$ would be in $3\ZZ$ due to primitivity of the columns of $P$. But then $d_{011} + d_{111} \in 3\ZZ$ and $v(\tau_1)'$ would not be primitive.
The cases $i=0$ and $i=1$ are equivalent, since the leaves $\lambda_1$ and $\lambda_0$ till now are interchangeable. So we only treat $i=0$ in the following. 
Moreover, by admissible operations, we can achieve $d_{011}=0$ and $d_{111}=-1$. 

This works as follows: We have $d_{011} - d_{111} \notin 3\ZZ$. 
Thus $2d_{011} + d_{111} \pm 1 \in 3\ZZ$. 
We have that $\imath$ and $3$ are coprime. 
Thus there always exists an $a \in \ZZ$ so that adding 
$
-a + (2d_{011}+d_{111} \pm 1)/3
$
times the first, $-a+d_{011}+d_{111}\pm 1$ times the second and 
$ 
(3a-2d_{011}-2d_{111}\pm 2)/\imath 
$
the fourth to the third row is an admissible operation. 
Then we have $d_{011}=0$ and $d_{111}=\pm 1$, where a possible negation of the third row if required leads to the desired form. 
Now we examine the two remaining cases $t=1,2$ with $P$ in this form.

\medskip

\noindent
\emph{Case 3.1: $t=1$}.
In order not to have two integer points in $A_X^c(\lambda)$, we require $s=-2/3$ or $s=1/3$. 
For $s=-2/3$, Corollary~\ref{corr:halfcone} applied to the polytope
$$
\conv(0_{\ZZ_4},v(\tau_1)',v(\tau_2)')
$$ 
tells us that $3d_{211}-3$ is a multiple of $\imath$ and then adding the $(d_{211}-1)$-fold of the first and second and the $3(1-d_{211})/\imath$-fold of the last to the third row, the matrix $P$ can be brought into the form:
$$
{\tiny
\begin{bmatrix}
-3 & -1 & 3 & 0  \\
-3 & -1 & 0 & 2  \\
0 & 0  & -1 & 1 \\
-2\imath & 0 & \imath & \imath 
\end{bmatrix}.
}
$$
But the first two columns of $P$ here build up a two-dimensional polytope in $A_X^c(\lambda_0)$ that can not be canonical due to Lemma~\ref{le:twopoints}.
For $s=1/3$, the leading block together with the column $v_{02}$ has the form:
$$
{\tiny
\begin{bmatrix}
-3 & -1 & 3 & 0  \\
-3 & -1 & 0 & 2  \\
0 & 2d_{211} -1  & -1 & d_{211} \\
-2\imath & 0 & \imath & \imath 
\end{bmatrix}.
}
$$
The polytope
$
\conv(0_{\ZZ_4},(0,0,3d_{211}-2,\imath),(0,0,3d_{211}-2+2/5,\imath))
$
is contained in $A_X^c(\lambda)$. According to Remark~\ref{rem:2513polytope}, we have two possible cases:

\smallskip

\noindent
\emph{Case 3.1.1: $3d_{211}-1$ is a multiple of $\imath$}. Then the admissible operation of adding the $d_{211}$-fold of the first, the $(d_{211}-1)$-fold of the second and the $(1-3d_{211})/\imath$-fold of the last to the third row brings the matrix into the form:
 $$  P_{1,\imath}:=
  {\tiny
\begin{bmatrix}
-3 & -1 & 3 & 0  \\
-3 & -1 & 0 & 2  \\
1 & 0  & 0 & -1 \\
-2\imath & 0 & \imath & \imath 
\end{bmatrix}.
}
$$
We have a look at 
\begin{align*}
 & \conv(0_{\ZZ_4},v(\tau_1)',v_{01},v_{02}) \\
=~ &
\conv(0_{\ZZ_4},(0,0,-1,\imath),(-3,-3,1,-2\imath),(-1,-1,0,0))  \subseteq 
A_X^c(\lambda_0).
\end{align*}
There is a "critical" point that can lie inside this polytope, namely 
$$
(-1,-1,0,-1)= \frac{1}{\imath} v(\tau_1)' +\frac{1}{\imath} v_{01} + \frac{\imath-3}{\imath} v_{02}.
$$
Only for $\imath=2$, it does not and we get our first canonical singularity with defining matrix $P_1:=P_{1,2}$.

\smallskip

\noindent
\emph{Case 3.1.2: $\imath$ is odd and $6d_{211}-3$ is a multiple of $\imath$}. Since $\imath$ and $3$ are coprime, we have that $2d_{211}-1$ is a multiple of $\imath$ and moreover, with admissible operations, we can achieve $d_{211}=(\imath+1)/2$. A simple computation as above shows that $(1,0,\imath-2,\imath-1)$ lies inside $A_X^c(\lambda_1)$ for any odd $\imath$ coprime to three.
We come to:

\medskip

\noindent
\emph{Case 3.2: $t=2$}. In order not to have two integer points in $A_X^c(\lambda)$, we require $s=-1/3$. Then Corollary~\ref{corr:halfcone} tells us that $3d_{211}-3$ is a multiple of $\imath$ and the same admissible operation as in Case 3.1 for $s=-2/3$ brings the matrix into the form
$$
{\tiny
\begin{bmatrix}
-3 & -2 & 3 & 0  \\
-3 & -2 & 0 & 2  \\
0 & 0  & -1 & 1 \\
-2\imath & 0 & \imath & \imath 
\end{bmatrix}.
}
$$
and again as in Case 3.1 the first two columns of $P$ together with $0_{\ZZ_4}$ build up a non-canonical two-dimensional polytope.

Now as we have seen, the only possible matrix with an additional column in $\lambda_0$ has leading block together with the additional column
$$
{\tiny
\begin{bmatrix}
-3 & -1 & 3 & 0  \\
-3 & -1 & 0 & 2  \\
1 & 0  & 0 & -1 \\
-4 & 0 & 2 & 2 
\end{bmatrix}.
}
$$
As we have seen above, the blocks $\lambda_0$ and $\lambda_1$ are interchangeable, so there must be a column in $\lambda_1$ that together with the leading block gives an equivalent matrix. This column is $(1,0,-1,2)$. This is because the admissible operations of subtracting the first from the third and adding the last to the third row, then negating the third row and finally changing the data of the leaves $\lambda_0$ and $\lambda_1$ brings the matrix from above in the form:
$$
{\tiny
\begin{bmatrix}
-3  & 3 & 1 & 0  \\
-3  & 0 & 0 & 2  \\
1   & 0 & -1 & -1 \\
-4  & 2 & 2 & 2 
\end{bmatrix}.
}
$$
Thus it is possible that a defining matrix  $P$ contains $\emph{both}$ of these columns and in fact all subpolytopes of the anticanonical complex of
$$
P_2 :=
{\tiny
\begin{bmatrix}
-3 & -1 & 3 & 1 & 0  \\
-3 & -1 & 0 & 0 & 2  \\
1 & 0  & 0 & -1 & -1 \\
-4 & 0 & 2 & 2 & 2 
\end{bmatrix}
}
$$
are canonical. Thus we get our first non-$\QQ$-factorial variety with defining matrix $P_2$.
As there are no other possible additional columns, we are done with Case 3.

\bigskip


\noindent
\emph{Case 4: leading platonic tuple $(l_{01},2,2)$}.
The matrix $P$ has leading block 
$$
{\tiny
\begin{bmatrix}
-l_{01} & 2  & 0 & 0 & \dots & 0 
\\
-l_{01} & 0 & 2 & 0 & \dots & 0 
\\
-l_{01} & 0 & 0 & 1 & \dots & 0 
\\
\vdots & \vdots & \vdots & \vdots & \ddots & \vdots 
\\
-l_{01} & 0 & 0 &  0 &  \dots &  1 
\\
d_{011} & d_{111} & d_{211} & 0 & \ldots & 0 
\\
\imath(1-l_{01}(r-1)) & \imath & \imath & \imath & \ldots & {\imath} 
\end{bmatrix}
}
$$
with $l_{01}\geq 2$ due to Proposition~\ref{prop:zetaexcep}. If $r\geq 3$, due to irredundancy we have at least another column $v_{32}=e_{3}+d_{321}e_{r+1}+\imath e_{r+2}$ in $P$ with $d_{321}\neq 0$. 
So $A_X^c(\lambda)$ contains the points
\begin{align*}
v(\tau_1)' &= \left(0,\ldots,0,d_{011} + \frac{l_{01}(d_{111} + d_{211})}{2},\imath\right), \\
v(\tau_2)' &= \left(0,\ldots,0,d_{011} + \frac{l_{01}(d_{111} + d_{211})}{2}  + l_{01}d_{321},\imath\right)
\end{align*}
and since $|l_{01}d_{321}| \geq 2$, with Lemma~\ref{le:twopoints}, the corresponding variety can not be canonical. Thus only $r=2$ is possible and the leading block is
$$
{\tiny
\begin{bmatrix}
-l_{01} & 2 & 0  \\
-l_{01} & 0 & 2  \\
d_{011} & d_{111} & d_{211} \\
(1-l_{01})\imath & \imath & \imath 
\end{bmatrix}
}
$$
with $v(\tau_1)' =  \left(0,\ldots,0,d_{011} + \frac{l_{01}(d_{111} + d_{211})}{2},\imath\right)\in \ZZ^4$. We distinguish two cases in the following:

\medskip

\noindent
\emph{Case 4.1: $\imath$ is even}. Since the last two columns of the leading block must be primitive, $d_{111}$ and $d_{211}$ must be odd and thus by adding appropriate multiples of the first two rows to the third, we can assume $d_{111}=d_{211}=1$. So $v(\tau_1)' =  \left(0,0,d_{011} + l_{01},\imath\right)$ is integer. The points $p_{i,t}$, defined as in Case 3, here have the form:
\begin{align*}
p_{0,t}
&:=(-t,-t,d_{011}+l_{01}-t,(1-t)\imath)
& t=0,\ldots,l_{01}
\\
p_{1,t}
&:=(t,0,d_{011}+l_{01}+t(1-d_{011}-l_{01})/2,\imath)
& t=0,1,2
\\
p_{2,t}
&:=(0,t,d_{011}+l_{01}+t(1-d_{011}-l_{01})/2,\imath)
& t=0,1,2
\end{align*}
Again, additional columns must be of the form $v_{i2}=p_{i,t}+se_3$.
 For $i=0,1,2$, we have $v(\tau_2)'= v(\tau_1)'+\delta_i e_3$, with
$$
\delta_0 = s,
\quad
\delta_1 = \delta_2 =  \frac{2l_{01}s}{2l_{01}-(l_{01}-2)t}.
$$ 
In the first case $i=0$, we have integer $s$ and thus $A_X^c$ contains two integer points, we get no canonical variety. The cases $i=1,2$ are again equivalent. We examine $i=1$. If $t=0,2$ or $d_{011}+l_{01}$ is odd, then $s$ is integer. In all these cases, since $l_{01}\geq 2$, we have that
$$
0 \neq \frac{2l_{01}s}{2l_{01}-(l_{01}-2)t} \in \ZZ,
$$
so $A_X^c(\lambda)$ contains two integer points and we are done with these cases. The case $t=1$ and $d_{011}+l_{01}$ even remains. Here $s \in \ZZ + 1/2$. Only for $s=\pm 1/2$, we have
$$
\left|\frac{2l_{01}s}{2l_{01}-(l_{01}-2)t}\right|<1,
$$
otherwise we have two integer points in $A_X^c(\lambda)$. Now with $s=\pm 1/2$, from Corollary~\ref{corr:halfcone} applied to the polytope
$$
\conv\left(0_{\ZZ_4}, \left(0,0,d_{011} + l_{01},\imath\right), \left(0,0,d_{011} + l_{01} \pm \frac{l_{01}}{l_{01}+2},\imath\right)\right) \subseteq A_X^c(\lambda)
$$
follows that $d_{011} + l_{01} \pm 1$ is a multiple of $\imath$. But this is a contradiction since $\imath$ is even and $d_{011} + l_{01}$ is even as well.

\smallskip

\noindent
\emph{Case 4.2: $\imath$ is odd}. If $d_{111} \equiv d_{211} \mod 2$, then by  if necessary adding the last row to the third and then adding appropriate multiples of the first and second row two the third, we achieve 
$d_{111}= d_{211} = 0$. If $d_{111} \equiv d_{211} + 1 \mod 2$ , we achieve $d_{111}= 0$, $d_{211} = 1$. We distinguish both cases in the following. Observe that in the first as well as in the second case, we do not have to distinguish between additional columns in the first or in the second leaf due to admissible operations.

\smallskip

\noindent
\emph{Case 4.2.1: $d_{111}= d_{211} = 0$}. So $v(\tau_1)' =  \left(0,0,d_{011},\imath\right)$. The points $p_{i,t}$, defined as in Case 4.1, here have the form:
\begin{align*}
p_{0,t}
&:=(-t,-t,d_{011},(1-t)\imath)
& t=0,\ldots,l_{01}
\\
p_{1,t}
&:=(t,0,d_{011}(1-t/2),\imath)
& t=0,1,2
\\
p_{2,t}
&:=(0,t,d_{011}(1-t/2),\imath)
& t=0,1,2
\end{align*}
Additional columns must be of the form $v_{i2}=p_{i,t}+se_3$.
We have $v(\tau_2)'= v(\tau_1)'+\delta_i e_3$ with $\delta_i$ as in Case 4.1.
Thus as in Case 4.1, the only possibility for a canonical singularity is $i=1$ ($i=2$ is equivalent), $t=1$, odd $d_{011}$, $s=\pm 1/2$, since otherwise $A_X^c(\lambda)$ would contain two integer points. Moreover, from Corollary~\ref{corr:halfcone} applied on the polytope
$$
\conv\left(0_{\ZZ_4}, \left(0,0,d_{011},\imath\right), \left(0,0,d_{011} \pm \frac{l_{01}}{l_{01}+2},\imath\right)\right) \subseteq A_X^c(\lambda)
$$
follows that $d_{011} \pm 1$ is a multiple of $\imath$. By the admissible operations of adding the $(d_{011}\pm 1)/2$-fold of the first and second to the third row and subtracting the $(d_{011} \pm 1)/\imath$-fold of the last from the third row and optionally negating the third row, we achieve the following form for our matrix:
$$
{\tiny
\begin{bmatrix}
-l_{01} & 2 & 1 & 0  \\
-l_{01} & 0 & 0 & 2  \\
1 & 0 & 0 &  0 \\
(1-l_{01})\imath & \imath & \imath & \imath 
\end{bmatrix}.
}
$$
But $A_X^c(\lambda_1)$ contains a two-dimensional polytope with vertices $0_{\ZZ_4}$ and the second and third column of $P$, which can not be canonical. We are done with Case 4.2.1.
\smallskip

\noindent
\emph{Case 4.2.2: $d_{111}= 0$, $d_{211} = 1$}. So $v(\tau_1)' =  \left(0,0,d_{011} + l_{01}/2,\imath\right)$ is integer. The points $p_{i,t}$, defined as in Case 3, here have the form:
\begin{align*}
p_{0,t}
&:=(-t,-t,d_{011}+(l_{01}-t)/2,(1-t)\imath)
& t=0,\ldots,l_{01}
\\
p_{1,t}
&:=(t,0,(1-t/2)(d_{011}+l_{01}/2),\imath)
& t=0,1,2
\\
p_{2,t}
&:=(0,t,(1-t/2)(d_{011}+l_{01}/2)+t/2,\imath)
& t=0,1,2
\end{align*}
Again, additional columns must be of the form $v_{i2}=p_{i,t}+se_3$.
Again we have $v(\tau_2)'= v(\tau_1)'+\delta_i e_3$ with $\delta_i$ as in Case 4.1.

\smallskip

\noindent
\emph{Case 4.2.2.1: $i=0$}. If $l_{01}$ is even and $t$ as well, then $s$ is integer and $A_X^c(\lambda)$ contains the two integer points $v(\tau_1)'$ and $v(\tau_2)'$. If both are odd, then Corollary~\ref{corr:halfhalfcone} applied to $\conv(0_{\ZZ_4},v(\tau_1)',v(\tau_2)')$ forces $\imath=2$, which is a contradiction, since $\imath$ is odd. So it remains $l_{01} + t$ odd and $s=\pm 1/2$. 

\smallskip

\noindent
\emph{Case 4.2.2.1.1: $l_{01}$ even, $t$ odd}.
Corollary~\ref{corr:halfcone} forces $d_{011} + l_{01}/2 \pm 1$ to be a multiple of $\imath$.
The polytope 
\begin{align*}
\Bb_1 := \conv( &0_{\ZZ_4},\left(0,0,d_{011} + l_{01}/2 ,\imath\right),\left(0,0,d_{011} + l_{01}/2 \pm 1/2,\imath\right), \\
&(-2,-2,d_{011}+(l_{01}-2)/2,-\imath),(-1,-1,d_{011}+(l_{01}-1)/2-\pm 1/2,0))
\end{align*}
is contained in $A_X^c(\lambda_0)$. Its union $\Bb_{12}:=\Bb_1 \cup \Bb_2$ with
\begin{align*}
\Bb_2 := \conv( &0_{\ZZ_4},\left(0,0,d_{011} + l_{01}/2 ,\imath\right),\left(0,0,d_{011} + l_{01}/2 \pm 1/2,\imath\right), \\
&(1,1,d_{011}+(l_{01}+1)/2,2\imath))
\end{align*}
is lattice equivalent to the polytope from Corollary~\ref{corr:quadrangle} and can thus not be canonical. But since $\Bb_2$ is canonical, $\Bb_1$ can not be.

\smallskip

\noindent
\emph{Case 4.2.2.1.2: $l_{01}$ odd, $t$ even}.
Corollary~\ref{corr:halfcone} forces $d_{011} + l_{01}/2 \pm 1/2$ to be a multiple of $\imath$.
The argument is the same as before, now with
\begin{align*}
\Bb_1 := \conv(&0_{\ZZ_4},\left(0,0,d_{011} + l_{01}/2 ,\imath\right),\left(0,0,d_{011} + l_{01}/2 \pm 1/2,\imath\right), \\
&(-2,-2,d_{011}+(l_{01}-2)/2 \pm 1/2,-\imath),(-1,-1,d_{011}+(l_{01}-1)/2,0)) \\
\Bb_2 := \conv(&0_{\ZZ_4},\left(0,0,d_{011} + l_{01}/2 ,\imath\right),\left(0,0,d_{011} + l_{01}/2 \pm 1/2,\imath\right), \\
&(1,1,d_{011}+(l_{01}+1)/2 \pm 1/2,2\imath)).
\end{align*}

\smallskip

\noindent
\emph{Case 4.2.2.2: $i=1$}. If $l_{01}$ and $t$ are even, then $s$ is integer and $A_X^c(\lambda)$ contains two integer points. If $l_{01}$ is even and $t=1$, then if $(d_{011}+l_{01}/2)$ is even, $s$ is integer and $A_X^c(\lambda)$ contains two integer points. If $d_{011}+l_{01}/2$ is odd, we require $s=\pm 1/2$ and Corollary~\ref{corr:halfcone} forces $d_{011}+l_{01}/2 \pm 1$ to be a multiple of $\imath$. By the admissible operations of adding the $(d_{011}+l_{01}/2 \pm 1)/2$-fold of the first and second to the third row and subtracting the $(d_{011}+l_{01}/2 \pm 1)/\imath$-fold of the last from the third row, we achieve the following form for our matrix:
$$
{\tiny
\begin{bmatrix}
-l_{01} & 2 & 1 & 0  \\
-l_{01} & 0 & 0 & 2  \\
-(l_{01}/2 \pm 1) & 0 & 0 &  1 \\
(1-l_{01})\imath & \imath & \imath & \imath 
\end{bmatrix}.
}
$$
But as in Case 4.2.1, $A_X^c(\lambda_1)$ contains a two-dimensional polytope with vertices $0_{\ZZ_4}$ and the second and third column of $P$, which can not be canonical. Now the case $l_{01}$ odd remains. We distinguish $t=0,1,2$.

\smallskip

\noindent
\emph{Case 4.2.2.2.1: $t=0$}. We need $s=\pm 1/2$ here and again Corollary~\ref{corr:halfcone} forces $d_{011}+l_{01}/2 \pm 1/2$ to be a multiple of $\imath$. Similar to the last case, by admissible operations we achieve the form
$$
{\tiny
\begin{bmatrix}
-l_{01} & 2 & 0 & 0  \\
-l_{01} & 0 & 2 & 0  \\
-(l_{01}/2 \pm 1/2) & 0 & 1 &  0 \\
(1-l_{01})\imath & \imath & \imath & \imath 
\end{bmatrix}
}
$$
for our matrix and again $A_X^c(\lambda_1)$ contains a non-canonical two-dimensional polytope with vertices $0_{\ZZ_4}$ and the second and fourth column of $P$.

\smallskip

\noindent
\emph{Case 4.2.2.2.2: $t=1$}. We require $d_{011}/2 + l_{01}/4 + s \in \ZZ$. By if necessary subtracting the second from the third row and afterwards negating the third row, we can achieve that $d_{011}$ is even. Now we distinguish two subcases:

\smallskip

\noindent
\emph{Case 4.2.2.2.2.1: $ l_{01} \equiv 1 \mod 4$}. So $s \in \ZZ + 3/4$. We have a look at the leaving points of the two  elementary cones:
\begin{align*}
v(\tau_1)' = &  \left(0,0,d_{011} + l_{01}/2 ,\imath\right), \\
v(\tau_2)' = &  \left(0,0,d_{011} + l_{01}/2 + 2l_{01}s/(l_{01}+2) ,\imath\right).
\end{align*}
For $|s|\geq 3/4$, we have $|l_{01}/2 + 2l_{01}s/(l_{01}+2)|\geq 3l_{01}/(2l_{01}+4) > 1$ since $l_{01} > 4$. In this case applying Corollary~\ref{corr:halfhalfcone} to $A_X^c(\lambda)$, we get $\imath=2$, which is a contradiction. Thus the only possibility is $s=-1/4$. 
Lemma~\ref{le:1245cone} tells us that either $d_{011} + (l_{01}-1)/2$ or $d_{011} + (l_{01}+1)/2$ is a multiple of $\imath$. 

\smallskip

\noindent
\emph{Case 4.2.2.2.2.1.1: $d_{011} + (l_{01}-1)/2$ is a multiple of $\imath$}. Since $ l_{01} \equiv 1 \mod 4$, it is even as well. Thus adding the $(d_{011} + (l_{01}-1)/2)/2$-fold of the first and the second to the third and subtracting the $(d_{011} + (l_{01}-1)/2)/\imath$-fold of the last from the third row, we  bring our matrix into the form
$$
{\tiny
\begin{bmatrix}
-l_{01} & 2 & 1 & 0  \\
-l_{01} & 0 & 0 & 2  \\
(1-l_{01})/2 & 0 & 0 &  1 \\
(1-l_{01})\imath & \imath & \imath & \imath 
\end{bmatrix}.
}
$$
Again $A_X^c(\lambda_1)$ contains a non-canonical two-dimensional polytope with vertices $0_{\ZZ_4}$ and the second and third column of $P$.

\smallskip

\noindent
\emph{Case 4.2.2.2.2.1.2: $d_{011} + (l_{01}+1)/2$ is a multiple of $\imath$}. It is odd as well. So we can write $d_{011}=k\imath-(l_{01}+1)/2$ with odd $k \in \ZZ$. Adding the $(k-1)\imath/2$-fold of the first and second and the $(1-k)$-fold of the last to the third row, we bring our matrix into the form
$$
{\tiny
\begin{bmatrix}
-l_{01} & 2 & 1 & 0  \\
-l_{01} & 0 & 0 & 2  \\
\imath-(1+l_{01})/2 & 0 & (\imath-1)/2 &  1 \\
(1-l_{01})\imath & \imath & \imath & \imath 
\end{bmatrix}.
}
$$
We have a look at the polytope
\begin{align*}
& \conv\left(0_{\ZZ_4},\left(-l_{01},-l_{01},\imath-\frac{1+l_{01}}{2},(1-l_{01})\imath\right),\left(0,0,\imath-\frac{1}{2},\imath\right),\left(0,0,\frac{3l_{01}+2}{2(l_{01}+2)}\right)\right) \\
 =&\conv(0_{\ZZ_4},v_{01},v(\tau_1)',v(\tau_2)') \subseteq  A_X^c(\lambda_0)
\end{align*}
and the point
$$
(-1,-1,\imath-2,-1)= \frac{1}{l_{01}}v_{01}+\left(1-\frac{\imath+2(1+l_{01})}{\imath l_{01}}\right)v(\tau_1)'+\frac{l_{01}+2}{\imath l_{01}}v(\tau_2)',
$$
which is in the relative interior since $\imath \geq 3$ and $l_{01} \geq 5$. Thus the corresponding variety is not canonical.

\smallskip

\noindent
\emph{Case 4.2.2.2.2.2: $ l_{01} \equiv 3 \mod 4$}. Here $s \in \ZZ+1/4$ and with the same argument as in Case  4.2.2.2.2.1, we have $s=1/4$. Lemma~\ref{le:1245cone} tells us that either $d_{011} + (l_{01}-1)/2$ or $d_{011} + (l_{01}+1)/2$ is a multiple of $\imath$. So we distinguish the same subcases as in Case 4.2.2.2.2.1.:

\smallskip

\noindent
\emph{Case 4.2.2.2.2.2.1: $d_{011} + (l_{01}-1)/2$ is a multiple of $\imath$}. Since $ l_{01} \equiv 3 \mod 4$, it is now odd. With the same admissible operations as in Case 4.2.2.2.2.1.2, we arrive at the matrix
$$
{\tiny
\begin{bmatrix}
-l_{01} & 2 & 1 & 0  \\
-l_{01} & 0 & 0 & 2  \\
\imath-(l_{01}-1)/2 & 0 & (\imath+1)/2 &  1 \\
(1-l_{01})\imath & \imath & \imath & \imath 
\end{bmatrix}.
}
$$
Now the point 
$$
(-1,-1,\imath-1,-1)= \frac{1}{l_{01}}v_{01}+\left(1-\frac{\imath+2(1+l_{01})}{\imath l_{01}}\right)v(\tau_1)'+\frac{l_{01}+2}{\imath l_{01}}v(\tau_2)'
$$
is again in the relative interior of $A_X^c(\lambda_0)$ for $\imath\geq 3$ and $l_{01}\geq 7$. But for $l_{01}=3$, this is the case only for $\imath\geq 4$. So for $\imath=3$, we get a canonical singularity from the matrix
$$
P_3:=
{\tiny
\begin{bmatrix}
-3 & 2 & 1 & 0  \\
-3 & 0 & 0 & 2  \\
2 & 0 & 2 &  1 \\
-6 & 3 & 3 & 3 
\end{bmatrix}.
}
$$

\smallskip

\noindent
\emph{Case 4.2.2.2.2.2.2: $d_{011} + (l_{01}+1)/2$ is a multiple of $\imath$}. Since $ l_{01} \equiv 3 \mod 4$, $d_{011} + (l_{01}+1)/2$ is now even. With the same admissible operations as in Case 4.2.2.2.2.1.1, we arrive at the matrix 
$$
{\tiny
\begin{bmatrix}
-l_{01} & 2 & 1 & 0  \\
-l_{01} & 0 & 0 & 2  \\
(1+l_{01})/2 & 0 & 0 &  1 \\
(1-l_{01})\imath & \imath & \imath & \imath 
\end{bmatrix}.
}
$$
Again $A_X^c(\lambda_1)$ contains a non-canonical two-dimensional polytope with vertices $0_{\ZZ_4}$ and the second and third column of $P$.

\smallskip

\noindent
\emph{Case 4.2.2.2.3: $t=2$}. Here $s$ is integer. Applying Corollary~\ref{corr:halfhalfcone} to $A_X^c(\lambda)$, we get $\imath=2$, which is a contradiction. 

So all in all we found one $\QQ$-factorial matrix in Case $4$, namely 
$$
P_3=
{\tiny
\begin{bmatrix}
-3 & 2 & 1 & 0  \\
-3 & 0 & 0 & 2  \\
2 & 0 & 2 &  1 \\
-6 & 3 & 3 & 3 
\end{bmatrix}.
}
$$
Thus we have to check if an additional column $(0,1,d_{211},3)$ in $\lambda_2$ is possible. But for such a column, we need $|s|=1/4$ and have the same situation as in Case 4.2.2.2.2.2.1, so we get no canonical singularity. We are done with Case $4$.

 
\bigskip

\noindent
\emph{Case 5: leading platonic tuple $(l_{01},l_{11},1)$}.
The matrix $P$ has leading block 
$$
{\tiny
\begin{bmatrix}
-l_{01} & l_{11}  & 0 & 0 & \dots & 0 
\\
-l_{01} & 0 & 1 & 0 & \dots & 0 
\\
-l_{01} & 0 & 0 & 1 & \dots & 0 
\\
\vdots & \vdots & \vdots & \vdots & \ddots & \vdots 
\\
-l_{01} & 0 & 0 &  0 &  \dots &  1 
\\
d_{011} & d_{111} & 0 & 0 & \ldots & 0 
\\
\imath(1-l_{01}(r-1)) & \imath & \imath & \imath & \ldots & {\imath} 
\end{bmatrix}
}
$$
with $l_{01}\geq l_{11} \geq 1$ due to Proposition~\ref{prop:zetaexcep}.
If $r\geq 3$, then due to irredundancy we have at least \emph{two} more columns $v_{22}=e_{2}+d_{221}e_{r+1}$ and $v_{32}=e_{3}+d_{321}e_{r+1}$, where we can assume $d_{221}\geq 1$ by if this is not the case adding the  $-d_{221}$ of the second to the penultimate row. The same holds for $d_{321}$. So the intersections of $\partial A_X^c$ with the elementary cones $\tau_1$ set up by the leading block and $\tau_2$ set up by the leading block with columns $v_{21}$ and $v_{31}$ replaced by $v_{22}$ and $v_{32}$ are the leaving points
\begin{align*}
v(\tau_1)' = & \left( 0,\ldots,0, \frac{l_{11}d_{011}+l_{01}d_{111}}{l_{11}+l_{01}},\imath\right), \\
v(\tau_2)' = &  \left(0,\ldots,0, \frac{l_{11}d_{011}+l_{01}d_{111} + (d_{221}+d_{321})l_{01}l_{11}}{l_{11}+l_{01}},\imath\right).
\end{align*}
Now we distinguish three subcases.

\medskip

\noindent
\emph{Case 5.1: $l_{11}\geq 2$}. Here if $r\geq 3$, we have $|v(\tau_1)' - v(\tau_2)'| \geq \frac{ 2l_{01}l_{11}}{l_{11}+l_{01}} \geq 2$ and thus $A_X^c(\lambda)$ contains two integer points besides $0_{\ZZ_{r+2}}$. Lemma~\ref{le:twopoints} tells us that the corresponding variety can not be canonical. So we have $r=2$.
The necessity of an additional column $v_{21}$ remains and the leading block together with it as well as the two leaving points are
$$
{\tiny
\begin{bmatrix}
-l_{01} & l_{11}  & 0 & 0 
\\
-l_{01} & 0 & 1 & 1 

\\
d_{011} & d_{111} & 0 & d_{221} 
\\
\imath(1-l_{01}) & \imath & \imath & \imath 
\end{bmatrix}, 
}
$$
$$
v(\tau_1)' =  \left( 0,0, \frac{l_{11}d_{011}+l_{01}d_{111}}{l_{11}+l_{01}},\imath\right), 
v(\tau_2)' =  v(\tau_1)' + \frac{d_{221}l_{01}l_{11}}{l_{11}+l_{01}}e_3
$$
So now $|v(\tau_1)' - v(\tau_2)'| \geq \frac{ l_{01}l_{11}}{l_{11}+l_{01}} \geq 2$ if $l_{11}\geq 4$ or if $l_{11}=3$ and $l_{01}\geq 6$. As above, we can exclude these cases with Lemma~\ref{le:twopoints}. We examine the remaining cases with $l_{11}=3$. If $l_{01}=3$, then even for $d_{211}=1$ one of  $v(\tau_1)'$ and $v(\tau_2)'$ is integer and their distance is $3/2$, Lemma~\ref{le:twopoints} rules out this case. If $l_{01}=4$, then $d_{211}=1$ must hold and we have 
$$
v(\tau_1)' =  \left( 0,0, \frac{3d_{011}+4d_{111}}{7},\imath\right), \quad
v(\tau_2)' =  \left(0,0, \frac{3d_{011}+4d_{111} + 12}{7},\imath\right).
$$
So $A_X^c(\lambda)$ might not contain two integer points only if $3d_{011}+4d_{111} \equiv 1 \mod 7$. Corollary~\ref{corr:halfhalfcone} forces $\imath=2$. By admissible operations, we achieve $d_{111}=0$. So $d_{011}=5+7k$ and with Corollary~\ref{corr:halfhalfcone}, we have odd $k$. By adding the $k$-fold of the first, the $3k$-fold of the second and the $-3k/2$-fold of the last to the third row, we achieve $d_{011}=5$. But then $(-1,-1,2,-1)$ lies inside $A_X^c(\lambda_0)$ and the corresponding variety is not canonical. It remains $l_{01}=5$. Again $d_{211}=1$ must hold and we have 
$$
v(\tau_1)' =  \left( 0,0, \frac{3d_{011}+5d_{111}}{8},\imath\right), \quad
v(\tau_2)' =  \left(0,0, \frac{3d_{011}+5d_{111} + 15}{8},\imath\right),
$$
so we have two integer points in $A_X^c(\lambda)$.

 We examine the cases with $l_{11}=2$. First from $|v(\tau_1)' - v(\tau_2)'| = \frac{ 2l_{01}d_{221}}{2+l_{01}} \geq 2$ if $d_{221}\geq 2$ follows that $d_{221} = 1$ must hold. We have 
$$
v(\tau_1)' =  \left( 0,0, \frac{2d_{011}+l_{01}d_{111}}{l_{01}+2},\imath\right), \quad
v(\tau_2)' =  \left(0,0, \frac{2d_{011}+l_{01}d_{111} + 2l_{01}}{l_{01}+2},\imath\right).
$$
Now for $l_{01}\geq 6$, we have $|v(\tau_1)' - v(\tau_2)'|\geq 3/2$  and therefore $A_X^c(\lambda)$ either contains two integer points or a polytope like in Corollary~\ref{corr:halfhalfcone}, which forces $\imath=2$. For $l_{01} = 5$, the same holds, since $|v(\tau_1)' - v(\tau_2)'|=10/7$ and for $2d_{011}+5d_{111} \equiv 0,4,5,6 \mod 7$, the polytope $A_X^c(\lambda)$ contains two integer points while for $2d_{011}+5d_{111} \equiv 1,2,3 \mod 7$, it contains a polytope like in Corollary~\ref{corr:halfhalfcone}. For $l_{01} = 4$, we have $|v(\tau_1)' - v(\tau_2)'|=4/3$ and for $d_{011}+2d_{111} \equiv 0,2 \mod 3$ we have two integer points while for $d_{011}+2d_{111} \equiv 1 \mod 3$, we have a polytope like in Corollary~\ref{corr:halfhalfcone}. For $l_{01} = 3$, we have $|v(\tau_1)' - v(\tau_2)'|=6/5$, so for $2d_{011}+3d_{111} \equiv 0,4 \mod 5$, we have two integer points and for $2d_{011}+3d_{111} \equiv 2 \mod 5$, we have a polytope like in Corollary~\ref{corr:halfhalfcone}. For $l_{01} = 2$, we have $|v(\tau_1)' - v(\tau_2)'|=1$, so for $d_{011}+ d_{111}$ even, we have two integer points and for $d_{011}+d_{111}$ odd, we have a polytope like in Corollary~\ref{corr:halfhalfcone}. This means that $\imath=2$ must hold in all cases but $l_{01} = 3$, $2d_{011}+3d_{111} \equiv 1,3 \mod 5$.
So first let $l_{01} = 3$ and $2d_{011}+3d_{111} \equiv 1,3 \mod 5$.
But using Corollary~\ref{corr:1313cone}, we either get $\imath=2$ or $\imath=3$. In the case $\imath=3$, by admissible operations we achieve $d_{111}=0$ and  our matrix has the form:
$$
{\tiny
\begin{bmatrix}
-3 & 2  & 0 & 0 
\\
-3 & 0 & 1 & 1 

\\
d_{011} & 0 & 0 & 1 
\\
-6 & 3 & 3 & 3 
\end{bmatrix}.
}
$$
Since $d_{111}=0$, we either have $2d_{011}/5=3k+1/5$ or $2d_{011}/5=3k+8/5$ for some $k \in \ZZ$. We can add multiples of $15$ to $d_{011}$ due to admissible operations. Since in the first case we have odd $k$ and in the second one, we have even $k$, we get $d_{011} \in \{4,8\}$. These two cases are equivalent by admissible operations as well. But for $d_{011}=4$, the point $(-1,-1,2,0)$ lies inside $A_X^c(\lambda_0)$ and the corresponding variety is not canonical.

So only the cases with $l_{11}=\imath=2$ are left. Here by primitivity of the second column and admissible operations, we achieve $d_{111}=1$ and have the matrix
$$
{\tiny
\begin{bmatrix}
-l_{01} & 2  & 0 & 0 
\\
-l_{01} & 0 & 1 & 1 

\\
d_{011} & 1 & 0 & 1 
\\
2(1-l_{01}) & 2 & 2 & 2 
\end{bmatrix}.
}
$$
By adding the $k(l_{01}+2)$-fold of the first, the $2k(l_{01}+2)$-fold of the second and the $-k(l_{01}+2)$-fold of the last to the third row for $k=\lfloor d_{011}/(l_{01}+2) \rfloor$, we can assume $0 \leq d_{011} \leq l_{01}+1$. Since
$$
v(\tau_1)' =  \left( 0,0, \frac{2d_{011}+l_{01}}{l_{01}+2},2\right), \quad
v(\tau_2)' =  \left(0,0, \frac{2d_{011}+l_{01} + 2l_{01}}{l_{01}+2},2\right),
$$
we need $2d_{011}+l_{01} = i(l_{01}+2)+j$ with $i=0,2$ and $j \in \{1,2,3\}$ in order not to have $(0,0,1,1)$ or $(0,0,2,1)$ inside $A_X^c(\lambda)$. First consider $i=0$. Since $l_{01}\geq 2$, this forces $d_{011}=0$ and $l_{01}=3$ due to primitivity of the first column.  This results in the matrix
$$
P_4:=
{\tiny
\begin{bmatrix}
-3 & 2  & 0 & 0 
\\
-3 & 0 & 1 & 1 

\\
0 & 1 & 0 & 1 
\\
-4 & 2 & 2 & 2 
\end{bmatrix},
}
$$
giving a canonical singularity. Now consider $i=2$. This leads to $d_{011}=(l_{01}+j')/2$ with $j' \in \{5,6,7\}$. 

\smallskip

\noindent
\emph{Case a: $j'=5$}. Then $l_{01}$ is odd. The point
$$
(-1,-1,2,-1)= \frac{1}{l_{01}}v_{01}+\frac{l_{01}-4}{4l_{01}}v(\tau_1)'+\frac{1}{4}v(\tau_2)'
$$
lies inside $A_X^c(\lambda_0)$ for $l_{01}\geq 5$. But for $l_{01}=3$, the resulting matrix is equivalent to the canonical one from above by adding the first, three times the second and $-2$ times the last to the third row and afterwards negating the third row. So we get no new canonical singularity in this case.

\smallskip

\noindent
\emph{Case b: $j'=6$}. Then $l_{01}$ is even. Due to primitivity of the first column, we require $l_{01} \equiv 0 \mod 4$. The point
$$
(-1,-1,2,-1)= \frac{1}{l_{01}}v_{01}+\frac{l_{01}-2}{4l_{01}}v(\tau_1)'+\frac{l_{01}-2}{4l_{01}}v(\tau_2)'
$$
lies in $A_X^c(\lambda_0)$ and so the corresponding singularity can not be canonical.

\smallskip

\noindent
\emph{Case c: $j'=7$}. Then $l_{01}$ is odd. The point
$$
(-1,-1,2,-1)= \frac{1}{l_{01}}v_{01}+\frac{1}{4}v(\tau_1)'+\frac{l_{01}-4}{4l_{01}}v(\tau_2)'
$$
lies inside $A_X^c(\lambda_0)$ for $l_{01}\geq 5$. For $l_{01}=3$, the resulting matrix again is equivalent to the canonical one from above.

The final step in Case 5.1 is now to check if to the canonical matrix
$$
{\tiny
\begin{bmatrix}
-3 & 2  & 0 & 0 
\\
-3 & 0 & 1 & 1 

\\
0 & 1 & 0 & 1 
\\
-4 & 2 & 2 & 2 
\end{bmatrix}
}
$$
columns in $\lambda_0$ or $\lambda_1$ can be added, as no additional columns in $\lambda_3$ are possible due to convexity.

\smallskip

\noindent
\emph{Case a: additional column $v_{02}=(-t,-t,s,2(1-t))$  with $t=1,2,3$}. Then for $\tau_3:=\cone(v_{02},v_{11},v_{21})$ and $\tau_4:=\cone(v_{02},v_{11},v_{22})$, we get
$$
v(\tau_3)'=\left(0,0,\frac{2s+t}{2+t},2\right), \quad v(\tau_4)'=\left(0,0,\frac{2s+3t}{2+t},2\right).
$$ 
Now $s=0$ is impossible, since in this case $\conv(0_{\ZZ_4},v_{01},v_{02})$ is a non-canonical two-dimensional cone. For $s\leq -1$, the point $(-1,-1,0,-1)$ lies inside $\conv(0_{\ZZ_4},v_{01},v_{02},v(\tau_1)')$. For $s\geq 2$ and $s=1$, $t=2,3$, the point $(0,0,1,1)$ lies inside the polytope $\conv(0_{\ZZ_4},v(\tau_1)',v(\tau_4)')$. But $s=1$, $t=1$ is impossible, since the corresponding point $(-1,-1,1,0)$ already lies inside $\conv(v_{01},v(\tau_1)',v(\tau_2)')$.

\smallskip

\noindent
\emph{Case b: additional column $v_{12}=(t,0,s,2)$  with $t=1,2$}. Here $s=1$ is impossible since the corresponding points already lie inside $\conv(v_{11},v(\tau_1)',v(\tau_2)')$. Then for $\tau_3:=\cone(v_{01},v_{12},v_{21})$ and $\tau_4:=\cone(v_{01},v_{12},v_{22})$, we get
$$
v(\tau_3)'=\left(0,0,\frac{3s}{3+t},2\right), \quad v(\tau_4)'=\left(0,0,\frac{3(s+t)}{3+t},2\right).
$$ 
For $s\leq 0$, the point  $(0,0,0,1)$ lies inside $\conv(0_{\ZZ_4},v(\tau_1)',v(\tau_3)')$ and for $s\geq 2$, the point $(0,0,1,1)$ lies inside $\conv(0_{\ZZ_4},v(\tau_2)',v(\tau_4)')$.
We are done with Case 5.1.

\medskip

\noindent
\emph{Case 5.2: $l_{11}=1$, $l_{01}\geq 2$}. We can assume $d_{111}=0$ and due to irredundancy, we need an additional column $v_{12}=(1,0,d_{121},\imath)$ in $\lambda_1$, where as before, we can assume $d_{121} \geq 1$. If $r\geq 3$, with 
$
\tau_2:=\cone(v_{01},v_{12},v_{32},v_{41},\ldots,v_{r1})
$
we have
\begin{align*}
v(\tau_1)' = & \left( 0,\ldots,0, \frac{d_{011}}{1+l_{01}},\imath\right), \\
v(\tau_2)' = &  \left(0,\ldots,0, \frac{d_{011} + (d_{121}+d_{221}+d_{321})l_{01}}{1+l_{01}},\imath\right).
\end{align*}
But $|v(\tau_1)' - v(\tau_2)'|\geq \frac{3l_{01}}{1+l_{01}}\geq 2$, since $l_{01}\geq 2$, so a corresponding singularity cannot be canonical.
Thus we have $r=2$ and our matrix has the form:
$$
{\tiny
\begin{bmatrix}
-l_{01} & 1 & 1  & 0 & 0 
\\
-l_{01} & 0 & 0 & 1 & 1 

\\
d_{011} & 0 & d_{121} & 0 & d_{221} 
\\
\imath(1-l_{01}) & \imath & \imath & \imath & \imath 
\end{bmatrix}. 
}
$$
With $\tau_2:=\cone(v_{01},v_{12},v_{32})$, we have
\begin{align*}
v(\tau_1)' = & \left( 0,0, \frac{d_{011}}{1+l_{01}},\imath\right), \quad
v(\tau_2)' =   \left(0,0, \frac{d_{011} + (d_{121}+d_{221})l_{01}}{1+l_{01}},\imath\right).
\end{align*}
So $|v(\tau_1)' - v(\tau_2)'|= \frac{(d_{121}+d_{221})l_{01}}{1+l_{01}}\leq 2$ only if $d_{121}=d_{221}=1$. But $A_X^c(\lambda)$ contains two integer points but for $d_{011} \equiv 1 \mod 1+l_{01}$. In this case it  contains a polytope like in Corollary~\ref{corr:halfhalfcone}. This forces $\imath=2$ and $d_{011} = 1 +2k(1+l_{01})$ for some $k \in \ZZ$. But by adding $2k$ times the first and second and $-k$ times the last to the third row, we arrive at the matrix
$$
P_5:=
{\tiny
\begin{bmatrix}
-l_{01} & 1 & 1  & 0 & 0 
\\
-l_{01} & 0 & 0 & 1 & 1 

\\
1 & 0 & 1 & 0 & 1 
\\
2(1-l_{01}) & 2 & 2 & 2 & 2 
\end{bmatrix},
}
$$
giving a canonical singularity for arbitrary $l_{01}\geq 2$.

\medskip

\noindent
\emph{Case 5.3: $l_{11}=l_{01}=1$}. We can assume $d_{111}=0$ and due to irredundancy, we need  additional columns $v_{02}=(-1,-1,d_{021},\imath)$ in $\lambda_0$ and $v_{12}=(1,0,d_{121},\imath)$ in $\lambda_1$, where as before, we can assume $d_{121} \geq 1$ and moreover $d_{021}>d_{011}$ . With $\tau_2:=\cone(v_{02},v_{12},v_{32},v_{41},\ldots,v_{r1})$, we have
\begin{align*}
v(\tau_1)' = & \left( 0,0, \frac{d_{011}}{2},\imath\right), \quad
v(\tau_2)' =   \left(0,0, \frac{d_{021} + d_{121} + d_{221}}{2},\imath\right),
\end{align*}
thus $|v(\tau_1)' - v(\tau_2)'|\geq 3/2$ and one of $v(\tau_1)'$ and $v(\tau_2)'$ is integer, so the corresponding singularity cannot be canonical.
We are done with all cases.

\end{proof}

Now we go through the cases of Proposition~\ref{prop:zetaexcep} (iii) with $\zeta_X >1$, beginning with Case $(i)$ in the following proposition:

\begin{proposition}
\label{prop:zeta=2-432}
Let $X$ be a non-toric threefold singularity of complexity one.
Assume that~$X$ is of canonical multiplicity 
two, even Gorenstein index $\imath \geq 2$, the leading platonic triple is $(4,3,2)$ and $X$ is at most canonical. 
Then $X$ is isomorphic to $X(P_i)$,
where $P_i$ is one of the following matrices:
 $$
 {\tiny
 P_6 = 
 \begin{bmatrix}
-4 & 3 & 0 & 0   \\
-4 & 0 & 2  & 0  \\
0 & 0 &1 & 5 \\
3 & -2 & 0 & 1
\end{bmatrix}
\quad
P_7
=
\begin{bmatrix}
-4 & 3 & 1 & 0   \\
-4 & 0 & 0  & 2  \\
0 & 0 & 3 & 1\\
3 & -2 & 0 & 0
\end{bmatrix}
\quad
P_8 =
 \begin{bmatrix}
-4 & 3 & 1 & 0 & 0   \\
-4 & 0 & 0 & 2  & 0  \\
0 & 0 & 3 & 1 & 7 \\
3 & -2 & 0 & 0 & 1
\end{bmatrix}
}
$$
$$
{\tiny
P_9 =
 \begin{bmatrix}
-4 & 3  & 0 & 0 & 0  \\
-4 & 0  & 2  & 0 & 0  \\
0 & 0  & 1 & 5 & 7\\
3 & -2  & 0 & 1 & 1
\end{bmatrix}
\quad
P_{10}
=
 \begin{bmatrix}
-4 & 3 & 1 & 1 & 0    \\
-4 & 0 & 0 & 0 & 2    \\
0 & 0 & 3  & 5 & 1  \\
3 & -2 & 0 & 0 & 0 
\end{bmatrix}
}
$$

\end{proposition}

\begin{proof}
The matrix $P$ falls under Case $(i)$ of Proposition~\ref{prop:zetaexcep} (iii), thus $\imath$ is even and the leading block has the shape
$$
{\tiny
\begin{bmatrix}
-4 & 3  & 0 & 0 & \dots & 0 
\\
-4 & 0 & 2 & 0 & \dots & 0 
\\
-4 & 0 & 0 & 1 & \dots & 0 
\\
\vdots & \vdots & \vdots & \vdots & \ddots & \vdots 
\\
-4 & 0 & 0 &  0 &  \dots &  1 
\\
d_{011} & d_{111} & d_{211} & 0 & \ldots & 0 
\\
\frac{\imath}{2}+2 & -\imath & \frac{\imath}{2} - 1 & 0 & \ldots & 0 
\end{bmatrix}.
}
$$
If $r\geq 3$, due to irredundancy we have at least another column $v_{32}=e_{3}+d_{321}e_{r+1}$ in $P$ with $d_{321} \geq 1$. So $A_X^c(\lambda)$ contains the integer points
\begin{align*}
v(\tau_1)' &= (0,\ldots,0,3d_{011}+ 4d_{111} + 6d_{211},\imath/2), \\
v(\tau_2)' &= (0,\ldots,0,3d_{011}+ 4d_{111} + 6d_{211}  + 12d_{321},\imath/2)
\end{align*}
and with Lemma~\ref{le:twopoints}, we get $\imath=2$.
By primitivity of the third column, we can write $d_{211}=2k+1$ for some $k\in\ZZ$, and by subtracting the $8k+2d_{011}+3d_{111}$-fold of the first, the $k$-fold of the second and the $12k+3d_{011}+4d_{111}$-fold of the last from the penultimate row, we achieve $d_{011}=d_{111}=0$ and $d_{211}=1$. so we get
$$
v(\tau_1)' = (0,\ldots,0, 6,1), 
\quad
v(\tau_2)' = (0,\ldots,0, 6  + 12d_{321},1).
$$
We have a look at $A_X^c(\lambda_0)$. The point
$$
(-1,-1,2,1)=\frac{1}{4}v_{01}+\frac{6d_{321}-1}{24d_{321}}v(\tau_1)'+\frac{1}{24d_{321}}v(\tau_2)'
$$
lies inside, so we have no canonical singularity.
Thus only $r=2$ is possible.
Now we have to check which additional columns are possible. 

\bigskip

\noindent
\emph{Case 1: additional column $v_{1}=(0,0,d_1,\imath/2)$ in $\lambda$}.
Since $v(\tau_1)'$ is integer, an additional column in the lineality part means two integer points in $A_X^c(\lambda)$, which forces $\imath=2$ with Lemma~\ref{le:twopoints}. With the same arguments as above, the leading block together with the additional column has the form:
$$
{\tiny
\begin{bmatrix}
-4 & 3 & 0 & 0   \\
-4 & 0 & 2  & 0  \\
0 & 0 &1 & d_{1} \\
3 & -2 & 0 & 1
\end{bmatrix}.
}
$$
Now the point
$$
(-1,-1,1,1)=\frac{1}{4}v_{01}+\frac{d_1-4}{4d_1-24}v(\tau_1)'+\frac{1}{12-2d_1}v_{1}'
$$
lies inside $A_X^c(\lambda_0)$ for $d_1\leq 4$, and the point
$$
(-1,-1,2,1)=\frac{1}{4}v_{01}+\frac{d_1-8}{4d_1-24}v(\tau_1)'+\frac{1}{2d_1-12}v_{1}'
$$
lies inside $A_X^c(\lambda_0)$ for $d_1\geq 8$. So the cases $d_1=5,7$ remain as $d_1=6$ is not possible due to $v(\tau_1)' = (0,0, 6,1)$. But by subtracting the second, eight times the first and 12 times the last from the third row and negating the third row afterwards, we see that both cases are equivalent and give a canonical singularity with matrix
$$
P_6:=
{\tiny
\begin{bmatrix}
-4 & 3 & 0 & 0   \\
-4 & 0 & 2  & 0  \\
0 & 0 &1 & 5 \\
3 & -2 & 0 & 1
\end{bmatrix}.
}
$$

\bigskip

\noindent
\emph{Case 2: additional column $v_{02}=(-t,-t,d_{021},(\imath+t)/2)$ in $\lambda_0$}. Since $v_{02}$ is integer, we have $t=4,2$.

\medskip

\noindent
\emph{Case 2.1: $t=4$}. We can assume $d_{021} > d_{011}$ and have the two integer points $v(\tau_i)' = (0,0,3d_{0i1}+ 4d_{111} + 6d_{211},\imath/2)$ inside $A_X^c(\lambda)$, forcing $\imath=2$ and by the same admissible operations as in Case 1, we get $d_{011}=d_{111}=0$, $d_{211}=1$, $v(\tau_1)' = (0,0,6,1)$ and $v(\tau_2)' = (0,0,6+3d_{021},1)$. As we have seen in Case 1, the point $(-1,-1,2,1)$ lies inside $A_X^c(\lambda_0)$.

\medskip

\noindent
\emph{Case 2.2: $t=2$}. Here  $v(\tau_1)'$ is integer and 
$
v(\tau_2)'=(0,0,3(d_{021}+d_{211})/2+d_{111},\imath/2)
$.
We distinguish two subcases in the following.

\smallskip

\noindent
\emph{Case 2.2.1: $\imath \equiv 2 \mod 4$}. Here due to primitivity of $v_{02}$ and $v_{21}$, we require both $d_{021}$ and $d_{211}$ to be odd. Thus $v(\tau_2)'$ is integer, which forces $\imath=2$. But then either $d_{021}\leq 1$ or $d_{021}\geq 5$ and either $(-1,-1,1,1)$ or $(-1,-1,2,1)$ lies inside $A_X^c(\lambda_0)$.

\smallskip

\noindent
\emph{Case 2.2.2: $\imath \equiv 0 \mod 4$}. Here $d_{021}+d_{211}$ must be odd due to Lemma~\ref{le:twopoints}. By if necessary adding the last row to the third, we can achieve $d_{021}$ odd and $d_{211}$ even. Then by adding appropriate multiples of the first two rows to the third, we get $d_{021}=1$ and $d_{211}=0$. We have
\begin{align*}
v(\tau_1)' &= (0,0,3d_{011}+ 4d_{111},\imath/2), \quad
v(\tau_2)' = (0,0,3/2+d_{111},\imath/2).
\end{align*}
Since $A_X^c(\lambda)$ must not contain two integer points due to Lemma~\ref{le:twopoints}, we require $3(d_{011}+ d_{111}) \in \{1,2\}$, which is not possible.

\bigskip

\noindent
\emph{Case 3: additional column $v_{12}=(t,0,d_{121},\imath(1-t)/2)$ in $\lambda_1$}. Here we have the possibilities $t=3,2,1$.

\medskip

\noindent
\emph{Case 3.1: $t=3$}. We can assume $d_{121} > d_{111}$ and have the two integer points $v(\tau_i)' = (0,0,3d_{011}+ 4d_{1i1} + 6d_{211},\imath/2)$ inside $A_X^c(\lambda)$, forcing $\imath=2$ and by the same admissible operations as in Case 1, we get $d_{011}=d_{111}=0$, $d_{211}=1$, $v(\tau_1)' = (0,0,6,1)$ and $v(\tau_2)' = (0,0,6+4d_{121},1)$. As we have seen in Case 1, the point $(-1,-1,2,1)$ lies inside $A_X^c(\lambda_0)$.

\medskip

\noindent
\emph{Case 3.2: $t=2$}. Here both $v(\tau_1)'$ and $v(\tau_2)'=(0,0,d_{011}+ 2(d_{121} + d_{211}),\imath/2)$ are integer, forcing $\imath=2$, and by the same admissible operations as in Case 1, we get $d_{011}=d_{111}=0$, $d_{211}=1$, $v(\tau_1)' = (0,0,6,1)$ and $v(\tau_2)' = (0,0,2(1+d_{121}),1)$. As we have seen in Case 1, one of the points $(-1,-1,1,1)$ and $(-1,-1,2,1)$ lies inside $A_X^c(\lambda_0)$, since either $2(1+d_{121})\leq 4$ or $2(1+d_{121})\geq 8$.

\medskip

\noindent
\emph{Case 3.3: $t=1$}. By subtracting $d_{121}$ times the first from the third row, we achieve $d_{121}=0$. Furthermore, either $\imath \equiv 2 \mod 4$, then due to primitivity of $v_{21}$, the entry $d_{211}$ is odd and by adding appropriate multiples of the second to the third row, we achieve $d_{211}=1$, or $\imath \equiv 0 \mod 4$, then $\imath/2-1$ is odd and by if necessary adding the last and then adding appropriate multiples of the second to the third row, we achieve $d_{211}=1$ as well. So $v(\tau_1)' = (0,0,3d_{011}+ 4d_{111} + 6,\imath/2)$ and $v(\tau_2)'=(0,0,(d_{011}+  2)/3,\imath/2)$. 

\smallskip

\noindent
\emph{Case 3.3.1: $\imath>2$}. Here Lemma~\ref{le:twopoints} forces
$$
3d_{011}+ 4d_{111} + 6 \in \{d_{011}/3, (d_{011} +  1)/3, d_{011}/3 +1, (d_{011}+  4)/3 \}.
$$
But if $3d_{011}+ 4d_{111} + 6  = d_{011}/3$, then $d_{011}=3k$ for some $k \in \ZZ$ and $4k +2d_{111} + 3 =0$ follows, a contradiction. If $3d_{011}+ 4d_{111} + 6  = d_{011}/3 +1$, then $8k +4d_{111} + 5 =0$ follows, a contradiction as well. For $3d_{011}+ 4d_{111} + 6 = (d_{011} +  1)/3$, we have $d_{011}=3k-1$ for some $k \in \ZZ$ and $8k +4d_{111} + 3 =0$ follows, while for $3d_{011}+ 4d_{111} + 6 = (d_{011} +  4)/3$, we have $d_{011}=3k-1$ for some $k \in \ZZ$ as well and $4k +2d_{111} + 1 =0$ follows, both contradictions.

\smallskip

\noindent
\emph{Case 3.3.2: $\imath=2$}. Here, our matrix has the form
$$
{\tiny
\begin{bmatrix}
-4 & 3 & 1 & 0   \\
-4 & 0 & 0  & 2  \\
0 & 0 & d_{121} & 1\\
3 & -2 & 0 & 0
\end{bmatrix}
}
$$
and we have $v(\tau_1)' = (0,0,6,1)$ and $v(\tau_2)'=(0,0,(2+4d_{121})/3,1)$. As we have seen in Case 1, in order not to have $(-1,-1,1,1)$ or $(-1,-1,2,1)$ contained in $A_X^c(\lambda_0)$, we require $4<(2+4d_{121})/3<8$, which leads to $d_{121} \in \{3,5\}$, as $d_{121}=4$ is impossible. But with the same admissible operations as in Case 1, these two are equivalent and give a canonical singularity with matrix
$$
P_7:=
{\tiny
\begin{bmatrix}
-4 & 3 & 1 & 0   \\
-4 & 0 & 0  & 2  \\
0 & 0 & 3 & 1\\
3 & -2 & 0 & 0
\end{bmatrix}.
}
$$

\bigskip

\noindent
\emph{Case 4: additional column $v_{22}=(0,t,d_{221},(\imath-t)/2)$ in $\lambda_2$}. We can assume $d_{021} > d_{011}$ and have the two integer points $v(\tau_i)' = (0,0,3d_{011}+ 4d_{111} + 6d_{2i1},\imath/2)$ inside $A_X^c(\lambda)$, forcing $\imath=2$ and by the same admissible operations as in Case 1, we get $d_{011}=d_{111}=0$, $d_{211}=1$, $v(\tau_1)' = (0,0,6,1)$ and $v(\tau_2)' = (0,0,6+6d_{221},1)$. As we have seen in Case 1, the point $(-1,-1,2,1)$ lies inside $A_X^c(\lambda_0)$.

We finally have to check if the possible additional columns can be combined.
Such a matrix must be of one of the following three forms.
\smallskip

\noindent
\emph{Case a:}
$$
{\tiny
 \begin{bmatrix}
-4 & 3 & 1 & 0 & 0   \\
-4 & 0 & 0 & 2  & 0  \\
0 & 0 & 3 & 1 & d_{1} \\
3 & -2 & 0 & 0 & 1
\end{bmatrix},
}
$$
with $d_{1} \in \{5,7\}$. But for $d_{1}=5$, the column $v_{1}$ would lie inside $\cone(v_{01},v_{12},v_{21})$. The Case $d_{1}=7$ gives a canonical singularity with matrix $P_8$.
\smallskip

\noindent
\emph{Case b:}
$$
{\tiny
 \begin{bmatrix}
-4 & 3  & 0 & 0 & 0  \\
-4 & 0  & 2  & 0 & 0  \\
0 & 0  & 1 & 5 & 7\\
3 & -2  & 0 & 1 & 1
\end{bmatrix},
}
$$
which gives a canonical singularity with matrix $P_9$.
\smallskip

\noindent
\emph{Case c:}
$$
{\tiny
 \begin{bmatrix}
-4 & 3 & 1 & 1 & 0    \\
-4 & 0 & 0 & 0 & 2    \\
0 & 0 & 3  & 5 & 1  \\
3 & -2 & 0 & 0 & 0 
\end{bmatrix},
}
$$
which gives the last canonical singularity with matrix $P_{10}$.

\end{proof}

We come to Case $(ii)$ of Proposition~\ref{prop:zetaexcep} (iii).

\begin{proposition}
\label{prop:zeta=3}
Let $X$ be a non-toric threefold singularity of complexity one.
Assume that~$X$ is of canonical multiplicity 
three, Gorenstein index $\imath \geq 2$, the leading platonic triple is $(3,3,2)$ and $X$ is at most canonical. 
Then $X$  is isomorphic to $X(P_i)$,
where $P_i$ is one of the following matrices:
 $$
 {\tiny
 P_{11} = 
 \begin{bmatrix}
-3 & 3 & 0 & 0   \\
-3 & 0 & 2  & 0  \\
1 & 0 & 0 & 1 \\
0 & 2 & -1 & 1
\end{bmatrix}
\quad
P_{12} =
\begin{bmatrix}
-3 & 3 & 0 & 0   \\
-3 & 0 & 2  & 1  \\
1 & 0 & 0 & 0 \\
0 & 2 & -1 & 0
\end{bmatrix}
}
$$
\end{proposition}

\begin{proof}
Our matrix $P$ falls under Case $(ii)$ of Proposition~\ref{prop:zetaexcep} (iii), thus $\imath$ is a multiple of $3$ and the leading block has the shape
$$
{\tiny
\begin{bmatrix}
-3 & 3  & 0 & 0 & \dots & 0 
\\
-3 & 0 & 2 & 0 & \dots & 0 
\\
-3 & 0 & 0 & 1 & \dots & 0 
\\
\vdots & \vdots & \vdots & \vdots & \ddots & \vdots 
\\
-3 & 0 & 0 &  0 &  \dots &  1 
\\
d_{011} & d_{111} & d_{211} & 0 & \ldots & 0 
\\
\frac{\imath}{3}-1 & \frac{\imath}{3}+1 & -\frac{\imath}{3} & 0 & \ldots & 0 
\end{bmatrix}.
}
$$
If $r\geq 3$, due to irredundancy we have at least another column $v_{32}=e_{3}+d_{321}e_{r+1}$ in $P$ with $d_{321} \geq 1$. So $A_X^c(\lambda)$ contains the integer points
\begin{align*}
v(\tau_1)' &= (0,\ldots,0,2(d_{011}+ d_{111}) + 3 d_{211},\imath/3), \\
v(\tau_2)' &= (0,\ldots,0,2(d_{011}+ d_{111}) + 3 d_{211}  + 6d_{321},\imath/3)
\end{align*}
and with Lemma~\ref{le:twopoints}, only $\imath=3$ is possible.
By primitivity of the first column and by if necessary negating the penultimate row, we can assume $d_{011}=3k+1$ for some $k \in \ZZ$.
By adding the $(2d_{211}+d_{111}+4k)$-fold of the first to the penultimate and subtracting the $(2d_{211}+d_{111}+3k)$-fold of the second and the $(3d_{211}+2d_{111}+6k)$-fold of the last to the penultimate row, we achieve $d_{011}=1$, $d_{111}=d_{211}=0$ and $v(\tau_1)' = (0,\ldots,0,2,1)$ and $v(\tau_2)' = (0,\ldots,0,2 + 6d_{321},\imath/3)$. But for $d_{321}\leq 0$, the point $(1,0,0,1)$ and for $d_{321}\geq 2$, the point $(1,0,1,1)$ lies inside $A_X^c(\lambda_1)$, while $d_{321}=1$ is not possible, since then $v_{32}$ would lie inside $\tau_1$. Thus only $r=2$ is possible and the leading block is
$$
{\tiny
\begin{bmatrix}
-3 & 3 & 0  \\
-3 & 0 & 2  \\
d_{011} & d_{111} & d_{211} \\
\frac{\imath}{3}-1 & \frac{\imath}{3}+1 & -\frac{\imath}{3} 
\end{bmatrix}.
}
$$
Now we have to check which additional columns are possible. 

\bigskip

\noindent
\emph{Case 1: additional column $v_{1}=(0,0,d_1,\imath/3)$ in $\lambda$}.
Since $v(\tau_1)'$ is integer, an additional column in the lineality part means two integer points in $A_X^c(\lambda)$, which forces $\imath=3$. The leading block together with the additional column has the form
$$
{\tiny
\begin{bmatrix}
-3 & 3 & 0 & 0   \\
-3 & 0 & 2  & 0  \\
d_{011} & d_{111} & d_{211} & d_{1} \\
0 & 2 & -1 & 1
\end{bmatrix}.
}
$$
By primitivity of the first column and by if necessary negating the third row, we can assume $d_{011}=3k+1$ for some $k \in \ZZ$.
By adding the $(2d_{211}+d_{111}+4k)$-fold of the first to the third and subtracting the $(2d_{211}+d_{111}+3k)$-fold of the second and the $(3d_{211}+2d_{111}+6k)$-fold of the last to the third row, we achieve $d_{011}=1$, $d_{111}=d_{211}=0$ and $v(\tau_1)' = (0,0,2,1)$. Thus $d_1=2$ is impossible and the following cases remain. 

\medskip

\noindent
\emph{Case 1.1: $d_1\leq 0$}. Then inside $A_X^c(\lambda_1)$ lies the point
$$
(1,0,0,1)=\frac{1}{3}v_{01}+\frac{d_1}{3(d_1-2)}v(\tau_1)'-\frac{2}{3(d_1-2)}v_1.
$$ 

\medskip

\noindent
\emph{Case 1.2: $d_1\geq 3$}. Then inside $A_X^c(\lambda_1)$ lies the point
$$
(1,0,1,1)=\frac{1}{3}v_{01}+\frac{(d_1-3)}{3(d_1-2)}v(\tau_1)'+\frac{1}{3(d_1-2)}v_1.
$$

\medskip

\noindent
\emph{Case 1.3: $d_1=1$}. Since $A_X^c$ is canonical in this case, we get a canonical singularity from the matrix
$$
P_{11}:=
{\tiny
\begin{bmatrix}
-3 & 3 & 0 & 0   \\
-3 & 0 & 2  & 0  \\
1 & 0 & 0 & 1 \\
0 & 2 & -1 & 1
\end{bmatrix}
}.
$$

\bigskip

\noindent
\emph{Case 2: additional column $v_{02}=(-t,-t,d_{021},(\imath-t)/3)$ in $\lambda_0$}. Since $v_{02}$ must be integer, we have $t=3$ and can assume $d_{021}>d_{011}$. Thus $A_X^c(\lambda)$ contains the two integer points $ v(\tau_i)' = (0,0,2(d_{0i1}+ d_{111}) + 3 d_{211},\imath/3)$ with $i=1,2$. As in Case 1, this forces $\imath=3$. By the same admissible operations as in Case 1, we can now achieve $d_{011}=1$, $d_{111}=d_{211}=0$, $v(\tau_1)' = (0,0,2,1)$ and $v(\tau_2)' = (0,0,2d_{021},1)$. But for $d_{021}\leq 0$, the point $(1,0,0,1)$ and for $d_{021}\geq 2$, the point $(1,0,1,1)$ lies inside $A_X^c(\lambda_1)$, as Case 1 shows.

\bigskip

\noindent
\emph{Case 3: additional column $v_{12}=(t,0,d_{121},(\imath+t)/3)$ in $\lambda_1$}. The proceeding is exactly the same as in Case 2.

\bigskip

\noindent
\emph{Case 4: additional column $v_{22}=(0,t,d_{221},\imath(1-t)/3)$ in $\lambda_2$}. Here we have the possibilities $t=1,2$.

\medskip

\noindent
\emph{Case 4.1: $t=2$}. Here we can assume $d_{221}>d_{211}$ and $A_X^c(\lambda)$ contains the two integer points $v(\tau_i)' = (0,0,2(d_{011}+ d_{111}) + 3 d_{2i1},\imath/3)$. This again forces $\imath=3$ and the proceeding is again the same as in the previous cases.

\medskip

\noindent
\emph{Case 4.2: $t=1$}. We have the two integer points $v(\tau_1)' = (0,0,2(d_{011}+ d_{111}) + 3 d_{211},\imath/3)$ and 
$v(\tau_2)' = (0,0,d_{011}+ d_{111} + 3 d_{221},\imath/3)$ inside $A_X^c(\lambda)$, again forcing $\imath=3$ and we thus achieve $d_{011}=1$, $d_{111}=d_{211}=0$, $v(\tau_1)' = (0,0,2,1)$ and $v(\tau_2)' = (0,0,1+3d_{021},1)$. So regarding $A_X^c(\lambda_1)$, we see that the only possibility is $d_{021}=0$, giving a canonical singularity with matrix
$$
P_{12}:=
{\tiny
\begin{bmatrix}
-3 & 3 & 0 & 0   \\
-3 & 0 & 2  & 1  \\
1 & 0 & 0 & 0 \\
0 & 2 & -1 & 0
\end{bmatrix}.
}
$$
The two matrices we found cannot be combined, since $(0,0,1,1)$ lies inside
$$
\cone((-3,-3,1,0),(3,0,0,2),(0,1,0,0)).
$$

\end{proof}

We proceed with  Case $(iii)$ of Proposition~\ref{prop:zetaexcep} (iii).

\begin{proposition}
\label{prop:zeta=4}
Let $X$ be a non-toric threefold singularity of complexity one.
Assume that~$X$ is of canonical multiplicity 
four, Gorenstein index $\imath \geq 2$, the leading platonic triple is $(l_{01},2,2)$ with $l_{01}$ odd and $X$ is at most canonical. 
Then $X$ is isomorphic to $X(P_{i})$,
where $P_{i}$ is one of the matrices
\setlength{\arraycolsep}{1mm}
 $$
 {\tiny
 P_{13}
=
\begin{bmatrix}
-(2\mathbf{k}+1) & 2   & 0  & 0 & 0 & \dots & 0 &0
\\
-(2\mathbf{k}+1) & 0   & 2  & 0 &0& \dots & 0 & 0
\\
-(2\mathbf{k}+1) & 0  & 0  & 1 &1 & \dots & 0 &0
\\
\vdots & \vdots  & \vdots  & \vdots & \vdots & \ddots & \vdots & \vdots
\\
-(2\mathbf{k}+1) & 0   & 0 & 0 & 0 & \dots &  1 & 1
\\
\mathbf{d_{0}} & 1  & 1  & 0 & d_{3} & \ldots & 0 & d_{r}
\\
\mathbf{k} & 0  & 1  & 0 & 0  & \ldots & 0 & 0
\end{bmatrix}
\quad
 P_{14}
=
\begin{bmatrix}
-(2\mathbf{k}+1) & 2   & 0 & 0 & 0 & 0 & \dots & 0 &0
\\
-(2\mathbf{k}+1) & 0   & 2 & 2 & 0 &0& \dots & 0 & 0
\\
-(2\mathbf{k}+1) & 0  & 0 & 0 & 1 &1 & \dots & 0 &0
\\
\vdots & \vdots  & \vdots & \vdots & \vdots & \vdots & \ddots & \vdots & \vdots
\\
-(2\mathbf{k}+1) & 0  &  0 & 0 & 0 & 0 & \dots &  1 & 1
\\
\mathbf{d_{0}} & 1  & 1 & d_{2} & 0 & d_{3} & \ldots & 0 & d_{r}
\\
\mathbf{k} & 0  & 1 & 1 & 0 & 0  & \ldots & 0 & 0
\end{bmatrix}
}
$$
In both cases $r\geq 2$ and  $\mathbf{k},\mathbf{d_{0}} \in \ZZ^t$ for some $t \in \ZZ_{\geq 1}$.
\end{proposition}

\begin{proof}
Our matrix $P$ falls under Case $(iii)$ of Proposition~\ref{prop:zetaexcep} (iii), thus $\imath \equiv 2 \mod 4$ and the leading block has the shape
$$
{\tiny
\begin{bmatrix}
-(2k+1) & 2  & 0 & 0 & \dots & 0 
\\
-(2k+1) & 0 & 2 & 0 & \dots & 0 
\\
-(2k+1) & 0 & 0 & 1 & \dots & 0 
\\
\vdots & \vdots & \vdots & \vdots & \ddots & \vdots 
\\
-(2k+1) & 0 & 0 &  0 &  \dots &  1 
\\
d_{011} & d_{111} & d_{211} & 0 & \ldots & 0 
\\
-k\imath/2 & (\imath-2)/4 & (\imath+2)/4 & 0 & \ldots & 0 
\end{bmatrix}.
}
$$
One of $(\imath-2)/4$ and $(\imath+2)/4$ is even, while the other is odd. So if both $d_{111}$ and $d_{211}$ are odd, adding appropriate multiples of the first two rows to the penultimate gives $d_{111}=d_{211}=1$. We achieve the same if one is odd and the other even by before adding the last to the penultimate row. So we proceed with $d_{111}=d_{211}=1$.

\bigskip

\noindent
\emph{Case 1: $r\geq 3$}. Due to irredundancy we have at least another column $v_{32}=e_{3}+d_{321}e_{r+1}$ in $P$ with $d_{321} > 0$. So $A_X^c(\lambda)$ contains the points
\begin{align*}
v(\tau_1)' &= \left(0,\ldots,0,d_{011} + 2k+1,\imath/4\right), \\
v(\tau_2)' &= \left(0,\ldots,0,d_{011} + 2k+1 + (2k+1)d_{321},\imath/4\right).
\end{align*}
Applying Lemma~\ref{le:halfheight}, we get $\imath=2$.
 In this case, all $A_X^c(\lambda_i)$ but for $i=1,2$ have height one. But the only possible integer points in the relative interiors of $A_X^c(\lambda_1)$, $A_X^c(\lambda_2)$ are those of the form $(1,0,\ldots,0,k,0)$. Thus in every leaf but $A_X^c(\lambda_1)$ arbitrary columns can be added to the leading block to get a canonical singularity - if the resulting matrix is an appropriate defining matrix $P$. We write this as the two series $P_{13}$ and $P_{14}$ from above.

\bigskip

\noindent
\emph{Case 2: $r=2$}. If $\imath=2$, all $\partial A_X^c(\lambda_i)$ have height one as above, we are in the same situation as in Case 1 and the resulting matrix is the case $r=2$ of the series $P_{13}$ and $P_{14}$. So we can assume $\imath \geq 6$. Our leading block is
$$
{\tiny
\begin{bmatrix}
-(2k+1)  & 2 & 0   \\
-(2k+1)  & 0  & 2  \\
d_{011} & 1 & 1 \\
-k\imath/2 & (\imath-2)/4 & (\imath+2)/4
\end{bmatrix}
}
$$
and $v(\tau_1)' = \left(0,0,d_{011} + 2k+1,\imath/4\right)$. But for any allowed additional column, we get $v(\tau_2)' = \left(0,0,\kappa,\imath/4\right)$ with $\kappa \in \ZZ$. Applying Lemma~\ref{le:halfheight} again, we get $\imath=2$, a contradiction.
\end{proof}

The next Proposition deals with Case $(iv)$ of Proposition~\ref{prop:zetaexcep} (iii).

\begin{proposition}
\label{prop:zeta=2(2k,2,2)}
Let $X$ be a non-toric threefold singularity of complexity one.
Assume that~$X$ is of canonical multiplicity 
two, even Gorenstein index, the leading platonic triple is $(2k,2,2)$ and $X$ is at most canonical. 
Then $X$  is isomorphic to $X(P_{i})$,
where $P_{i}$ is one of the matrices
 $$
 {\tiny
 P_{15}
 =
 \begin{bmatrix}
-2k  & 2 & 0 & 0  \\
-2k  & 0  & 2 & 0 \\
1 & 1 & 1 & 2k+2\\
1-k & 2 & -1 & 1
\end{bmatrix}
\,
P_{16}
=
 \begin{bmatrix}
-2k  & 2 & 0 & 0 & 0  \\
-2k  & 0  & 2 & 0 & 0 \\
1 & 1 & 1 & 2k & 2k+2\\
1-k & 2 & -1 & 1 & 1
\end{bmatrix}
}
$$
$$
{\tiny
P_{17}
=
\begin{bmatrix}
-4  & 2 & 0 & 0  \\
-4  & 0  & 2 & 1 \\
1 & 1 & 1 & 4 \\
-1 & 2 & -1 & 0
\end{bmatrix}
\,
P_{18}
=
\begin{bmatrix}
-4  & 2 & 0 & 0 & 0 \\
-4  & 0  & 2 & 1 & 1\\
1 & 1 & 1 & 2 & 4 \\
-1 & 2 & -1 & 0 & 0
\end{bmatrix}
\,
P_{19}
=
\begin{bmatrix}
-4  & 2 & 0 & 0 & 0 \\
-4  & 0  & 2 & 1 & 0\\
1 & 1 & 1 & 4 & 4 \\
-1 & 2 & -1 & 0 & 1
\end{bmatrix}
}
 $$
\end{proposition}

\begin{proof}
Our matrix $P$ falls under Case $(iv)$ of Proposition~\ref{prop:zetaexcep} (iii), thus $\imath$ is even and the leading block has the shape
$$
{\tiny
\begin{bmatrix}
-2k & 2  & 0 & 0 & \dots & 0 
\\
-2k & 0 & 2 & 0 & \dots & 0 
\\
-2k & 0 & 0 & 1 & \dots & 0 
\\
\vdots & \vdots & \vdots & \vdots & \ddots & \vdots 
\\
-2k & 0 & 0 &  0 &  \dots &  1 
\\
d_{011} & d_{111} & d_{211} & 0 & \ldots & 0 
\\
\imath/2-k & \imath/2+1 & -\imath/2 & 0 & \ldots & 0 
\end{bmatrix}
}
$$
with $k\geq 2$ and we have  $v(\tau_1)' = \left(0,\ldots,0,d_{011} + k(d_{111}+d_{211}),\imath/2\right)$.
If now $r\geq 3$, we have at least another column $v_{321}$ so that $v(\tau_2)' = v(\tau_1)'+2kd_{321}e_{r+1}$ and we can assume $d_{321}\geq 1$ as before. Since both $v(\tau_i)'$ are integer, this forces $\imath=2$. Now $d_{111}$ must be odd, so by if necessary adding the last to the penultimate and appropriate multiples of the first two to the penultimate row we achieve $d_{111}=d_{211}=1$. Now by adding  $-2\lfloor d_{011}/2 \rfloor$ times the first, $\lfloor d_{011}/2 \rfloor$ times the second and $2\lfloor d_{011}/2 \rfloor$ times the last to the penultimate row, we achieve $d_{011} \in \{0,1\}$. Assume $d_{011}=0$, then
$$
(-1,-1,k-1,0)= \frac{1}{2k}v_{01}+\frac{k-1}{2k}v(\tau_1)'
$$
lies inside $A_X^c(\lambda_0)$. But if $d_{011}=1$, then since $k\geq 2$ inside $A_X^c(\lambda_0)$ lies
$$
(-1,-1,k,0)= \frac{1}{2k}v_{01}+\frac{2d_{321}k+2d_{321}-1}{4kd_{321}}v(\tau_1)'+\frac{1}{4kd_{321}}v(\tau_2)'.
$$
So we can assume $r=2$. We achieve $d_{111}=d_{211}=1$ as before. Additional columns in $\lambda$, $\lambda_0$, $\lambda_1$ as well as in $\lambda_2$ with $l_{22}=2$ lead to two integer points  in $A_X^c(\lambda)$ and thus to $\imath=2$. We treat this in the following first case while in the second case we consider the case of an additional column in $\lambda_2$ with $l_{22}=1$.

\bigskip

\noindent
\emph{Case 1}. We have $\imath=2$ and achieve $d_{011}=1$. So $v(\tau_1)' = \left(0,0,1 + 2k,1\right)$. Now if we have $i=1,2$ and an additional column $v_{i2}$ in $\lambda_i$ with $l_{i2}=2$, then we can assume $d_{i21}\geq 2$ and thus $v(\tau_2)' = v(\tau_1)'+k(d_{i21}-1)e_3$. The point
$$
(-1,-1,k,0)= \frac{1}{2k}v_{01}+\frac{k(k-1)(d_{i21}-1)-k}{2k^2(d_{i21}-1)}v(\tau_1)'+\frac{1}{2k(d_{i21}-1)}v(\tau_2)'
$$
lies inside $A_X^c(\lambda_0)$.
For an additional column $v_{02}=(-2k',-2k',d_{021},1-k')$ in $\lambda_0$, where we can assume $d_{021}$ odd and $d_{021}+ 2k' \geq 3+2k$, again $(-1,-1,k,0)$ lies inside. Now assume an additional column $v_1=(0,0,d_1,1)$ in $\lambda$. We can assume $d_1>1+2k$ for if not, i.e. $d_1 \leq 1+2k$, adding the $1+4k$-fold of the first, the $-2(k+1)$-fold of the second and the $-2(2k+1)$-fold of the last to the third row and negating the third row afterwards gives $d_1>1+2k$. Then the point 
$$
(-1,-1,k,0)= \frac{1}{2k}v_{01}+\frac{d_1k-2k^2-d_1+1}{2k(d_1-2k-1)}v(\tau_1)'+\frac{1}{2(d_1-2k-1)}v_1
$$
lies inside $A_X^c(\lambda_0)$ for $d_1\geq 3+2k$. Thus we require $d_1=2+2k$ and get the matrix
$$
P_{15}:=
{\tiny
\begin{bmatrix}
-2k  & 2 & 0 & 0  \\
-2k  & 0  & 2 & 0 \\
1 & 1 & 1 & 2k+2\\
1-k & 2 & -1 & 1
\end{bmatrix},
}
$$
giving a canonical singularity.

\bigskip

\noindent
\emph{Case 2: additional column $v_{22}=(0,1,d_{221},0)$ in $\lambda_2$}. The leading block together with the additional column has the form
$$
{\tiny
\begin{bmatrix}
-2k  & 2 & 0 & 0  \\
-2k  & 0  & 2 & 1 \\
d_{011} & 1 & 1 & d_{221} \\
\imath/2-k & \imath/2+1 & -\imath/2 & 0
\end{bmatrix}
}
$$
and we have $v(\tau_1)' = \left(0,0,d_{011} + 2k,\imath/2\right)$. We write $d_{221}=(d_{011} + 2k+1)/2+s$ for suitable $s$ and get
$$
v(\tau_2)' = \left(0,0,d_{011} + 2k +2ks/(k+1),\imath/2\right).
$$
Then $|v(\tau_1)'-v(\tau_2)'| < 1$ if $|s| < 1/2 +1/(2k)$. Since $k\geq 2$, this is only possible for $d_{011}$ even, $s= \pm 1/2$. We examine these cases and the case $|v(\tau_1)'-v(\tau_2)'| \geq 1$ afterwards.

\medskip

\noindent
\emph{Case 2.1: $d_{011}$ even, $s=1/2$}. Here we have $|v(\tau_1)'-v(\tau_2)'|=k/(k+1)\geq 1/2$. 
With Corollary~\ref{corr:halfcone}, this forces $(\imath/2) | (d_{011} + 2k +1)$ and thus $\imath \equiv 2 \mod 4$ and $k$ even. We achieve $d_{011}=\imath/2-2k-1$ by adding the $(\imath+2)(2d_{011}+4k+2-\imath)/(4\imath)$-fold of the first, the $(\imath -2d_{011}-4k-2)/4$-fold of the second and the $(\imath-2d_{011}-4k-2)/\imath$-fold of the last to the third row. But then inside $A_X^c(\lambda_0)$ lies the point 
$$
(-1,-1,(\imath-6)/4,(\imath-2)/4)=\frac{1}{2k}v_{01}+\frac{k-1}{2k}v(\tau_1)'.
$$

\medskip

\noindent
\emph{Case 2.2: $d_{011}$ even, $s=-1/2$}. Here we have $|v(\tau_1)'-v(\tau_2)'|=k/(k+1)\geq 1/2$. 
With Corollary~\ref{corr:halfcone}, this forces $(\imath/2) | (d_{011} + 2k -1)$ and thus $\imath \equiv 2 \mod 4$ and $k$ even. We achieve $d_{011}=\imath/2-2k+1$ by adding the $(\imath+2)(2d_{011}+4k-2-\imath)/(4\imath)$-fold of the first, the $(\imath -2d_{011}-4k+2)/4$-fold of the second and the $(\imath-2d_{011}-4k+2)/\imath$-fold of the last to the third row. But then the point 
$$
(-1,-1,(\imath-2)/4,(\imath-2)/4)=\frac{1}{2k}v_{01}+\frac{k-1}{2k}v(\tau_1)'
$$
lies inside $A_X^c(\lambda_0)$.

\medskip

\noindent
\emph{Case 2.3: $|v(\tau_1)'-v(\tau_2)'| \geq 1$}. This leads to $\imath=2$ and we can assume $d_{011}=1$ as we have seen above. Now $s$ is integer and we can assume $s\geq 1$ by the same admissible operation as in Case 1. Then $|v(\tau_1)'-v(\tau_2)'|=2sk/(k+1) \geq 2$ if $s\geq 2 $ and this gives the interior point $(-1,-1,k,0)$ in $A_X^c(\lambda_0)$, as we also have seen in Case 1. So $s=1$ remains. 
The point
$$
(-1,-1,k,0)= \frac{1}{2k}v_{01}+\frac{k-3}{4k}v(\tau_1)'+\frac{k+1}{4k}v(\tau_2)'
$$
lies inside $A_X^c(\lambda_0)$ for $k\geq 3$. So we get a canonical singularity from the matrix
$$
P_{17}:=
{\tiny
\begin{bmatrix}
-4  & 2 & 0 & 0  \\
-4  & 0  & 2 & 1 \\
1 & 1 & 1 & 4 \\
-1 & 2 & -1 & 0
\end{bmatrix}.
}
$$

Finally, we have to check which of the additional columns can be combined. On the one hand, we have the matrix
$
P_{15}$
where in the last column, we can also write $2k$ instead of $2k+2$, see Case 1. Thus we can also add \emph{both} columns to the leading block and get another canonical singularity with matrix $P_{16}$. On the other hand, we have the matrix
$
P_{17}$
where we can take $(0,1,2,0)^\top$ instead of the last column and get the same singularity. Thus here as well we can take both of these columns in addition to the leading block and get the matrix $P_{18}$. Now as the leading block of $P_{17}$ is the special case $k=2$ of the leading block of $P_{15}$, we also have to check if we can add columns $(0,0,4,1)^\top$ and $(0,0,6,1)^\top$. But $(0,0,6,1)^\top$ lies inside $\cone(v_{01},v_{11},v_{22})$, while $(0,0,4,1)^\top$ can be added to achieve the matrix 
$$
P_{19}:=
{\tiny
\begin{bmatrix}
-4  & 2 & 0 & 0 & 0 \\
-4  & 0  & 2 & 1 & 0 \\
1 & 1 & 1 & 4 & 4\\
-1 & 2 & -1 & 0 & 1
\end{bmatrix},
}
$$
giving a canonical singularity. To this matrix, the column $(0,1,2,0)^\top$ can not be added, since then $(0,0,4,1)^\top$ would lie inside $\cone(v_{01},v_{11},(0,1,2,0)^\top)$.

\end{proof}

We come to Case $(v)$ of Proposition~\ref{prop:zetaexcep} (iii).

\begin{proposition}
\label{prop:zeta=2(k,2,2)}
Let $X$ be a non-toric threefold singularity of complexity one.
Assume that~$X$ is of canonical multiplicity 
two, even Gorenstein index, the leading platonic triple is $(k,2,2)$ and $X$ is at most canonical. 
Then $X$ is isomorphic to $X(P_{i})$,
where $P_{i}$ is one of the matrices
 $$
 {\tiny
 P_{20}
 =
 \begin{bmatrix}
-\mathbf{k} & 2   & 0  & 0 & 0 & \dots & 0 & 0 & 0& 0
\\
-\mathbf{k}  & 0   & 2  & 0 &0& \dots & 0 & 0 & 0 & 0
\\
-\mathbf{k}  & 0  & 0  & 1 &1 & \dots & 0 & 0 & 0  & 0
\\
\vdots & \vdots  & \vdots  & \vdots & \vdots & \ddots & \vdots & \vdots & \vdots & \vdots
\\
-\mathbf{k}  & 0   & 0 & 0 & 0 & \dots &  1 & 1 & 0 & 0
\\
\mathbf{d_{0}} & 1  & 1  & 0 & d_{3} & \ldots & 0 & d_{r} & d'_{1} & d'_{2}
\\
1-\mathbf{k} & 0  & 2  & 0 & 0  & \ldots & 0 & 0 & 1 & 1
\end{bmatrix}
\enskip
P_{21}
=
\begin{bmatrix}
-(2k+1)  & 2 & 0 & 0   \\
-(2k+1)  & 0  & 2 & 0 \\
-k & 0 & 1  & 1 \\
-4k & 1 & 3 & 2
\end{bmatrix}
}
 $$
 One or both of the last two columns of $P_{20}$ may be omitted, if the general requirements on the matrix are still fulfilled.

\end{proposition}

\begin{proof}
Our matrix $P$ falls under Case $(v)$ of Proposition~\ref{prop:zetaexcep} (iii), thus $\imath$ is even and the leading block has the shape
$$
{\tiny
\begin{bmatrix}
-l_{01} & 2  & 0 & 0 & \dots & 0 
\\
-l_{01} & 0 & 2 & 0 & \dots & 0 
\\
-l_{01} & 0 & 0 & 1 & \dots & 0 
\\
\vdots & \vdots & \vdots & \vdots & \ddots & \vdots 
\\
-l_{01} & 0 & 0 &  0 &  \dots &  1 
\\
d_{011} & d_{111} & d_{211} & 0 & \ldots & 0 
\\
\imath/2(1-l_{01}) & \imath/2-1 & \imath/2+1 & 0 & \ldots & 0 
\end{bmatrix}
}
$$
with $l_{01}\in \ZZ_{\geq 2}$. Now we distinguish the case $\imath=2$ and $\imath\geq 4$.

\bigskip

\noindent
\emph{Case 1: $\imath=2$}.  Here we can assume $d_{111}=d_{211}=1$ by adding appropriate multiples of the first two rows to the penultimate and also $d_{011}=0$ by afterwards adding the $d_{011}$-fold of the second and the $-d_{011}$ fold of the last to the penultimate row. Now we see that $A_X^c(\lambda_0)$ and the $A_X^c(\lambda_i)$ for $i=3,\ldots,r$  have height one and are thus canonical for any allowed additional column. In $A_X^c(\lambda_1)$, $A_X^c(\lambda_2)$, the points that may be inside are of the forms $(1,0,\ldots,0,k,0)$ and $(0,1,0,\ldots,0,k,0)$. So in these leaves, no additional columns are allowed. We denote the resulting series by 
$$
{\tiny
\begin{bmatrix}
-\mathbf{k} & 2   & 0  & 0 & 0 & \dots & 0 & 0 & 0& 0
\\
-\mathbf{k}  & 0   & 2  & 0 &0& \dots & 0 & 0 & 0 & 0
\\
-\mathbf{k}  & 0  & 0  & 1 &1 & \dots & 0 & 0 & 0  & 0
\\
\vdots & \vdots  & \vdots  & \vdots & \vdots & \ddots & \vdots & \vdots & \vdots & \vdots
\\
-\mathbf{k}  & 0   & 0 & 0 & 0 & \dots &  1 & 1 & 0 & 0
\\
\mathbf{d_{0}} & 1  & 1  & 0 & d_{3} & \ldots & 0 & d_{r} & d'_{1} & d'_{2}
\\
1-\mathbf{k} & 0  & 2  & 0 & 0  & \ldots & 0 & 0 & 1 & 1
\end{bmatrix}.
}
$$

\bigskip

\noindent
\emph{Case 2: $\imath\geq 4$}. Here $r\geq 3$ is impossible, since then we would have at least an additional column $v_{321}$ and would get $|v(\tau_1)'-v(\tau_2)'|\geq 2$, so that $A_X^c(\lambda)$ contains two integer points forcing $\imath=2$, a contradiction. So we have $r=2$. We have
$$
v(\tau_1)' = (0,0,d_{011}+l_{01}(d_{111}+d_{211})/2,\imath/2)
$$
and distinguish three subcases in the following.
\medskip

\noindent
\emph{Case 2.1: $\imath \equiv 2 \mod 4$}. Here both $d_{111}$ and $d_{211}$ must be odd and adding appropriate multiples of the first two rows to the third, we achieve $d_{111}=d_{211}=1$. Thus $v(\tau_1)'$ is integer and no second integer point is allowed in $A_X^c(\lambda)$. Since for any allowed additional column such a second integer point will appear, we get no canonical singularity in this case.

\medskip

\noindent
\emph{Case 2.2: $\imath \equiv 0 \mod 4$, $d_{111} \equiv d_{211} \mod 2$}. By if necessary adding the last to the third and afterwards adding appropriate multiples of the first two rows to the third, we achieve $d_{111}=d_{211}=0$, i.e. a leading block of the form
$$
{\tiny
\begin{bmatrix}
-l_{01}  & 2 & 0   \\
-l_{01}  & 0  & 2  \\
d_{011} & 0 & 0 \\
\imath/2(1-l_{01}) & \imath/2-1 & \imath/2+1
\end{bmatrix}.
}
$$
Again $v(\tau_1)'$ is integer and the only additional columns that do not give a second integer point in $A_X^c(\lambda)$ are those of the form $(2,0,d_{121},\imath/2-1)$ and $(0,2,d_{221},\imath/2-1)$ in $\lambda_1$ and $\lambda_2$ respectively with $d_{i21}$ odd. Both leaves are equivalent by adding the first and subtracting the second from the last row, so we consider additional columns $(2,0,d_{121},\imath/2-1)$. For such a column, we have
$$
v(\tau_2)' = (0,0,d_{011}+l_{01}d_{121}/2,\imath/2),
$$
so we require $l_{01}d_{121}/2=\pm 1/2$, which is impossible due to $l_{01}\geq 2$.

\medskip

\noindent
\emph{Case 2.3: $\imath \equiv 0 \mod 4$, $d_{111} \equiv d_{211} +1 \mod 2$}. By the same proceeding as in the previous case, we get $d_{111}=0$, $d_{211}=1$ and $v(\tau_1)' = (0,0,d_{011}+l_{01}/2,\imath/2)$. Now in the following subcases, we check which additional columns are possible.

\smallskip

\noindent
\emph{Case 2.3.1: additional column $v_1=(0,0,d_1,\imath/2)$ in $\lambda$}. With Lemma~\ref{le:twopoints}, we get $l_{01}=2k+1$ odd for some $k \in \ZZ_{\geq 1}$ and  $d_1=d_{011}+(l_{01}\pm 1)/2$. From this we get that $\imath/2$ divides $d_{011}+(l_{01}\pm (-1))/2$.

\smallskip

\noindent
\emph{Case 2.3.1.1: $d_1=d_{011}+(l_{01}+1)/2$}. Here $\imath/2$ divides $d_{011}+k$. By adding the $1-\imath/2$-fold of the first, the $-(1+\imath/2)$-fold of the second and two times the last row to the third, we can add multiples of $\imath$ to $d_{011}$, so we can achieve one of the following cases:

\smallskip

\noindent
\emph{Case 2.3.1.1.1: $d_{011}=-k$}. Here if $\imath \geq 8$, the point
$$
(-1,-1,0,-1)=\frac{1}{2k+1}v_{01}+\frac{2(\imath k -4k-2)}{\imath (2k+1)}v(\tau_1)' + \frac{2}{\imath}v_1
$$
lies inside $A_X^c(\lambda_0)$, so we require $\imath=4$. We get the canonical matrix
$$
P_{21}:=
{\tiny
\begin{bmatrix}
-(2k+1)  & 2 & 0 & 0   \\
-(2k+1)  & 0  & 2 & 0 \\
-k & 0 & 1  & 1 \\
-4k & 1 & 3 & 2
\end{bmatrix}.
}
$$

\smallskip

\noindent
\emph{Case 2.3.1.1.2: $d_{011}=\imath/2-k$}. This is equivalent to Case 2.3.1.1.1 by adding the $\imath/4$-fold of the first two and the $-1$-fold of the last to the third row.

\smallskip

\noindent
\emph{Case 2.3.1.2: $d_1=d_{011}+(l_{01}-1)/2$}. Here $\imath/2$ divides $d_{011}+k+1$. By the same admissible operation as in Case 2.3.1.1, we can assume $d_{011}=-k-1$. But now this is equivalent to the Case  2.3.1.1 by subtracting the second from the third row and afterwards negating the third row.

\smallskip

\noindent
\emph{Case 2.3.2: additional column $v_{02}=(-l_{02},-l_{02},d_{021},\imath/2(1-l_{02}))$ in $\lambda_0$}. For such a column, we have $v(\tau_2)'=(0,0,d_{021}+l_{02}/2, \imath/2)$, which of course must not be equal to $v(\tau_1)'$. Thus we distinguish several subcases.

\smallskip

\noindent
\emph{Case 2.3.2.1: $l_{01} \equiv l_{02} \equiv 0 \mod 2$}. In this case, both $v(\tau_i)'$ are integer, which is a contradiction to $\imath \geq 4$ with Lemma~\ref{le:twopoints}.

\smallskip

\noindent
\emph{Case 2.3.2.2: $l_{01} \equiv l_{02} \equiv 1 \mod 2$}. Here, the third entry of both $v(\tau_i)'$ is in $\ZZ + 1/2$ and so, $A_X^c(\lambda)$ contains a polytope like in Corollary~\ref{corr:halfhalfcone}, forcing $\imath=4$ and - writing $l_{01}=2k_1+1$, $l_{02}=2k_2+1$ - we achieve $d_{011}=-k_1$ as in the previous case, forcing $d_{021}=1-k_2$. But for $k_1 \geq 2$, inside $A_X^c(\lambda_0)$ lies the point
$$
(-1,-1,0,-1)=\frac{1}{2k_1+1}v_{01}+\frac{2k_1-3}{4(2k_1+1)}v(\tau_1)'+\frac{1}{4}v(\tau_2)',
$$
while for $k_1=1$, we have $k_2 \in \{0,1\}$, and thus inside $A_X^c(\lambda_0)$ lies
$$
(-1,-1,0,-1)=\frac{3-2k_2}{12}v_{01}+\frac{1}{4}v_{02}+\frac{k_2}{6}v(\tau_1)'.
$$

\smallskip

\noindent
\emph{Case 2.3.2.3: $l_{01} \equiv l_{02} + 1 \equiv 1 \mod 2$}. Here we drop the condition $l_{01}\geq l_{02}$ coming from $(v_{01},v_{11},v_{21})$ being the leading block so that we do not have to examine a separate subcase. From Lemma~\ref{le:twopoints} and Corollary~\ref{corr:halfcone}, we get that (by if necessary negating the third row) for the third entries of the $v(\tau_i)'$ 
$$
d_{021}+l_{02}/2 = d_{011}+l_{01}/2+ 1/2 = k \imath/2 +1
$$
By the same admissible operations as in Case 2.3.1, we get - writing $l_{01}=2k_1+1$, $l_{02}=2k_2$ - that $d_{011}=-k_1$, $d_{021}=1-k_2$ holds. But as we also have seen in Case 2.3.1.1.1, this forces $\imath=4$ with $v_1$ replaced  with $v(\tau_2)'$. But here, the point
$$
(-1,-1,0,-1)=\frac{1}{2k_2}v_{02}+\frac{k_2-1}{2k_2}v(\tau_2)'
$$
again lies inside $A_X^c(\lambda_0)$, so we get no canonical singularity.

\smallskip

\noindent
\emph{Case 2.3.3: additional column $v_{12}=(-2,-2,d_{121},\imath/2-1)$ in $\lambda_1$}. This leads to $v(\tau_2)'= (0,0,d_{011}+l_{01}(1+d_{121})/2,\imath/2)$, where we can assume $d_{121}\geq 1$. We have $|v(\tau_1)'-v(\tau_2)'|=\frac{l_{01}d_{121}}{2} \geq 3/2$ if $d_{121}\geq 2$ or $l_{01}\geq 3$, which is not possible due to Lemma~\ref{le:twopoints}. So we require $d_{121}=1$ and $l_{01}=2$. But here both $v(\tau_i)'$ are integer, so this is not possible as well.
The proof is complete.
\end{proof}

Finally, we deal with the last Case $(vi)$ of Proposition~\ref{prop:zetaexcep} (iii).

\begin{proposition}
\label{prop:zeta(l0,l1,1)}
Let $X$ be a non-toric threefold singularity of complexity one.
Assume that~$X$ is of arbitrary canonical multiplicity 
$\zeta >1$, the leading platonic triple is $(l_{01},l_{11},1)$ and $X$ is at most canonical. 
Then $X$  is isomorphic to $X(P_{i})$,
where $P_{i}$ is one of the matrices
\setlength{\arraycolsep}{4pt}
$$
{\tiny
P_{22}
\!=\!
\begin{bmatrix}
-4 & 4 & 0 & 0 \\
-4 & 0 & 1 & 1 \\
1  & 0 & 0 & 1 \\
-1 & 3 & 0 & 0 
\end{bmatrix}
\;
P_{23}
\!=\!
\begin{bmatrix}
-4 & 2 & 0 & 0 \\
-4 & 0 & 1 & 1 \\
0 & 1 & 0 & 1 \\
3 & 0 & 0 & 0 
\end{bmatrix}
\;
P_{24}
\!=\!
\begin{bmatrix}
-8 & 2 & 0 & 0 \\
-8 & 0 & 1 & 1 \\
0 & 1 & 0 & 1 \\
5 & 0 & 0 & 0 
\end{bmatrix}
\;
P_{25}
\!=\!
\begin{bmatrix}
-2 & 2 & 0 & 0  \\
-2 & 0 & 1 & 1  \\
1 &  0 &  0 & 1 \\
1 & 3 & 0 & 0
\end{bmatrix}
}
$$
$$
{\tiny
P_{26}
\!=\!
\begin{bmatrix}
-2 & 2     & 0 & 0 & \dots & 0 & 0 & 0& 0
\\
-2  & 0    & 1 &1& \dots & 0 & 0 & 0 & 0
\\
\vdots   & \vdots  & \vdots & \vdots & \ddots & \vdots & \vdots & \vdots & \vdots
\\
-2  & 0    & 0 & 0 & \dots &  1 & 1 & 0 & 0
\\
1 & 1   & 0 & d_{2} & \ldots & 0 & d_{r} & d'_{1} & d'_{2}
\\
0 & 2    & 0 & 0  & \ldots & 0 & 0 & 1 & 1
\end{bmatrix}
\,
P_{27}
\!=\!
\begin{bmatrix}
-3 & 3 & 0 & 0  \\
-3 & 0 & 1 & 1  \\
0 &  1 &  0 & 1\\
2 & 0 & 0 & 0
\end{bmatrix}
\,
P_{28}
\!=\!
\begin{bmatrix}
-5 & 4 & 0 & 0 \\
-5 & 0 & 1 & 1 \\
0 & 1 & 0 & 1 \\
-1 & 2 & 0 & 0
\end{bmatrix}
}
$$
$$
{\tiny
P_{29}
\!=\!
\begin{bmatrix}
2-m_0\zeta & 2 & 0 & 0 \\
2-m_0\zeta & 0 & 1 & 1 \\
1 & 0 & 0 & 1 \\
1+\frac{m_0(2-\zeta)}{2} & 1 & 0 & 0
\end{bmatrix}, \zeta \in 4\ZZ_{\geq 1}
\qquad
P_{30}
\!=\!
\begin{bmatrix}
2-5\zeta & 2 & 0 & 0 \\
2-5\zeta & 0 & 1 & 1 \\
0 & 1 & 0 & 1 \\
5 & 0 & 0 & 0
\end{bmatrix}, \zeta \geq 3
}
$$
$$
{\tiny
P_{31}
\!=\!
\begin{bmatrix}
4-3\zeta & 4 & 0 & 0 \\
4-3\zeta & 0 & 1 & 1 \\
0 & 1 & 0 & 1 \\
\frac{7-3\zeta}{2} & 2 & 0 & 0
\end{bmatrix}, \zeta \in 2\ZZ_{\geq 2}+1
\qquad
P_{32}
\!=\!
\begin{bmatrix}
2-3\zeta & 2 & 0 & 0 & 0 & 0 \\
2-3\zeta & 0 & 1 & 1 & 0 & 0\\
2-3\zeta & 0 & 0 & 0 & 1 & 1\\
0 & 1 & 0 & 1 & 0 & 1 \\
3 & 0 & 0 & 0 & 0 & 0
\end{bmatrix}
}
$$
$$
{\tiny
P_{33}
\!=\!
\begin{bmatrix}
2-3\zeta & 2 & 0 & 0 \\
2-3\zeta & 0 & 1 & 1 \\
0 & 1 & 0 & 2 \\
3 & 0 & 0 & 0
\end{bmatrix}
\,
P_{34}
\!=\!
\begin{bmatrix}
2-3\zeta & 2 & 0 & 0 \\
2-3\zeta & 0 & 1 & 1 \\
0 & 1 & 0 & 1 \\
3 & 0 & 0 & 0
\end{bmatrix}
\,
P_{35}
\!=\!
\begin{bmatrix}
2-3\zeta & 2-3\zeta & 2 & 0 & 0 \\
2-3\zeta & 2-3\zeta & 0 & 1 & 1 \\
0 & 1 & 1 & 0 & 1 \\
3 & 3 & 0 & 0 & 0
\end{bmatrix}
}
$$
$$
{\tiny
P_{36}
\!=\!
\begin{bmatrix}
2-3\zeta & 2-\zeta & 2 & 0 & 0 \\
2-3\zeta & 2-\zeta & 0 & 1 & 1 \\
0 & 3 & 1 & 0 & 1 \\
3 & 1 & 0 & 0 & 0
\end{bmatrix}
}
$$
For 
$
P_{37}
-P_{39} 
$
it holds $k\mu+1 \equiv 0 \mod \zeta$, $d_{11} \leq 0$, $d_{12} \geq 2+k\sum\limits_{i=2}^rd_i$.
\setlength{\arraycolsep}{1.5pt}
$$
{\tiny
P_{37}
\!=\!
\begin{bmatrix}
2(\zeta-k) & 2k     & 0 & 0 & \dots & 0 & 0 
\\
2(\zeta-k)  & 0    & 1 &1& \dots & 0 & 0 
\\
\vdots   & \vdots  & \vdots & \vdots & \ddots & \vdots & \vdots 
\\
2(\zeta-k)  & 0    & 0 & 0 & \dots &  1 & 1 
\\
1 & 1   & 0 & d_{2} & \ldots & 0 & d_{r} 
\\
2\left(\frac{k\mu+1}{\zeta}-\mu\right) & 2\frac{k\mu+1}{\zeta}    & 0 & 0  & \ldots & 0 & 0 
\end{bmatrix}
\,
P_{38}
\!=\!
\begin{bmatrix}
2(\zeta-k) & 2k-\zeta & 2k     & 0 & 0 & \dots & 0 & 0 
\\
2(\zeta-k) & 2k-\zeta  & 0    & 1 &1& \dots & 0 & 0 
\\
\vdots & \vdots  & \vdots  & \vdots & \vdots & \ddots & \vdots & \vdots 
\\
2(\zeta-k) & 2k-\zeta  & 0    & 0 & 0 & \dots &  1 & 1 
\\
1 & 1 & d_{11}  & 0 & d_{2} & \ldots & 0 & d_{r} 
\\
2\left(\frac{k\mu+1}{\zeta}-\mu\right) & 2\frac{k\mu+1}{\zeta}-\mu & 2\frac{k\mu+1}{\zeta}    & 0 & 0  & \ldots & 0 & 0 
\end{bmatrix}
}
$$
$$
{\tiny
P_{39}
\!=\!
\begin{bmatrix}
2(\zeta-k) & 2k-\zeta & 2k-\zeta & 2k     & 0 & 0 & \dots & 0 & 0 
\\
2(\zeta-k) & 2k-\zeta & 2k-\zeta  & 0    & 1 &1& \dots & 0 & 0 
\\
\vdots & \vdots & \vdots  & \vdots  & \vdots & \vdots & \ddots & \vdots & \vdots 
\\
2(\zeta-k) & 2k-\zeta & 2k-\zeta & 0    & 0 & 0 & \dots &  1 & 1 
\\
1  & d_{11} & d_{12} & 1 & 0 & d_{2} & \ldots & 0 & d_{r} 
\\
2\left(\frac{k\mu+1}{\zeta}-\mu\right) & 2\frac{k\mu+1}{\zeta}-\mu  & 2\frac{k\mu+1}{\zeta}-\mu & 2\frac{k\mu+1}{\zeta}    & 0 & 0  & \ldots & 0 & 0 
\end{bmatrix}
}
$$
For $P_{40}-P_{49}$ it holds  $\zeta \equiv 5 \mod 6$.
\setlength{\arraycolsep}{4pt}
$$
{\tiny
P_{40}
\!=\!
\begin{bmatrix}
2-2\zeta & 2 & 0 & 0 & 0 & 0 \\
2-2\zeta & 0 & 1 & 1 & 0 & 0\\
2-2\zeta & 0 & 0 & 0 & 1 & 1\\
1 & 0 & 0 & 1 & 0 & 1 \\
4-\zeta & 1 & 0 & 0 & 0 & 0
\end{bmatrix}
\,
P_{41}
\!=\!
\begin{bmatrix}
2-2\zeta & 2 & 0 & 0 \\
2-2\zeta & 0 & 1 & 1 \\
1 & 0 & 0 & 2 \\
4-\zeta & 1 & 0 & 0
\end{bmatrix}
\,
P_{42}
\!=\!
\begin{bmatrix}
2-2\zeta & 2 & 0 & 0 \\
2-2\zeta & 0 & 1 & 1 \\
1 & 0 & 0 & 1 \\
4-\zeta & 1 & 0 & 0
\end{bmatrix}
}
$$
$$
{\tiny
P_{43}
\!=\!
\begin{bmatrix}
2-2\zeta & 2 & 0 & 0 \\
2-2\zeta & 0 & 1 & 1 \\
2 & 0 & 0 & 1 \\
4-\zeta & 1 & 0 & 0
\end{bmatrix}
\!
P_{44}
\!=\!
\begin{bmatrix}
2-2\zeta & 2-2\zeta & 2 & 0 & 0 \\
2-2\zeta & 2-2\zeta & 0 & 1 & 1 \\
1 & 2 & 0 & 0 & 1 \\
4-\zeta & 4-\zeta & 1 & 0 & 0
\end{bmatrix}
\!
P_{45}
\!=\!
\begin{bmatrix}
2-2\zeta & 2-2\zeta & 2 & 2 & 0 & 0 \\
2-2\zeta & 2-2\zeta & 0 & 0 & 1 & 1 \\
1 & 2 & 0 & 1 & 0 & 1 \\
4-\zeta & 4-\zeta & 1 & 1 & 0 & 0
\end{bmatrix}
}
$$
$$
{\tiny
P_{46}
\!=\!
\begin{bmatrix}
2-2\zeta & 2-2\zeta & 2 & 0 & 0 \\
2-2\zeta & 2-2\zeta & 0 & 1 & 1 \\
1 & 3 & 0 & 0 & 1 \\
4-\zeta & 4-\zeta & 1 & 0 & 0
\end{bmatrix}
\,
P_{47}
\!=\!
\begin{bmatrix}
2-2\zeta & 2 & 2 & 0 & 0 \\
2-2\zeta & 0 & 0 & 1 & 1 \\
1 & 0 & 1 & 0 & 1 \\
4-\zeta & 1 &1 & 0 & 0
\end{bmatrix}
\,
P_{48}
\!=\!
\begin{bmatrix}
2-2\zeta & 2 & 2 & 0 & 0 \\
2-2\zeta & 0 & 0 & 1 & 1 \\
1 & 0 & 2 & 0 & 1 \\
4-\zeta & 1 &1 & 0 & 0
\end{bmatrix}
}
$$
$$
{\tiny
P_{49}
\!=\!
\begin{bmatrix}
2-2\zeta & 2-\zeta & 2 & 0 & 0 \\
2-2\zeta & 2- \zeta& 0  & 1 & 1 \\
1 & 2 & 0  & 0 & 1 \\
4-\zeta & \frac{5-\zeta}{2} & 1  & 0 & 0
\end{bmatrix}
\,
P_{50}
\!=\!
\begin{bmatrix}
3-2\zeta & 3 & 0 & 0 \\
3-2\zeta & 0 & 1 & 1 \\
0 & 1 & 0 & 1 \\
2 & 0 & 0 & 0
\end{bmatrix}
\,
P_{51}
\!=\!
\begin{bmatrix}
3-2\zeta & 3 & 3 & 0 & 0 \\
3-2\zeta & 0 & 0 & 1 & 1 \\
0 & 1 & 2 & 0 & 1 \\
2 & 0 & 0 & 0 & 0
\end{bmatrix}
}
$$
For $P_{52}-P_{54}$ it holds  $\zeta \equiv 6 \mod 9$.
$$
{\tiny
P_{52}
\!=\!
\begin{bmatrix}
3-2\zeta & 3 & 0 & 0 \\
3-2\zeta & 0 & 1 & 1 \\
1 & 0 & 0 & 1 \\
\frac{9-2\zeta}{3} & 1 & 0 & 0
\end{bmatrix}
\quad
P_{53}
\!=\!
\begin{bmatrix}
3-2\zeta & 3 & 3 & 0 & 0 \\
3-2\zeta & 0 & 0 & 1 & 1 \\
1 & 0 & 1 & 0 & 1 \\
\frac{9-2\zeta}{3} & 1 & 1 & 0 & 0
\end{bmatrix}
\quad
P_{54}
\!=\!
\begin{bmatrix}
3-2\zeta & 3-2\zeta & 3 & 0 & 0 \\
3-2\zeta & 3-2\zeta & 0 & 1 & 1 \\
1 & 2 & 0 & 0 & 1 \\
\frac{9-2\zeta}{3} & \frac{9-2\zeta}{3} & 1 & 0 & 0
\end{bmatrix}
}
$$
For $P_{55}-P_{57}$ it holds  $\zeta \equiv 0 \mod 9$.
$$
{\tiny
P_{55}
\!=\!
\begin{bmatrix}
3-2\zeta & 3 & 0 & 0 \\
3-2\zeta & 0 & 1 & 1 \\
1 & 0 & 0 & 1 \\
\frac{12-4\zeta}{3} & 2 & 0 & 0
\end{bmatrix}
\quad
P_{56}
\!=\!
\begin{bmatrix}
3-2\zeta & 3 & 3 & 0 & 0 \\
3-2\zeta & 0 & 0 & 1 & 1 \\
1 & 0 & 1 & 0 & 1 \\
\frac{12-4\zeta}{3} & 2 & 2 & 0 & 0
\end{bmatrix}
\quad
P_{57}
\!=\!
\begin{bmatrix}
3-2\zeta & 3-2\zeta & 3 & 0 & 0 \\
3-2\zeta & 3-2\zeta & 0 & 1 & 1 \\
1 & 2 & 0 & 0 & 1 \\
\frac{12-4\zeta}{3} & \frac{9-2\zeta}{3} & 1 & 0 & 0
\end{bmatrix}
}
$$
$$
{\tiny
P_{58}
\!=\!
\begin{bmatrix}
4-2\zeta & 4 & 0 & 0 \\
4-2\zeta & 0 & 1 & 1 \\
1 & 0 & 0 & 1 \\
\frac{5-\zeta}{2} & 1 & 0 & 0
\end{bmatrix},
\zeta \equiv 7 \mod 12
\qquad
P_{59}
\!=\!
\begin{bmatrix}
4-2\zeta & 4 & 0 & 0 \\
4-2\zeta & 0 & 1 & 1 \\
1 & 0 & 0 & 1 \\
\frac{9-3\zeta}{2} & 3 & 0 & 0
\end{bmatrix},
\zeta \equiv 1 \mod 12
}
$$
For $P_{60}$, $P_{61}$ it holds $\zeta \in 2\ZZ_{\geq 2} +1$.
$$
{\tiny
P_{60}
\!=\!
\begin{bmatrix}
1-m_0\zeta & 1 & 1 & 0 & 0 \\
1-m_0\zeta & 0 & 0 & 1 & 1 \\
1 & 0 & 1 & 0 & 1 \\
2m_0 & 0 & 0 & 0 & 0
\end{bmatrix}
, m_0\geq 2
\qquad
P_{61}
\!=\!
\begin{bmatrix}
2-\zeta & 2 & 0 & 0 \\
2-\zeta & 0 & 1 & 1 \\
-1 & -1 & 0 & 1 \\
2 & 0 & 0 & 0
\end{bmatrix}
}
$$
For $P_{62}-P_{65}$ it holds $\zeta=k\imath+\lll$, $\gcd(k,\lll)=\gcd(\mathfrak{d},\imath)=1$ for all $\mathfrak{d} \in \ZZ\cap[d,d+d_0]$.
\setlength{\arraycolsep}{1.5pt}
$$
{\tiny
P_{62}
\!=\!
\begin{bmatrix}
-k\imath & \lll & 0 & 0 & \cdots & 0 & 0 \\
-k\imath & 0 & 1 & 1 & \cdots & 0 & 0 \\
\vdots & \vdots & \vdots & \vdots & \ddots & \vdots & \vdots \\
-k\imath & 0 & 0 & 0 & \cdots & 1 & 1 \\
d & d & 0 & d_{2} & \cdots & 0 & d_{r} \\
\imath\frac{1-\mu k}{\lll+k \imath} & \frac{\imath+\mu \lll}{\lll+k \imath} & 0 & 0 & \cdots & 0 & 0
\end{bmatrix}
\quad
P_{63}
\!=\!
\begin{bmatrix}
-k\imath & \lll & \lll & 0 & 0 & \cdots & 0 & 0 \\
-k\imath & 0 & 0 & 1 & 1 & \cdots & 0 & 0 \\
\vdots & \vdots & \vdots & \vdots & \vdots & \ddots & \vdots & \vdots \\
-k\imath & 0 & 0 & 0 & 0 & \cdots & 1 & 1 \\
d & 0 & d+d_1 &  0 & d_{2} & \cdots & 0 & d_{r} \\
\imath\frac{1-\mu k}{\lll+k \imath} & \frac{\imath+\mu \lll}{\lll+k \imath} & \frac{\imath+\mu \lll}{\lll+k \imath} & 0 & 0 & \cdots & 0 & 0
\end{bmatrix}
}
$$
$$
{\tiny
P_{64}
\!=\!
\begin{bmatrix}
-k\imath & -k\imath & \lll & 0 & 0 & \cdots & 0 & 0 \\
-k\imath & -k\imath & 0 & 1 & 1 & \cdots & 0 & 0 \\
\vdots & \vdots & \vdots & \vdots & \vdots & \ddots & \vdots & \vdots \\
-k\imath & -k\imath & 0 & 0 & 0 & \cdots & 1 & 1 \\
d & d+d_0 & 0 & 0 & d_{2} & \cdots & 0 & d_{r} \\
\imath\frac{1-\mu k}{\lll+k \imath} & \imath\frac{1-\mu k}{\lll+k \imath} & \frac{\imath+\mu \lll}{\lll+k \imath} & 0 & 0 & \cdots & 0 & 0
\end{bmatrix}
\quad
P_{65}
\!=\!
\begin{bmatrix}
-k\imath & -k\imath & \lll & \lll & 0 & 0 & \cdots & 0 & 0 \\
-k\imath & -k\imath & 0 & 0 & 1 & 1 & \cdots & 0 & 0 \\
\vdots & \vdots & \vdots & \vdots & \vdots & \vdots & \ddots & \vdots & \vdots \\
-k\imath & -k\imath & 0 & 0 & 0 & 0 & \cdots & 1 & 1 \\
d & d+d_0 & 0 & d+d_1 &  0 & d_{2} & \cdots & 0 & d_{r} \\
\imath\frac{1-\mu k}{\lll+k \imath} & \imath\frac{1-\mu k}{\lll+k \imath} & \frac{\imath+\mu \lll}{\lll+k \imath} & \frac{\imath+\mu \lll}{\lll+k \imath} & 0 & 0 & \cdots & 0 & 0
\end{bmatrix}
}
$$
\end{proposition}

\begin{proof}
According to Proposition~\ref{prop:zetaexcep} (iii), there exists an integer $\mu$ with $\gcd(\mu,\zeta,\imath)=1$, so that the matrix $P$ has leading block 
\setlength{\arraycolsep}{4pt}
$$
{\tiny
\begin{bmatrix}
-l_{01} & l_{11}  & 0 & 0 & \dots & 0 
\\
-l_{01} & 0 & 1 & 0 & \dots & 0 
\\
-l_{01} & 0 & 0 & 1 & \dots & 0 
\\
\vdots & \vdots & \vdots & \vdots & \ddots & \vdots 
\\
-l_{01} & 0 & 0 &  0 &  \dots &  1 
\\
d_{011} & d_{111} & 0 & 0 & \ldots & 0 
\\
\frac{\imath-\mu l_{01}}{\zeta} & \frac{\imath+\mu l_{11}}{\zeta} & 0 & 0 & \ldots &0 
\end{bmatrix}
}
$$
with $l_{01}\geq l_{11} \geq 1$.
Moreover we require $\gcd(\mu,\zeta)=1$, otherwise since $(\imath-\mu l_{01})/\zeta$ is integer, $\gcd(\imath,\mu,\zeta)\neq 1$ would follow. Moreover, subtracting the last entry of $v_{01}$ from the last entry of $v_{11}$, we see that $\mu(l_{01}+l_{11})/\zeta$ is integer, leading to $l_{01}+l_{11}=m\zeta$ for some $m \in \ZZ_{\geq 1}$.
We distinguish three cases in the following.

\bigskip

\noindent
\emph{Case 1: $\zeta | \imath$}. Here we require $\gcd(\mu,\zeta)=1$, so $\zeta|l_{01},l_{11}$. Set $\lll_i=l_{i1}/\zeta$ for $i=0,1$. Due to irredundancy we have at least another column $v_{22}=e_{2}+d_{221}e_{r+1}$, where we can assume $d_{221}\geq 1$ by if this is not the case adding the  $-d_{221}$ of the second to the penultimate row. So the intersections of $\partial A_X^c$ with the elementary cones $\tau_1$ set up by the leading block and $\tau_2$ set up by the leading block with column $v_{21}$ replaced by $v_{22}$ are the leaving points
\begin{align*}
v(\tau_1)' = & \left( 0,\ldots,0, \frac{\lll_1d_{011}+\lll_0d_{111}}{\lll_0+\lll_1},\imath/\zeta\right), \\
v(\tau_2)' = &  \left(0,\ldots,0, \frac{\lll_1d_{011}+\lll_0d_{111} + d_{221}\lll_0\lll_1\zeta}{\lll_0+\lll_1},\imath/\zeta\right).
\end{align*}
We distinguish three subcases in the following.

\medskip

\noindent
\emph{Case 1.1: $\zeta = 2$}. We achieve $\mu=1$ by admissible operations and distinguish some more subcases.

\smallskip

\noindent
\emph{Case 1.1.1: $\lll_1\geq 2$}.
Then $|v(\tau_1)'-v(\tau_2)'|\geq 2$, forcing $\imath=2$. Then $A_X^c(\lambda_0)\cap\{x_{r+2}=0\}$ is the convex hull of $0_{\ZZ_{r+2}}$ and the two points
\begin{align*}
w_1 =& \left(-2,-2,\frac{\lll_1d_{011}+\lll_0d_{111}+d_{011}-d_{111}}{\lll_0+\lll_1},0\right),
\\
w_2 =& \left(-2,-2,\frac{\lll_1d_{011}+\lll_0d_{111}+d_{011}-d_{111}+2d_{221}\lll_1(\lll_0-1)}{\lll_0+\lll_1},0\right).
\end{align*}
Thus if $\lll_1\geq 3$ or $\lll_1=2$ and $\lll_0\geq 4$, we have $|w_1-w_2|\geq 2$ and due to Lemma~\ref{le:twopoints}, $A_X^c(\lambda_0)\cap\{x_{r+2}=0\}$ can not be canonical. So we have $\lll_1=2$ and $\lll_0 \in \{2,3\}$ and can assume $d_{111}=0$. 

\smallskip

\noindent
\emph{Case 1.1.1.1: $\lll_0=3$}. In this case,
\begin{align*}
w_1 =& \left(-2,-2,\frac{3}{5}d_{011},0\right),
w_2 = \left(-2,-2,\frac{3}{5}d_{011}+\frac{8}{5}d_{221},0\right),
\end{align*}
which forces $d_{221}=1$, $r=2$ and $d_{011} \equiv 7 \mod 10$ due to Corollary~\ref{corr:halfhalfcone}. By admissible operations, we achieve $d_{011}=-3$. But then, $A_X^c(\lambda_1)$ contains the point $(1,0,0,1)$.

\smallskip

\noindent
\emph{Case 1.1.1.2: $\lll_0=2$}. In this case,
\begin{align*}
w_1 =& \left(-2,-2,\frac{3}{4}d_{011},0\right),
w_2 = \left(-2,-2,\frac{3}{4}d_{011}+d_{221},0\right),
\end{align*}
which forces $d_{221}=1$, $r=2$ and $d_{011} \equiv 1,3,6 \mod 8$ due to Corollary~\ref{corr:halfhalfcone}. By admissible operations, we get $d_{011} \in \{1,6\}$. For $d_{011} =6$, we have the point $(1,0,1,1)$ in $A_X^c(\lambda_1)$, while for  $d_{011} =1$, the resulting matrix
$$
P_{22}:=
{\tiny
\begin{bmatrix}
-4 & 4 & 0 & 0 \\
-4 & 0 & 1 & 1 \\
1  & 0 & 0 & 1 \\
-1 & 3 & 0 & 0 
\end{bmatrix}
}
$$
gives a canonical singularity. It is easy to check that no column can be added to this matrix.

\smallskip

\noindent
\emph{Case 1.1.2: $\lll_1 = 1, \lll_0 \geq 2$}.
We have $|v(\tau_1)'-v(\tau_2)'| \geq 3/2 $ if $\lll_0 \geq 2$ and for $\lll_0 = 2$, the polytope $A_X^c(\lambda)$ either contains two points or a cone like in Corollary~\ref{corr:halfhalfcone} and thus, due to this corollary and Lemma~\ref{le:twopoints}, we get $\imath \in \{2,4\}$.

\smallskip

\noindent
\emph{Case 1.1.2.1: $\imath=4$}. Here we require $d_{211}=1$ and $r=2$, since otherwise $|v(\tau_1)'-v(\tau_2)'| \geq 3$. Moreover, we achieve $d_{111}=0$ by admissible operations. Recall the leaving points
\begin{align*}
v(\tau_1)' = & \left( 0,\ldots,0, \frac{d_{011}}{\lll_0+1},2\right), 
v(\tau_2)' =   \left(0,\ldots,0, \frac{d_{011} + 2\lll_0}{\lll_0+1},2\right).
\end{align*}
We require $d_{011} \equiv 1 \mod 2(\lll_0+1)$ due to Corollary~\ref{corr:halfhalfcone} and Lemma~\ref{le:twopoints} and by admissible operations, we achieve $d_{011} = 1$. But then inside $A_X^c(\lambda_0)$ lies the point
$$
(-1,-1,1,1)=\frac{1}{2\lll_0}v_{01}+\frac{2\lll_0-3}{8\lll_0}v(\tau_1)'+\frac{4\lll_0-1}{8\lll_0}v(\tau_2)'.
$$

\smallskip

\noindent
\emph{Case 1.1.2.2: $\imath=2$}. Here $d_{111}$ must be odd and by adding appropriate multiples of the first to the penultimate row, we achieve $d_{111}=1$. As in Case 1.1.1, we investigate	 $A_X^c(\lambda_0)\cap\{x_{r+2}=0\}$, which here is the convex hull of $0_{\ZZ_{r+2}}$ and the two points
\begin{align*}
w_1 =& \left(-2,-2,\frac{2d_{011}+\lll_0-1}{\lll_0+1},0\right),
w_2 = \left(-2,-2,\frac{2d_{011}+\lll_0-1 + 2d_{211}(\lll_0-1)}{\lll_0+1},0\right).
\end{align*}
Lemma~\ref{le:twopoints} forces $d_{211}=1$ and $r=2$ again and together with Corollary~\ref{corr:halfhalfcone}, it leads to $2d_{011} \equiv \lll_0+4, \lll_0+5, \lll_0+6 \mod 2\lll_0+2$. Since by admissible operations we can add multiples of $\lll_0+1$ to $d_{011}$, we achieve $2d_{011} \in \{ \lll_0+4, \lll_0+5, \lll_0+6 \}$.

\smallskip

\noindent
\emph{Case 1.1.2.2.1: $\lll_0$ odd}. Here we get $d_{011}=(\lll_0+5)/2$. Since $v_{01}$ must be primitive, we require $\lll_0+5 \equiv 2 \mod 4 $ in addition. So $\lll_0 \geq 5$. Then $A_X^c(\lambda_0)$ contains
\begin{align*}
w_1 =& \left(-6,-6,\frac{3\lll_0+7}{\lll_0+1},-2\right),
w_2 = \left(-6,-6,\frac{5\lll_0+1}{\lll_0+1},-2\right)
\end{align*}
and thus the point $(-3,-3,2,-1)$. It can not be canonical.

\smallskip

\noindent
\emph{Case 1.1.2.2.2: $\lll_0=2$}. By admissible operations, the two possible cases $d_{011}=3,4$ are equivalent to $d_{011}=0$ and give a canonical singularity with defining matrix $P_{23}$.

\smallskip

\noindent
\emph{Case 1.1.2.2.3: $\lll_0 \geq 4$ even}. By admissible operations, the two possibilities for $d_{011}$ are equivalent, so we let $d_{011}=(\lll_0+4)/2$. Here $A_X^c(\lambda_0)$ contains the points
\begin{align*}
w_1 =& \left(-6,-6,\frac{3\lll_0+5}{\lll_0+1},-2\right),
w_2 = \left(-6,-6,\frac{5\lll_0-1}{\lll_0+1},-2\right)
\end{align*}
and thus the point $(-3,-3,2,-1)$ if $\lll_0 > 4$. If $\lll_0=4$, we get a canonical singularity with matrix $P_{24}$.

\smallskip

\noindent
\emph{Case 1.1.3: $\lll_1 = \lll_0 =1$}. We distinguish two subcases.

\smallskip

\noindent
\emph{Case 1.1.3.1: $\imath\geq 4$}. Lemma~\ref{le:twopoints} forces $d_{211}=1$ and $r=2$. Together with Corollary~\ref{corr:halfhalfcone}, it leads to $\imath=4$, $d_{111}=0$ and $d_{011} \equiv 1 \mod 4$. By admissible operations, we achieve  $d_{011}=1$. The resulting matrix
$$
P_{25}:=
{\tiny
\begin{bmatrix}
-2 & 2 & 0 & 0  \\
-2 & 0 & 1 & 1  \\
1 &  0 &  0 & 1\\
1 & 3 & 0 & 0
\end{bmatrix}
}
$$
gives a canonical singularity. It is easy to check that no columns can be added.

\smallskip

\noindent
\emph{Case 1.1.3.2: $\imath=2$}. We achieve $d_{011}=d_{111}=1$ here. It is clear that $r$ and the $d_{i11}$ can be arbitrary and canonicity is preserved. In $\lambda_0$ and $\lambda_1$, no columns can be added, but up to two in the lineality part. We denote this series by
$$
P_{26}:=
{\tiny
\begin{bmatrix}
-2 & 2     & 0 & 0 & \dots & 0 & 0 & 0& 0
\\
-2  & 0    & 1 &1& \dots & 0 & 0 & 0 & 0
\\
\vdots   & \vdots  & \vdots & \vdots & \ddots & \vdots & \vdots & \vdots & \vdots
\\
-2  & 0    & 0 & 0 & \dots &  1 & 1 & 0 & 0
\\
1 & 1   & 0 & d_{2} & \ldots & 0 & d_{r} & d'_{1} & d'_{2}
\\
0 & 2    & 0 & 0  & \ldots & 0 & 0 & 1 & 1
\end{bmatrix}.
}
$$

\medskip

\noindent
\emph{Case 1.2: $\zeta = 3$}. Here if $\lll_0\geq 2$, then $|v(\tau_1)'-v(\tau_2)'|=\left|3d_{221}\lll_0\lll_1/(\lll_0+\lll_1)\right|\geq 2$. Moreover, if $\lll_0=\lll_1=1$, then $A_X^c(\lambda)$ contains two integer points as well. Thus $\imath = 3$ due to Lemma~\ref{le:twopoints} and Corollary~\ref{corr:halfhalfcone}. By adding appropriate multiples of the first to the last row, we achieve $\mu \in \{-1,1\}$.

\smallskip

\noindent
\emph{Case 1.2.1: $\lll_0=\lll_1=1$}. Here by interchanging the data of the first two leaves, we achieve $\mu=-1$. We achieve $d_{111}=1$ and $d_{011} \in \{0,1\}$ by admissible operations. The polytope $A_X^c(\lambda_0)$ contains the points
\begin{align*}
w_1 =& \left(-1,-1,\frac{3d_{011}+1}{6},1\right),
w_2 = \left(-1,-1,\frac{3d_{011}+1+3d_{211}}{6},1\right).
\end{align*}
Because otherwise there would be an integer point inbetween $w_1$ and $w_2$, which violates canonicity, we require $d_{211}=1$, $r=2$ and $d_{011}=0$. This gives the next canonical singularity with matrix $P_{27}$. It is easy to check that no additional columns are possible.

\smallskip

\noindent
\emph{Case 1.2.2: $\lll_0\geq 2$, $\lll_1=1$}. We distinguish two subcases.

\smallskip

\noindent
\emph{Case 1.2.2.1: $\mu=-1$}. We achieve $d_{111}=1$ as above. The polytope $A_X^c(\lambda_0)$ contains the points
\begin{align*}
w_1 =& \left(-1,-1,\frac{3d_{011}-1+2\lll_0}{3(\lll_0+1)},1\right),
w_2 = \left(-1,-1,\frac{3d_{011}-1+2\lll_0+3(2d_{211}\lll_0-1)}{3(\lll_0+1)},1\right)
\end{align*}
and thus, there is an integer point inbetween, violating canonicity.

\smallskip

\noindent
\emph{Case 1.2.2.2: $\mu=1$}. We achieve $d_{111}=0$ by admissible operations. The polytope $A_X^c(\lambda_0)$ contains the points
\begin{align*}
w_1 =& \left(-3,-3,\frac{2d_{011}}{\lll_0+1},0\right),
w_2 = \left(-3,-3,\frac{2d_{011} +3d_{211}(\lll_0-1)}{\lll_0+1},0\right),
\end{align*}
so for $\lll_0\geq 3$, due to Corollary~\ref{corr:halfhalfcone} and Lemma~\ref{le:twopoints}, the corresponding singularity cannot be canonical. For $\lll_0=2$, we get $d_{211}=1$ and $r=2$ as above, moreover $d_{011} \equiv 5,7 \mod 9$. Thus by admissible operations, we achieve $d_{011}=-2$. But then the point $(1,0,0,1)$ is contained in $A_X^c(\lambda_1)$. 

\smallskip

\noindent
\emph{Case 1.2.3: $\lll_0,\lll_1\geq 2$}. By if necessary exchanging the data of the first two leaves, we achieve $\mu=1$. Bear in mind that by doing this, we cannot assume $\lll_0\geq \lll_1$ anymore, i.e. the matrix possibly is not in standard form. As in Case 1.2.2.2, there are two points $w_1, w_2$ in $A_X^c(\lambda_0) \cap \{x_1=x_2=-1\}$, but now their distance is $3d_{211}\lll_1(\lll_0-1)/(\lll_0+\lll_1)\geq 3/2$, so that the corresponding singularity cannot be canonical.

\medskip

\noindent
\emph{Case 1.3: $\zeta \geq 4$}. Here  $|v(\tau_1)'-v(\tau_2)'|\geq 2$. Thus $\imath = \zeta$ due to Lemma~\ref{le:twopoints}. If now $A_X^c(\lambda)$ contains two integer points $(0,0,k,1)$ and $(0,0,k+2,1)$, then Theorem~\ref{th:toric3dim} forces $\lll_0=\lll_1=1$. But for $\lll_0,\lll_1\geq 2$, we have $|v(\tau_1)'-v(\tau_2)'|\geq 3$, so $A_X^c(\lambda)$ contains two such points. If $\lll_0=2$, $\lll_1=1$, then $|v(\tau_1)'-v(\tau_2)'| = 8/3$ and the penultimate entry of $v(\tau_1)'$ is a multiple of $1/3$. Therefore also in this case, $A_X^c(\lambda)$ contains two such points. The case $\lll_0=\lll_1=1$ is left. Now if $\zeta\geq 5$, then $A_X^c(\lambda)$ contains two such points as well. Applying Theorem~\ref{th:toric3dim} to $A_X^c(\lambda_0) \cup A_X^c(\lambda_1)$, we get $\zeta=\imath=2$, a contradiction. Thus only $\zeta=4$ is left. We require $d_{221}=1$, $r=2$, since otherwise again two such integer points would be inside $A_X^c(\lambda)$. Moreover, we have $\mu \equiv 1,3 \mod 4$. By exchanging the data of the first two leaves if necessary and adding appropriate multiples of the first row to the last, we achieve $\mu=1$, resulting in $d_{011}=d_{111}=1$. But then, the point $(1,0,1,1)$ lies inside $A_X^c(\lambda_1)$.

\bigskip

\noindent
\emph{Case 2: $\zeta \nmid \imath$, $\zeta < \imath$}. We can write $\imath=k\zeta+i$ for integer $k \geq 1$ and $ 1\leq i \leq \zeta -1$. Moreover, $\zeta \nmid l_{01},l_{11}$ due to integrality of the last entries of $v_{01},v_{11}$ and $\zeta \nmid \imath$.
Thus as in Case 1, the polytope $A_X^c(\lambda)$ contains the points 
\begin{align*}
v(\tau_1)' = & \left( 0,\ldots,0, \frac{l_{11}d_{011}+l_{01}d_{111}}{l_{11}+l_{01}},\imath/\zeta\right), \\
v(\tau_2)' = &  \left(0,\ldots,0, \frac{l_{11}d_{011}+l_{01}d_{111} + d_{221}l_{01}l_{11}}{l_{11}+l_{01}},\imath/\zeta\right).
\end{align*}
Remember that $l_{01}=m\zeta-l_{11}$ for some $m \in \ZZ_{\geq 1}$.
If $A_X^c(\lambda)$ contains an integer point with last coordinate $k$, then it is not canonical. But the line segment $A_X^c(\lambda)\cap \{x_{r+2}=k\}$ contains the points 
\begin{align*}
w_1 = & \left( 0,\ldots,0, k\frac{l_{11}d_{011}+(m\zeta-l_{11})d_{111}}{m(k\zeta+i)},k\right), \\
w_2 = &  \left(0,\ldots,0, k\frac{l_{11}d_{011}+(m\zeta-l_{11})d_{111} + d_{221}(m\zeta-l_{11})l_{11}}{m(k\zeta+i)},k\right).
\end{align*}
Since for $l_{11}\geq 4$, $|w_1-w_2|\geq 1$, we have $l_{11}\leq 3$. Moreover, we require 
$$
d_{221}k(m\zeta-l_{11})l_{11} \leq m(k\zeta+i) - 2
$$
in order to have no integer point in $A_X^c(\lambda)\cap \{x_{r+2}=k\}$. So $d_{211}=1$, $r=2$. 

\medskip

\noindent
\emph{Case 2.1: $l_{11}=3$}.
The inequality from above becomes
$$
m(2k\zeta-i)\leq 9k-2
$$
here. If $m\geq 2$, at least $k(2\zeta-9/2)\leq i-1$ must be fulfilled. Since $i\leq\zeta-1$, this forces $k=i=1$, $m=2$, $\zeta=2$ and thus $l_{01}=1$, a contradiction. If $m=1$, at least $k(2\zeta-9) \leq i-2$ must hold. For $k\geq 2$, this leads to $\zeta \leq 5$ and thus $l_{01} \leq 2$, a contradiction. For $k=1$, it leads to $\zeta =6$, $i=5$ and $l_{01}=3$. But since $\frac{11+3\mu}{6}$ cannot be integral, this is not possible.

\medskip

\noindent
\emph{Case 2.2: $l_{11}=2$}.
The inequality from above becomes
$$
m(k\zeta-i)\leq 4k-2
$$
here and we have $\zeta\geq 3$. The case $m\geq 3$ can be excluded.

\smallskip

\noindent
\emph{Case 2.2.1: $m=2$}. The fraction in the penultimate entries of $w_1$ and $w_2$ can be reduced by two, so even the inequality $k(\zeta-2)\leq i-2$ must be fulfilled. But this is impossible.

\smallskip

\noindent
\emph{Case 2.2.2: $m=1$}. We require $\zeta\geq 4 $ here. The inequality from above becomes $k(\zeta-4)\leq i-2$. This leads to $i \in \{\zeta-1,\zeta-2\}$.

\smallskip

\noindent
\emph{Case 2.2.2.1: $\zeta=4$}. Here $i=2$ is impossible, since again, the fraction in the penultimate entries of $w_1$ and $w_2$ can be reduced by two. The case $i=3$ remains, but the last entry of $v_{01}$, which is $(4k+3-2\mu)/4$, must be integer, which is impossible.

\smallskip

\noindent
\emph{Case 2.2.2.2: $\zeta\geq 5$}. Here $k\geq 3$ is impossible and $k=2$ leads to $\zeta=5$.

\smallskip

\noindent
\emph{Case 2.2.2.2.1: $k=2$}. We require $i=4$ here. But the fraction in the penultimate entries of $w_1$ and $w_2$ can be reduced by two, which makes this case impossible.

\smallskip

\noindent
\emph{Case 2.2.2.2.2: $k=1$}. If $i= \zeta-1$, then $\zeta$ must be odd and we achieve $\mu = (1-\zeta)/2$ and $d_{111}=0$ as well as $d_{011}=\zeta$. But then the point
$$
(-1,-1,3,2)=\frac{1}{\zeta-2}v_{01}+\frac{(\zeta-3)\zeta}{(\zeta-2)(2\zeta-1)}v(\tau_1)'+\frac{(\zeta-3)\zeta}{4\zeta^2-10\zeta+4}v(\tau_1)'
$$
lies in $A_X^c(\lambda_0)$. If $i= \zeta-2$, then $\zeta$ must be odd and we achieve $\mu=1-\zeta$ and $d_{111}=1$ as well as $d_{011}=(3-\zeta)/2$.
But then in $A_X^c(\lambda_0)$ lies the point
$$
(-1,-1,0,2)=\frac{1}{\zeta-2}v_{01}+\frac{(\zeta-3)\zeta}{4(\zeta-2)(\zeta-1)}v(\tau_1)'+\frac{(\zeta-3)\zeta}{4\zeta^2-12\zeta+8}v(\tau_1)'.
$$

\medskip

\noindent
\emph{Case 2.3: $l_{11}=1$, $l_{01}\geq 2$}. We need an additional column $v_{12}=(1,0,\ldots,0,d_{121},0)$ in $\lambda_1$  here.
The polytope $A_X^c(\lambda)$ contains the points 
\begin{align*}
v(\tau_1)' = & \left( 0,\ldots,0, \frac{d_{011}+l_{01}d_{111}}{1+l_{01}},\imath/\zeta\right), \\
v(\tau_3)' = &  \left(0,\ldots,0, \frac{d_{011}+l_{01}d_{111} + (d_{121}+d_{221})l_{01}}{1+l_{01}},\imath/\zeta\right).
\end{align*}
Thus the line segment $A_X^c(\lambda)\cap \{x_{r+2}=k\}$ contains the points 
\begin{align*}
w_1 = & \left( 0,\ldots,0, k\frac{d_{011}+(m\zeta-1)d_{111}}{m(k\zeta+i)},k\right), \\
w_3 = &  \left(0,\ldots,0, k\frac{d_{011}+(m\zeta-1)d_{111} + (d_{121}+d_{221})(m\zeta-1)}{m(k\zeta+i)},k\right).
\end{align*}
So $k(d_{121}+d_{221})(m\zeta-1)\leq m(k\zeta+i)-2$ must hold. But this is impossible.

\medskip

\noindent
\emph{Case 2.4: $l_{01}=l_{11}=1$}. This forces $\zeta=2$ and $i=1$. We need another column $v_{02}=(-1,\ldots,-1,d_{021},0)$ in $\lambda_0$  here.
The polytope $A_X^c(\lambda)$ contains the points 
\begin{align*}
v(\tau_1)' = & \left( 0,\ldots,0, \frac{d_{011}+d_{111}}{2},\imath/2\right), \\
v(\tau_4)' = &  \left(0,\ldots,0, \frac{d_{011}+d_{111} + (d_{021}+d_{121}+d_{221})l_{01}}{2},\imath/2\right).
\end{align*}
Thus the line segment $A_X^c(\lambda)\cap \{x_{r+2}=k\}$ contains the points 
\begin{align*}
w_1 = & \left( 0,\ldots,0, k\frac{d_{011}+d_{111}}{2k+1},k\right), \\
w_3 = &  \left(0,\ldots,0, k\frac{d_{011}+d_{111} + (d_{021}+d_{121}+d_{221})}{2k+1},k\right).
\end{align*}
But $|w_1-w_3|\geq 1$, so $A_X^c(\lambda)\cap \{x_{r+2}=k\}$ contains an integer point.

\bigskip

\noindent
\emph{Case 3: $\zeta > \imath \geq 2$}. Recall that $l_{11}+l_{01}=m\zeta$. We now write $l_{11}=m_1\zeta+\lll$ and $l_{01}=m_0\zeta-\lll$ with $m_0+m_1=m$. Now the leading block is of the form
$$
{\tiny
\begin{bmatrix}
-(m_0\zeta-\lll) & m_1\zeta+\lll  & 0 & 0 & \dots & 0 
\\
-(m_0\zeta-\lll) & 0 & 1 & 0 & \dots & 0 
\\
-(m_0\zeta-\lll) & 0 & 0 & 1 & \dots & 0 
\\
\vdots & \vdots & \vdots & \vdots & \ddots & \vdots 
\\
-(m_0\zeta-\lll) & 0 & 0 &  0 &  \dots &  1 
\\
d_{011} & d_{111} & 0 & 0 & \ldots & 0 
\\
\frac{\imath+\mu \lll}{\zeta} -\mu m_0 & \frac{\imath+\mu \lll}{\zeta} +\mu m_1 & 0 & 0 & \ldots &0 
\end{bmatrix}.
}
$$ 
Since $A_X^c(\lambda)$ provides no restrictions in this case any more, our approach is different than in the previous cases. Observe that all $v_{0i}$ and $v_{1i}$ lie on one affine hyperplane and moreover  $v(\tau_1)'$ lies on the line segment with end points $v_{01}$ and $v_{11}$. Thus $A_X^c(\lambda_0) \cup A_X^c(\lambda_1)$ contains a polytope of the form $\Qq:=\conv ( \nu_0, \nu_1, \nu_\tau )$ with
\begin{align*}
\nu_0 & := \left(\lll-m_0\zeta,d_{011},\frac{\imath+\mu \lll}{\zeta} -\mu m_0\right), \\
\nu_1 & := \left(m_1\zeta+\lll,d_{111},\frac{\imath+\mu \lll}{\zeta} +\mu m_1\right), \\
\nu_\tau & :=  \left(0,\frac{(m_0\zeta-\lll)d_{111}+(m_1\zeta+\lll)d_{011}+\left(\sum_{i=2}^{r}d_{i21}\right)(m_0\zeta-\lll)(m_1\zeta+\lll)}{(m_0+m_1)\zeta},\frac{\imath}{\zeta}\right).
\end{align*}
Now we search for integer points inside this polytope. If we find a line segment inside that contains three  or more integer points and neither $\nu_0$ nor $\nu_1$, then Theorem~\ref{th:toric3dim} inter alia forces $\imath=2$. If we find two parallel line segments inside that both contain two integer points, Theorem~\ref{th:toric3dim} forces $\imath=2$ as well. If one of those parallel line segments even contains three integer points, the corresponding singularity cannot be canonical. We distinguish some subcases.

\medskip

\noindent
\emph{Case 3.1: $m_1\geq 1$.} We have $m_0 \geq 2$ here, due to $l_{01}\geq l_{11}$. We have a look at the line segment $\Qq_1=\Qq \cap \{x_1 =\lll\}$, it has the endpoints
\begin{align*}
w_1 & := \left(\lll, \frac{d_{111}m_0+d_{011}m_1}{m_0+m_1},\frac{\imath+\mu \lll}{\zeta}\right) \\
w_2 & := w_1 + \frac{\left(\sum_{i=2}^{r}d_{i21}\right)m_1(m_0\zeta-\lll)}{m_0+m_1} e_2.
\end{align*}
So it contains at least $k$ integer points if 
$$
\left(\sum_{i=2}^{r}d_{i21}\right)m_1(m_0\zeta-\lll) -k(m_0+m_1)+1 \geq 0.
$$
The line segment $\Qq_0=\Qq \cap \{x_1 =\lll-\zeta\}$ on the other hand contains at least $k$ integer points if 
$$
\left(\sum_{i=2}^{r}d_{i21}\right)m_0(m_1\zeta+\lll) -k(m_0+m_1)+1 \geq 0.
$$

\smallskip

\noindent
\emph{Case 3.1.1: $m_1\geq 2$.} We have $m_0 \geq 3$ here. Both line segments contain three or more integer points.

\smallskip

\noindent
\emph{Case 3.1.2: $m_1=1$.} We have $m_0 \geq 2$ here.

\smallskip

\noindent
\emph{Case 3.1.2.1: $\zeta =3$.} This forces $\imath=r=2$, $d_{221}=1$. 

\smallskip

\noindent
\emph{Case 3.1.2.1.1: $\lll=1$.} We achieve $\mu=d_{111}=1$. Now both line segments contain two or more points if $m_0\geq 3$. If $m_0 \geq  5$, then $\Qq_0$ contains three or more points. We have a look at the remaining cases.

\smallskip

\noindent
\emph{Case 3.1.2.1.1.1: $m_0=4$.} Here $\Qq_1$ contains three or more integer points if not $d_{011} \equiv 2,3,4 \mod 5$ and $\Qq_0$ if not $d_{011} \equiv 2,4 \mod 5$. We achieve $d_{011} \in \{ 2,4 \}$ by admissible operations. If $d_{011} = 2$, then $(-1,-1,1,0)$ is an inner point of $A_X^c(\lambda_0)$. If $d_{011}=4$, then $(-1,-1,2,0)$ is an inner point of $A_X^c(\lambda_0)$.

\smallskip

\noindent
\emph{Case 3.1.2.1.1.2: $m_0=3$.}  Here $\Qq_1$ or $\Qq_0$ contains three or more integer points if not $d_{011} \equiv 0,2 \mod 4$. But this is not possible due to primitivity of $v_{01}$.

\smallskip

\noindent
\emph{Case 3.1.2.1.1.3: $m_0=2$.} We achieve $d_{011} \in \{ 0,1,2\}$ here. But for $d_{011}=1,2$, the point $(-1,-1,1,0)$ lies inside $A_X^c(\lambda_0)$, while for $d_{011}=0$, we get a canonical singularity from the matrix
$$
P_{28}:=
{\tiny
\begin{bmatrix}
-5 & 4 & 0 & 0 \\
-5 & 0 & 1 & 1 \\
0 & 1 & 0 & 1 \\
-1 & 2 & 0 & 0
\end{bmatrix}.
}
$$
It is easy to check that no column can be added.

\smallskip

\noindent
\emph{Case 3.1.2.1.2: $\lll=2$.} We achieve $\mu=-1$, $d_{111}=0$ and can assume $m_0 \geq 3$ here. But $\Qq_1$ contains at least two and $\Qq_0$ at least three integer points, so we get no canonical singularity.

\smallskip

\noindent
\emph{Case 3.1.2.2: $\zeta = 4$.} Both $\Qq_1$ and $\Qq_0$ contain at least two integer points. This forces $\imath=2$. Since $\gcd(\zeta,\imath,\mu)=1$, we have $\mu \in \{-1,1\}$ and consequently $\lll=2$. Then $Q_0$ contains three or more points if not $m_0=2$. We achieve $\mu=-1$ and $d_{011}=1$. But then $A_X^c(\lambda_1)$ contains the point $(1,0,1,0)$.

\smallskip

\noindent
\emph{Case 3.1.2.3: $\zeta \geq 5$.} At least one of $\Qq_1$ and $\Qq_0$ contains three or more integer points while the other contains two or more. So we get no canonical singularity from this case.

\medskip

\noindent
\emph{Case 3.2: $m_1=0, m_0 \geq 2$.} 

\smallskip

\noindent
\emph{Case 3.2.1: $\lll=1$.} This leads to $\mu=-\imath$ and we require an additional column $v_{12}=(1,0,\ldots,0,d_{121},0)$, while we can assume $d_{111}=0$. Thus $A_X^c(\lambda_0) \cup A_X^c(\lambda_1)$ now contains a polytope of the form $\Qq^*:=\conv ( \nu_0, \nu_1, \nu_{1,2}, \nu_{\tau,2} )$ with
\begin{align*}
\nu_{1,2} & := \left(1,d_{121},0\right), \\
\nu_{\tau,2} & :=  \left(0,\frac{d_{011}+\left(\sum_{i=1}^{r}d_{i21}\right)(m_0\zeta-1)}{m_0\zeta},\frac{\imath}{\zeta}\right).
\end{align*}
Now the line segment $\Qq_0^*=\Qq^* \cap \{x_1=1-\zeta\}$ has the endpoints
\begin{align*}
w_1^* & := \left(1-\zeta, \frac{d_{011}}{m_0},\imath\right) \\
w_2^* & := w_1^* + \frac{\left(\sum_{i=1}^{r}d_{i21}\right) (m_0-1)}{m_0} e_2.
\end{align*}
If this line segment contains two integer points, then with the two integer points in $\Qq^* \cap \{x_1=1\}$ and applying Corollary~\ref{corr:quadrangle}, we get a contradiction. This forces $r=2$ and $d_{121}=d_{221}=1$ as well as $d_{011} \equiv 1 \mod m_0$. Our matrix is of the form
$$
{\tiny
\begin{bmatrix}
1-m_0\zeta & 1 & 1 & 0 & 0 \\
1-m_0\zeta & 0 & 0 & 1 & 1 \\
d_{011} & 0 & 1 & 0 & 1 \\
m_0\imath & 0 & 0 & 0 & 0
\end{bmatrix}.
}
$$
Now $\Qq^*$ lies inside a polytope $\Qq^+=\conv(\nu_0,\nu_1,(1,2,0))$, while $\Qq^+ \setminus \Qq^*$ can not contain any integer point. Since $\Qq^+$ is of type $(ii)$ from Theorem~\ref{th:toric3dim}, we get $\imath=2$, odd $\zeta\geq 3$ and $d_{011}=1$. This gives the series $P_{60}$ of canonical singularities.

\smallskip

\noindent
\emph{Case 3.2.2: $\lll \geq 2$.} 
 Now the line segment $\Qq_0=\Qq \cap \{x_1=\lll-\zeta\}$ has the endpoints
\begin{align*}
w_1 & := \left(\lll-\zeta, \frac{d_{111}(m_0-1)+d_{011}}{m_0},\frac{\imath+\mu( \lll-\zeta)}{\zeta}\right) \\
w_2 & := w_1 + \frac{\left(\sum_{i=2}^{r}d_{i21}\right)\lll (m_0-1)}{m_0} e_2.
\end{align*}
We have a second line segment $\Qq_0'=\Qq \cap \{x_1=\lll-2\zeta\}$ with the endpoints
\begin{align*}
w_1' & := \left(\lll-2\zeta, \frac{d_{111}(m_0-2)+2d_{011}}{m_0},\frac{\imath+\mu( \lll-2\zeta)}{\zeta}\right), \\
w_2' & := w_1 + \frac{\left(\sum_{i=2}^{r}d_{i21}\right)\lll (m_0-2)}{m_0} e_2.
\end{align*}
If both of these line segments have two or more integer points, the singularity cannot be canonical.

\smallskip

\noindent
\emph{Case 3.2.2.1: $m_0 \geq 4$.} Here, both line segments contain two integer points if $\lll \geq 3$. This forces $\lll=r=2$, $d_{221}=1$ and we achieve $d_{111}=1$.
Now we have $A_X^c(\lambda_0)=\Qq^-\cap \{x_1\leq 0\}$, where $\Qq^-$ is the polytope
$$
\left(\!
\left(2-m_0\zeta,2-m_0\zeta,d_{011},\frac{\imath+2\mu}{\zeta}-\mu m_0\right),
\left(2,2,1,\frac{\imath+2\mu }{\zeta}\right),
\left(2,2,3,\frac{\imath+2\mu }{\zeta}\right)\!
\right).
$$
We have two possibilities for $\mu$ now.

\smallskip

\noindent
\emph{Case 3.2.2.1.1: $\mu=(\zeta-\imath)/2$.} We achieve $d_{111}=0$. Since $\Qq^-\cap \{x_1 > 0\}$
can not contain any integer point, $A_X^c(\lambda_0)$ is canonical if and only if $\Qq^-$ is. Then Theorem~\ref{th:toric3dim} forces  $\imath=2$ and $d_{011}=1$. Since $\gcd(\mu,\zeta,\imath)=1$ must hold, we require $\zeta \equiv 0 \mod 4$. We get a canonical singularity from the corresponding matrix
$$
P_{29}:=
{\tiny
\begin{bmatrix}
2-m_0\zeta & 2 & 0 & 0 \\
2-m_0\zeta & 0 & 1 & 1 \\
1 & 0 & 0 & 1 \\
1+\frac{m_0(2-\zeta)}{2} & 1 & 0 & 0
\end{bmatrix}
}.
$$

\smallskip

\noindent
\emph{Case 3.2.2.1.2: $\mu=-\imath/2$.} Here $\imath$ must be even. We achieve $d_{111}=1$. First assume $d_{011} =km_0+2$ with $k \in \ZZ$. But then the point $(1-\zeta,1-\zeta,k+1,\imath/2)$ lies in $A_X^c(\lambda_0)$. If $d_{011}=km_0$, then $\Qq_0$ and $\Qq_0'$ contain two integer points. This is impossible. Now let $d_{011}=km_0+j$ with $j \neq 0,2$. Only for $2d_{011} \equiv 3,4,5 \mod m_0$, the line segment $\Qq_0'$ contains one integer point. But since this point is in the relative interior of $\Qq_0'$ and since $\Qq_0$ has two integer points in its relative interior, Theorem~\ref{th:toric3dim} forces $\imath=2$.

\smallskip

\noindent
\emph{Case 3.2.2.1.2.1: $m_0$ even.} Here we achieve $d_{011}=(4-m_0)/2$. But $A_X^c(\lambda_0)$ contains the point $(1-2\zeta,1-2\zeta,0,2)$.

\smallskip

\noindent
\emph{Case 3.2.2.1.2.2: $m_0$ odd.} Here we achieve $d_{011}=(3-m_0)/2$. If $m_0 \geq 7$, then the point $(1-2\zeta,1-2\zeta,0,2)$ again lies in $A_X^c(\lambda_0)$. For $m_0=5$, we get a canonical singularity with defining matrix $P_{30}$.

\smallskip

\noindent
\emph{Case 3.2.2.2: $m_0 =3$.} Here for $\lll\geq 5$, both $\Qq_0$ and $\Qq_0'$ contain two or more integer points. 

\smallskip

\noindent
\emph{Case 3.2.2.2.1: $\lll=4$.} Here we require $r=2$, $d_{221}=1$ and since $\Qq_0$ contains three or more integer points, also $\imath=2$.  

\smallskip

\noindent
\emph{Case 3.2.2.2.1.1: $\zeta$ odd.} We get $\mu=(\zeta-1)/2$ here and achieve $d_{111}=1$, $d_{011}=0$. The corresponding singularity is canonical with matrix $P_{31}$.

\smallskip

\noindent
\emph{Case 3.2.2.2.1.2: $\zeta$ even.} We get $\mu \in \{(\zeta-2)/4,(3\zeta-2)/4\}$. If $\mu = (\zeta-2)/4$, we achieve $d_{111}=0$, $\zeta=8k+2$, since otherwise $\Qq_0'$ would contain a 2-fold point. But then $\Qq_0$ contains a $2$-fold point. If $\mu = (3\zeta-2)/4$, we achieve $d_{111}=0$, $\zeta=8k+6$, since otherwise $\Qq_0'$ would contain a $2$-fold point. But then $\gcd(\mu,\imath,\zeta)=2\neq 1$, a contradiction.

\smallskip

\noindent
\emph{Case 3.2.2.2.2: $\lll=3$.} Here we require $r=2$, $d_{221}=1$ and since $\Qq_0$ contains two points that do not lie on the line between $\nu_0$ and $\nu_1$ and $\Qq_0'$ contains at least one integer point, $\imath=2$ is forced. Moreover $\zeta$ must be odd, since otherwise due to integrality of $(3\mu+2)/\zeta$, we would have even $\mu$ as well, violating $\gcd(\mu,\zeta,\imath)=1$. We have the possibilities $\mu \in \{(\zeta-2)/3,2(\zeta-1)/3\}$. But in both cases $\Qq_0$ contains a 2-fold point.

\smallskip

\noindent
\emph{Case 3.2.2.2.3: $\lll=2$.} Here $\sum_{i=2}^{r}d_{i21}\geq 3$ is impossible.

\smallskip

\noindent
\emph{Case 3.2.2.2.3.1: $\sum_{i=2}^{r}d_{i21}=2$.} Here $\Qq_0$ contains three or more integer points and $\Qq_0'$ at least one, which forces $\imath=2$. We have the possibilities $\mu \in \{-1,(\zeta-2)/2\}$. If $\mu=-1$, we achieve $d_{111}=1$ and $d_{011}=0$. Since $\sum_{i=2}^{r}d_{i21}=2$ is possible for either $r=2$ and $d_{221}=2$ or $r=3$ and $d_{321}=d_{221}=1$, we get two canonical singularities with defining matrices $P_{32}$ and $P_{33}$ from this. It is easy to check that no column can be added. If $\mu=(\zeta-2)/2$, then due to $\gcd(\zeta,\mu,\imath)=1$, we have $\zeta \in 4\ZZ$. But then $\Qq_0$ contains a $2$-fold point.

\smallskip

\noindent
\emph{Case 3.2.2.2.3.2: $\sum_{i=2}^{r}d_{i21}=1$.} This requires $r=2$, $d_{221}=1$. We have the possibilities $\mu \in \{-\imath/2, (\zeta-\imath)/2\}$.

\smallskip

\noindent
\emph{Case 3.2.2.2.3.2.1: $\mu=-\imath/2$.} We need $\imath$ even and get $d_{111}=1$ here.

\smallskip

\noindent
\emph{Case 3.2.2.2.3.2.1.1: $d_{011} \equiv 0 \mod 3$.} This forces $\imath=2$ since $\Qq_0$ contains two integer points one of which lies inside the polytope spanned by the other and $\nu_0$ and $\nu_1$. We achieve $d_{011}=0$. This gives a canonical singularity. We see that two columns can be added: $(2-3\zeta,2-3\zeta,1,3)$ and $(2-\zeta,2-\zeta, 3,0)$. Both together can not be added since then $(1-\zeta,1-\zeta,1,1)$ would lie inside $A_X^c(\lambda_0)$. We get the defining matrices $P_{34}$, $P_{35}$, $P_{36}$.

\smallskip

\noindent
\emph{Case 3.2.2.2.3.2.1.2: $d_{011} \equiv 1 \mod 3$.} 
For $\imath=2$, this is equivalent to $d_{011} \equiv 0 \mod 3$. So we can assume $\imath>2$.
We write $d_{011}=3k+1$. Here $\Qq$ contains the polytope 
$$
\Qq' = \conv\left(\left(2-3\zeta,3k+1,3\imath/2\right),\left(2,1,0\right), 
\left(2-\zeta,k+2,\imath/2\right)\right)
$$
with integer vertices. This polytope  must fall under Case $(iii)$ of Theorem~\ref{th:toric3dim} if canonical. Let $\alpha,\beta \in \ZZ$ with $-\alpha \zeta  + \beta \imath/2 =1$. This is possible due to $1 = \gcd(\zeta,\imath,\mu)=\gcd(\zeta,\imath/2)$. Then multiplication  from the left with the unimodular matrix
$$
{\tiny
\begin{bmatrix}
\alpha(1+k) & -1 & \beta(1+k) \\
-\alpha k & 1 & -\beta k \\
\imath/2 & 0 & \zeta
\end{bmatrix}
}
$$
transforms $\Qq'$ to the polytope
\begin{align*}
\conv \left(\right. & \left(2\alpha(k+1)+2, -2\alpha k +1, \imath\right), \left(2\alpha(k+1)-1, -2\alpha k +2, \imath\right), \\
&\left. \left(2\alpha(k+1)-1, -2\alpha k +1, \imath\right)\right).
\end{align*}
Then due to Theorem~\ref{th:toric3dim}, we get $k \equiv -\zeta \mod \imath/2$. By admissible operations,  we achieve $k=-\zeta$. Then by subtracting the first from the third row, we achieve $d_{011}=d_{111}=-1$. We have a look at $\Qq \cap \{x_2=0\}$, this polytope contains the points
$$
\left(2-3\zeta/2,0,3\imath/4\right),
\quad
\left(2-\zeta,0,2\imath/2\right),
\quad
\left(\frac{3\zeta-4}{3\zeta-2},0,\frac{3\imath}{6\zeta-4}\right).
$$
But since the first two differ by $\zeta/2$ and the last two by $\frac{3\zeta-5}{3\zeta-2}\geq 1/3$ in the last coordinate, by Corollary~\ref{corr:halfcone} we get $\imath=2$, which we excluded.

\smallskip

\noindent
\emph{Case 3.2.2.2.3.2.1.3: $d_{011} \equiv 2 \mod 3$.} We write $d_{011}=3k-1$. Here $\Qq_0'$ contains the 2-fold point $(2-2\zeta,2-2\zeta,2k,\imath)$.

\smallskip

\noindent
\emph{Case 3.2.2.2.3.2.2: $(\zeta-\imath)/2$.} We get $d_{111}=1$ as well. Since $\gcd(\mu,\zeta,\imath)=1$, we require $\zeta-\imath \equiv 2 \mod 4$.

\smallskip

\noindent
\emph{Case 3.2.2.2.3.2.2.1: $d_{011} \equiv 2 \mod 3$.}
Since $\Qq$ contains a polytope of Type $(iii)$ from Theorem~\ref{th:toric3dim}, we get that $\Qq$ must be lattice equivalent to a polytope $\Qq_s$ with vertices
$$
(0,k,\imath),(1,k+2,\imath), \left( \frac{6\zeta-2}{3\zeta},k+\frac{8-6\zeta}{3\zeta},\imath\right),
$$ 
where we can assume $0\leq k \leq \imath-1$. Now the point $(1,k,\imath-1)$ lies inside $\conv(\Qq_s,0_{\ZZ_3})$, if $2/(k+4)\geq 1/(\imath-1)$. This is the case for any possible $k$ if $\imath\geq 5$. For $\imath=4$, it is not the case only for $k=3$, but this is equivalent to $k=-1$, and then the point $(-1,1,3)$ lies inside. For $\imath=3$ and $k=0,1$, the point $(1,0,2)$ lies inside, while for $k=2$, the point $(1,1,2)$ lies inside. Now $\imath=2$ remains, from which follows $\zeta \in 4\ZZ$. Since $d_{011}$ must be odd, we achieve $d_{011}=-1$. We get a canonical singularity, for which it is easy to check, that no columns can be added. The matrix appears as the case $m_0=3$ of $P_{29}$.

\smallskip

\noindent
\emph{Case 3.2.2.2.3.2.2.2: $d_{011} \equiv 0 \mod 3$.} As we have seen above in the respective case for $\mu-\imath/2$, this forces $\imath=2$. But then $\Qq_0$ contains a 2-fold point.

\smallskip

\noindent
\emph{Case 3.2.2.2.3.2.2.3: $d_{011} \equiv 1 \mod 3$.} If $\imath$ is even, $\Qq_0$ contains a $2$-fold point. So $\imath$ must be odd. But as we have seen in the respective case for $\mu=-\imath/2$, we require $\imath=2$, a contradiction.

\smallskip

\noindent
\emph{Case 3.2.2.3: $m_0 = 2$.}

\smallskip

\noindent
\emph{Case 3.2.2.3.1: $\sum_{i=2}^{r}d_{i21}\geq 3$.}
This forces $\imath=2$. Since $\Qq_0$ contains a 2-fold point otherwise, we require that $\lll$ is even. This determines $\mu$ if we suppose $1 \leq \mu \leq \zeta-1$. By admissible operations, we achieve $d_{011}=d_{111}=1$. We get a series of canonical singularities, where optionally one or two columns of the form $(\lll-\zeta,\lll-\zeta,d_{0i1},(\lll\mu+2)/\zeta-\mu)$ can be added. This gives the defining matrices $P_{37}$, $P_{38}$, $P_{39}$.

\smallskip

\noindent
\emph{Case 3.2.2.3.2: $\sum_{i=2}^{r}d_{i21}=2$.}
The polytope $\Qq$ in any case contains a lattice polytope and at least one integer point in the relative interior, leading to $\imath=2,3$ with Theorem~\ref{th:toric3dim}.

\smallskip

\noindent
\emph{Case 3.2.2.3.2.1: $\imath=3$.} This forces $\lll=2$. By admissible operations, we achieve $d_{111}=0$, since $(2\mu+3)/\zeta$ is odd. We achieve $\mu=(\zeta-3)/2$. Moreover $d_{011}$ must be odd, so we achieve $d_{011} \in \{-1,1\}$. Due to $\gcd(\mu,\zeta,\imath)=1$ and since $\zeta$ must be odd, we achieve $\zeta \equiv 1,-1 \mod 6$.

First let $d_{011}=-1$. Then $\zeta \equiv -1 \mod 6$ is not allowed since $\Qq_0$ would contain a 3-fold point. If $\zeta \equiv 1 \mod 6$, then $A_X^c(\lambda_0)$ contains the point $((4-\zeta)/3,(4-\zeta)/3,0,(7-\zeta)/6)$.

Now let $d_{011}=1$. If $\zeta \equiv 1 \mod 6$, then $A_X^c(\lambda_0)$ contains the point $((4-\zeta)/3,(4-\zeta)/3,1,(7-\zeta)/6)$. For $\zeta \equiv 5 \mod 6$, we get a canonical singularity. Since for $\sum_{i=2}^{r}d_{i21}=2$, we have the two possibilities $r=2, d_{221}=2$ and $r=3, d_{221}=d_{321}=1$, we get two different singularities with defining matrices $P_{40}$, $P_{41}$.

\smallskip

\noindent
\emph{Case 3.2.2.3.2.2: $\imath=2$.} As we have seen in Case 3.2.2.3.1, we require $\lll$ even. The resulting singularities can be seen as part of the series from Case 3.2.2.3.1.

\smallskip

\noindent
\emph{Case 3.2.2.3.3: $\sum_{i=2}^{r}d_{i21}=1$.}
This forces $r=2$, $d_{221}=1$. 

\smallskip

\noindent
\emph{Case 3.2.2.3.3.1: $\imath=2$.}
The resulting singularities must be part of the series from Case 3.2.2.3.1, if $\lll$ is even. So  we assume that $\lll$ is odd. But then $\Qq_0$ contains a  $2$-fold point, which is not possible.

\smallskip

\noindent
\emph{Case 3.2.2.3.3.2: $\imath=3$.} We have $\lll\leq 4$.

\smallskip

\noindent
\emph{Case 3.2.2.3.3.2.1: $\lll=2$.} As in Case 3.2.2.3.2.1, we achieve $d_{111}=0$,  $\mu=(\zeta-3)/2$ and $\zeta \equiv 5 \mod 6$. The possibilities for $d_{011}$ are $1,2$. These two can also be combined, yielding one additional column. Another column of the form $(2,0,1,1)$ can be added to this combination. Now let $d_{011}=1$ not  in combination with $d_{021}=2$. Then in addition $d_{021}=3$ is possible or a column of one of the forms $(2,0,1,1), (2,0,2,1), (2-\zeta,2-\zeta,2,(5-\zeta)/2)$. Any combination of possible columns with $d_{011}=2$ and without $d_{021}=1$ is isomorphic to one of the previous ones. We get the defining matrices $P_{42}-P_{49}$.

\smallskip

\noindent
\emph{Case 3.2.2.3.3.2.2: $\lll=3$.} We achieve $d_{111}=1$. If $3 \nmid \zeta$, then we achieve $\mu=-1$. We achieve $d_{011} \in \{0,1\}$. If $d_{011} = 1$, then the point $(2-\zeta,2-\zeta,1,1)$ lies inside $A_X^c(\lambda_0)$. If $d_{011} = 0$, the resulting singularity is canonical. We can also add the column $(3,0,2,0)$ and get the defining matrices $P_{50}$, $P_{51}$. These singularities are also possible for $3 | \zeta$, $\mu=-1$. So we assume $3 | \zeta$, $\mu \in \{\zeta/3-1,2\zeta/3-1\}$ now. We achieve $d_{111}=0$. 

First assume $\mu=\zeta/3-1$, then we require $\zeta/3 \equiv 0,2 \mod 3$. If $\zeta/3 \equiv 0 \mod 3$, then we achieve $d_{011} \in \{1,2,4,5\}$. But then the point $(2-\zeta/3,2-\zeta/3,1,1-\zeta/9)$ lies inside $A_X^c(\lambda_0)$. If $\zeta/3 \equiv 2 \mod 3$, then we achieve $d_{011} \in \{1,2\}$, both giving isomorphic canonical singularities. For $d_{011}=1$, we can also add the column $(3,0,1,1)$. We get the defining matrices $P_{52}-P_{54}$.

Now finally assume $\mu=2\zeta/3-1$, then we require $\zeta/3 \equiv 0,1 \mod 3$. If $\zeta/3 \equiv 1 \mod 3$, then we achieve $d_{011} \in \{1,2,4,5\}$. But then the point $(2-\zeta/3,2-\zeta/3,1,(15-2\zeta)/9)$ lies inside $A_X^c(\lambda_0)$. If $\zeta/3 \equiv 0 \mod 3$, we achieve $d_{011} \in \{1,2\}$, both giving isomorphic canonical singularities. For $d_{011}=1$, we can also add the column $(3,0,1,2)$. We get the defining matrices $P_{55}-P_{57}$.

\smallskip

\noindent
\emph{Case 3.2.2.3.3.2.3: $\lll=4$.} We have the possibilities $\mu \in \{(\zeta-3)/4, 3(\zeta-1)/4\}$ and can assume $d_{111}=0$. We have $\zeta \equiv 1,2 \mod 3$, otherwise $\gcd(\mu,\zeta,\imath)=3\neq 1$.

\smallskip

\noindent
\emph{Case 3.2.2.3.3.2.3.1: $\mu =(\zeta-3)/4$.} This leads to $\zeta \equiv 7,-1 \mod 12$.  For $\zeta \equiv 7 \mod 12$, we get $d_{011}=1$ and get a canonical singularity with defining matrix $P_{58}$. For $\zeta \equiv -1 \mod 12$, we achieve $d_{011} \in \{1,2,4,5\}$. But then the point $((8-\zeta)/3,(8-\zeta)/3,1,(11-\zeta)/8)$ lies inside $A_X^c(\lambda_0)$.

\smallskip

\noindent
\emph{Case 3.2.2.3.3.2.3.2: $\mu =3(\zeta-1)/4$.} This leads to $\zeta \equiv 1,5 \mod 12$. For $\zeta \equiv 5 \mod 12$, we achieve $d_{011} \in \{1,2,4,5\}$. But then for $\zeta \geq 17$, the point $((8-\zeta)/3,(8-\zeta)/3,1,(9-\zeta)/4)$ lies inside $A_X^c(\lambda_0)$. For $\zeta=5$, the point $(1,0,1,1)$ lies inside $A_X^c(\lambda_1)$. For $\zeta \equiv 1 \mod 12$, we get $d_{011}=1$ and get a canonical singularity with defining matrix $P_{59}$.

\smallskip

\noindent
\emph{Case 3.2.2.3.3.3: $\imath\geq 4$.} Here $\lll \geq 4 $ is impossible due to Theorem~\ref{th:toric3dim}.

\smallskip

\noindent
\emph{Case 3.2.2.3.3.3.1: $\lll=3$.} We require $d_{011} \equiv d_{111} \mod 2$.  Here $\Qq$ contains the polytope
$$
\Qq' = \conv(\nu_0,\nu_1,(3-\zeta,(d_{011}+d_{111})/2+1,(\imath+3\mu)/\zeta-\mu)). 
$$ 
This polytope must fall under Case $(iii)$ of Theorem~\ref{th:toric3dim} if canonical. But then $\Qq$ contains a polytope of the form
$$
\conv\left(\left(k+1/3,\imath\right),\left(k+(5-2\zeta)/(3-2\zeta),\imath\right)\right),
$$
which can not be canonical for $\imath \geq 4$ due to Corollary~\ref{corr:1313cone}.

\smallskip

\noindent
\emph{Case 3.2.2.3.3.3.2: $\lll=2$.} 

\smallskip

\noindent
\emph{Case 3.2.2.3.3.3.2.1: $\imath$ odd.} This leads to odd $\zeta$ and $\mu=(\zeta-\imath)/2$. We achieve $d_{111}=0$. First assume that $d_{011}$ is odd. But then $\Qq$ contains a polytope of the form
$$
\conv\left(\left(k+1/2,\imath\right),\left(k+(3-2\zeta)/(2-2\zeta),\imath\right)\right),
$$
which can not be canonical for $\imath \geq 4$ due to Corollary~\ref{corr:1313cone}. Now assume that $d_{011}$ is even. Then applying Corollary~\ref{corr:halfcone}, we get that $\imath|2$, a contradiction.

\smallskip

\noindent
\emph{Case 3.2.2.3.3.3.2.2: $\imath$ even.} Here $\zeta$ can be odd or even.

\smallskip

\noindent
\emph{Case 3.2.2.3.3.3.2.2.1: $\zeta$ odd.} This forces $\mu = -\imath/2$ and we achieve $d_{111}=1$ and $d_{011}$ odd. Again applying Corollary~\ref{corr:halfcone}, we get that $\imath|2$, a contradiction since $\imath\geq 4$.

\smallskip

\noindent
\emph{Case 3.2.2.3.3.3.2.2.2: $\zeta$ even.} For $\mu = -\imath/2$, we have the same conclusion as in the previous case. So $\mu=(\zeta-\imath)/2$ remains. The proceeding is exactly the same as in Case 3.2.2.3.3.3.2.1.

\medskip

\noindent
\emph{Case 3.3: $m_1=0, m_0 =1$.} 

\smallskip

\noindent
\emph{Case 3.3.1: $\lll=1$.} Here as in Case 3.2.1, the polytope $A_X^c(\lambda_0) \cup A_X^c(\lambda_1)$ now contains a polytope of the form $\Qq^*:=\conv ( \nu_0, \nu_1, \nu_{1,2}, \nu_{\tau,2} )$ with
\begin{align*}
\nu_{1,2} & := \left(1,d_{121},0\right), \\
\nu_{\tau,2} & :=  \left(0,\frac{d_{011}+\left(\sum_{i=1}^{r}d_{i21}\right)(m_0\zeta-1)}{m_0\zeta},\frac{\imath}{\zeta}\right).
\end{align*}
Moreover, $\Qq^*$ lies inside a polytope $\Qq^+=\conv(\nu_0,\nu_1,(1,\sum_{i=1}^{r}d_{i21},0))$, while $\Qq^+ \setminus \Qq^*$ can not contain any integer point. Since $\sum_{i=1}^{r}d_{i21}\geq 2$ and $\Qq^+$ is of type $(iii)$ from Theorem~\ref{th:toric3dim}, we get $\zeta =k\imath+1$ with $k \in  \ZZ_{\geq 1}$ and $\gcd(d_{011},\imath)=1$.

\smallskip

\noindent
\emph{Case 3.3.2: $\lll\geq 2$.}
We first assume $\sum_{i=2}^{r} d_{i21}=1$, i.e. $r=2$ and $d_{i21}=1$. For greater values of $\sum_{i=2}^{r} d_{i21}$, the polytope $\Qq$ contains a respective  polytope $\Qq'$ of a singularity with $\sum_{i=2}^{r} d_{i21}=1$. Thus only if for a configuration of $\imath, \zeta, \lll, \mu$ we find a canonical singularity with $\sum_{i=2}^{r} d_{i21}=1$, it is possible to find one with $\sum_{i=2}^{r} d_{i21}>1$.
Our matrix $P$ has the following form:
$$
{\tiny
\begin{bmatrix}
\lll-\zeta & \lll & 0 & 0 \\
\lll-\zeta & 0 & 1 & 1 \\
d_{011} & d_{111} & 0 & 1 \\
\frac{\imath+\mu\lll}{\zeta}-\mu & \frac{\imath+\mu\lll}{\zeta} & 0 & 0
\end{bmatrix}.
}
$$
Now by admissible operations, multiples of $\imath$ can be added to $d_{011}$ without changing $d_{111}$ and multiples of $\zeta$ and $\mu$ can be added to \emph{the difference between $d_{011}$ and $d_{111}$}. Since $\gcd(\mu,\zeta,\imath)=1$, we achieve $d_{011}=d_{111}:=d \in \ZZ$.
We bring $P^*$ in Smith normal form. Since $\gcd(\mu,\zeta)=1$, we can choose $\alpha,\beta \in \ZZ$ with $\alpha\zeta+\beta\mu=1$. Then
$$
{\tiny
S:=
\begin{bmatrix}
-1 & 1 & \lll-\zeta & 0 \\
0 & 0 & 1 & 0 \\
0 & 0 & -1 & 1 \\
\frac{\beta \imath +\lll}{\zeta} & 1-\frac{\beta \imath +\lll}{\zeta} & d+\lll\left(1-\frac{\beta \imath +\lll}{\zeta}\right) & -d
\end{bmatrix},
\quad
T:=
\begin{bmatrix}
\alpha & 0 & 0 & -\mu \\
0 & 1 & 0 & 0 \\
0 & 0 & 1 & 0 \\
\beta & 0 & 0 & \zeta
\end{bmatrix}
}
$$
are unimodular and $S\cdot P^* \cdot T = {\rm Diag}(1,1,1,\imath)$.
This means that the class group is $\ZZ/\imath\ZZ$ and the total coordinate space given by the equation $T_0^{\lll-\zeta}+T_1^{\lll}+T_2T_3$ is the index one cover, while
$$
Q=
\begin{bmatrix}
\frac{\beta \imath +\lll}{\zeta} & 1-\frac{\beta \imath +\lll}{\zeta} & d+\lll\left(1-\frac{\beta \imath +\lll}{\zeta}\right) & -d
\end{bmatrix} \in (\ZZ/\imath\ZZ)^4
$$ 
is the grading matrix. Moreover the only integer points in $\partial A_X^c$ are  the columns of the matrix $P$. This means that the corresponding singularity is canonical if and only if it is terminal. We thus can use the classification of Mori~\cite{mori} to determine the canonical ones of this type. The relevant theorem of~\cite{mori} is Theorem 12 due to the form of the equation of the index one cover. There are three possible cases.

\smallskip

\noindent
\emph{Case 3.3.2.1: Case (1) of~\cite[Thm. 12]{mori} holds.} Since we can exchange the data of the first two leaves, we can assume $\lll-\zeta=-k\imath$ for some $k \in \ZZ\geq 1$. Thus from~\cite[Thm. 12]{mori} we get that $\gcd(d,\imath)=1$ must hold and $\frac{\beta \imath +\lll}{\zeta} \equiv 1 \mod \imath$, so $\frac{1-\mu k}{\zeta}$ must be integer. So $1 =\gcd(k,\zeta)=\gcd(k,k\imath+\lll)=\gcd(k,\lll)$ must hold.
Our matrix now has the form
$$
{\tiny
\begin{bmatrix}
-k\imath & \lll & 0 & 0 \\
-k\imath & 0 & 1 & 1 \\
d & d & 0 & 1 \\
\imath\frac{1-\mu k}{\lll+k \imath} & \frac{\imath+\mu \lll}{\lll+k \imath} & 0 & 0
\end{bmatrix}.
}
$$
It is now clear that $\sum_{i=2}^{r} d_{i21}$ can be augmented in this case in any way  perhaps losing terminality but keeping canonicity. Also in $\lambda_1$ columns of the form $(-k\imath,\ldots,-k\imath,d,\frac{\imath+\mu \lll}{\lll+k \imath})$ can be added as well as  columns of the form $(\lll,0,\ldots,0,d_{021},\imath\frac{1-\mu k}{\lll+k \imath})$ in $\lambda_0$, as long as for all $\delta \in \ZZ \cap [d,d_{021}]$, we have $\gcd(\delta,\imath)=1$. We get four series $P_{62}$-$P_{65}$ of canonical singularities.

\smallskip

\noindent
\emph{Case 3.3.2.2: Case (2) of~\cite[Thm. 12]{mori} holds.} Here we have $\imath=4$, require $d=-1$ and a look at $Q$ tells us that $\lll=2$ must hold and $2|(2\mu+4)/\zeta$, so we achieve $\mu=-2$. Thus $\zeta$ must be odd. It is easy to check that no column can be added to the resulting matrix $P_{61}$.

\smallskip

\noindent
\emph{Case 3.3.2.3: Case (3) of~\cite[Thm. 12]{mori} holds.} 
Let without restriction $\lll-\zeta=-2k$ for some $k \in \ZZ_{\geq 1}$ and $(1-\mu k)/\zeta$ integer.
So $1=\gcd(k,\zeta)=\gcd(k,2k+\lll)=\gcd(k,l)$ follows. That means we are in Case (1) of~\cite[Thm. 12]{mori}.
The determination of defining matrices $P$ for canonical threefold singularities with two-torus action is complete.
\end{proof}

\subsection{Proof of Theorems and Corollaries~\ref{th:class} - \ref{cor:cdvroot}}

Propositions~\ref{prop:zeta=1}-\ref{prop:zeta(l0,l1,1)} provide   the defining matrices $P$ of the non-toric canonical threefold singularities with two-torus action, while Proposition~\ref{prop:ClXQtoric} provides  the Cox rings and class groups of the toric ones. Thus the remaining task for proving Theorems~\ref{th:class} and~\ref{th:class2} is to determine the Cox rings and class groups of those non-toric ones, that do not belong to a "many parameter series".

\begin{proof}[Proof of Theorems~\ref{th:class} and~\ref{th:class2}]
The Cox ring (without the grading) can be read off directly from the defining matrix $P$, see Construction~\ref{constr:RAP0}.
As already stated in the proof of Proposition~\ref{prop:ClXQtoric}, to determine class group and grading matrix $Q$, we need to find unimodular matrices $V$ and $W$, so that $V \cdot P^* \cdot W = S$ is in Smith normal form. The cases without parameters can be done by computer, e.g.~using~\cite{HaKe}.
 
The few parameter series can be done by computer as well - with a little more effort. Bear in mind that we have to find Smith normal forms over $\ZZ[x_1,\ldots,x_\mathfrak{n}]$ with $\mathfrak{n} \in \{1,2,3\}$ here. But this is not a principal ideal domain, so it is not guaranteed that a Smith normal form exists. In other words, it is not clear a priori that for a defining matrix $P$ with parameters, we get polynomial formulae for the class group and the grading matrix $Q$. 
But we can compute the Smith normal form for several integer values of the parameters and interpolate the entries of respective $V$, $W$, $S$ by polynomials. We then take the interpolating polynomials as entries of matrices $\mathbf{V}$, $\mathbf{W}$, $\mathbf{S}$ and check a posteriori if these matrices are well defined.

This works for all few parameter series and we are done with the computation of the class groups and grading matrices. 
 \end{proof}

\begin{remark}
The following list provides information about which matrices from Propositions~\ref{prop:zeta=1}-\ref{prop:zeta(l0,l1,1)} correspond to the singularities from Theorems~\ref{th:class} and~\ref{th:class2}. Bear in mind that the toric  singularities with Nos. 1-5 from Theorem~\ref{th:class} correspond to the cones given by the Cases  $(ii)-(vi)$ of Theorem~\ref{th:toric3dim}.

\begin{longtable}{cccccccccccccc}
 $P_{1-4}$ & $P_5$ & $P_{6-12}$ & $P_{13}$ & $P_{14}$ & $P_{15,16}$  & $P_{17-19}$ & $P_{20}$ & $P_{21}$  \\
 \hline
 6-9 & 26 & 10-16 & 56a & 56b & 27,28 & 17-19 & 57a & 29 
 \\
 \\
$P_{22-25}$ & $P_{27,28}$ & $P_{29-31}$ & $P_{32,33}$ & $P_{34}$ & $P_{35,36}$ & $P_{37-39}$ & $P_{40,41}$ & $P_{42}$ 
\\
\hline
 20-23 & 24,25 & 30-32 & 34,35 & 33 & 36,37 & 58a-c & 39,40 & 38 
\\
 \\
  $P_{43}$ & $P_{44-51}$ & $P_{52-54}$ & $P_{55-57}$ & $P_{58,59}$ & $P_{60,61}$ & $P_{62-65}$
 \\
 \hline
  41 & 42-49 & 50-$52_{\zeta \in 9\ZZ-3}$ & 50-$52_{\zeta \in 9\ZZ}$ & 53 & 54,55 & 59a-d
\end{longtable}
\end{remark}

\begin{proof}[Proof of Corollary~\ref{cor:terminal}]
To check which of the canonical singularities are terminal, we have a look at $\partial A_X^c$ in each case. If there is an integer point different from the columns of the defining matrix $P$, the singularity $X(P)$ is not terminal.
\end{proof}

\begin{proof}[Proof of Theorem~\ref{th:CRI}]
For the canonical threefold singularities, the Cox rings can be read off directly from the defining matrix $P$. For the occuring compound du Val singularities, the complete Cox ring iteration is given by~\cite[Th. 1.8]{ltpticr}. Cox rings of the singularities $Y_i$ can be computed using~\cite[Rem. 6.7]{ltpticr}. Factoriality is characterized by~\cite[Cor. 5.8 (i)]{ltpticr}. This gives us all necessary information to draw the Cox ring iteration tree.
\end{proof}

\begin{proof}[Proof of Corollary~\ref{cor:cdvroot}]
We only have to show the generalized compound du Val property for the singularities $Y_1,Y_3,Y_5,Y_7$. We consider the hyperplane section given by $T_1=T_3$ and are done. The second assertion follows from~\cite[Th. 1.6]{ltpticr}.
\end{proof}

\end{document}